\documentclass[11pt, reqno]{amsart}
\usepackage[T2A]{fontenc}
\usepackage[english]{babel}
\usepackage{amsmath,amssymb,graphicx,enumerate,bbm, 
xcolor
}
\usepackage{svg}

\usepackage[T2A,T1]{fontenc}

\makeatletter
\def\easycyrsymbol#1{\mathord{\mathchoice
		{\mbox{\fontsize\tf@size\z@\usefont{T2A}{\rmdefault}{m}{it}#1}}
		{\mbox{\fontsize\tf@size\z@\usefont{T2A}{\rmdefault}{m}{it}#1}}
		{\mbox{\fontsize\sf@size\z@\usefont{T2A}{\rmdefault}{m}{it}#1}}
		{\mbox{\fontsize\ssf@size\z@\usefont{T2A}{\rmdefault}{m}{it}#1}}
}}

\makeatletter
\def\easycyrsymbolup#1{\mathord{\mathchoice
		{\mbox{\fontsize\tf@size\z@\usefont{T2A}{\rmdefault}{m}{n}#1}}
		{\mbox{\fontsize\tf@size\z@\usefont{T2A}{\rmdefault}{m}{n}#1}}
		{\mbox{\fontsize\sf@size\z@\usefont{T2A}{\rmdefault}{m}{n}#1}}
		{\mbox{\fontsize\ssf@size\z@\usefont{T2A}{\rmdefault}{m}{n}#1}}
}}
\makeatother

\newcommand{\ShaUp}{\easycyrsymbolup{\CYRSH}}

\newcommand{\be}{\easycyrsymbol{\cyrb}}

\usepackage{amsfonts}
\usepackage{amstext}
\usepackage{amsthm}
\usepackage{MnSymbol}

\usepackage[
hypertexnames=false, colorlinks, citecolor=black, linkcolor=blue, urlcolor=black]{hyperref}

	\normalbaselines						  %

\usepackage{mathrsfs}
\newtheorem{thm}{Theorem}[section]
\newtheorem{lm}[thm]{Lemma}

\newtheorem{proc}[thm]{Procedure}

\theoremstyle{definition}
\newtheorem{df}[thm]{Definition}
\newtheorem*{df*}{Definition}

\theoremstyle{remark}
\newtheorem{rem}[thm]{Remark}
\newtheorem*{rem*}{Remark}

\numberwithin{equation}{section}

\newcommand{\ci}[1]{{\vphantom{\rule[-0.65ex]{0ex}{0.35ex}}}_{#1}}

\newcommand{\tp}[2]{\texorpdfstring{#1}{#2}}
\newcommand{\tup}[1]{\textup{#1}}

\newcommand{\ti}[1]{_{\scriptstyle \text{\rm #1}}}
\newcommand{\ut}[1]{^{\scriptstyle \text{\rm #1}}}

\newcommand{\dd}{{\mathrm{d}}}

\newcommand{\sh}{\mathbb{S}}

\newcommand{\shst}{\sh_{0}}

\newcommand{\Sinv}{\cF}
\newcommand{\Start}{\operatorname{Start}}
\newcommand{\Startq}{\operatorname{Start}'}

\newcommand{\wh}{\widehat}

\newcommand{\cA}{\mathcal{A}}

\newcommand{\cD}{\mathscr{D}}
\newcommand{\cC}{\mathcal{C}}

\newcommand{\cP}{\mathcal{P}}
\newcommand{\cQ}{\mathcal{Q}}
\newcommand{\cR}{\mathcal{R}}
\newcommand{\cS}{\mathscr{S}}

\newcommand{\cE}{\mathcal{E}}
\newcommand{\cF}{\mathcal{F}}

\newcommand{\cH}{\mathcal{H}}
\newcommand{\cT}{\mathcal{T}}

\newcommand{\T}{\mathbb{T}}

\renewcommand{\f}{\varphi}
\newcommand{\vf}{\varphi}
\newcommand{\e}{\varepsilon}

\renewcommand{\C}{\mathbb{C}}
\newcommand{\R}{\mathbb{R}}
\newcommand{\Z}{\mathbb{Z}}
\newcommand{\N}{\mathbb{N}}

\newcommand{\F}{\mathbb{F}}

\newcommand{\bI}{\mathbf{I}}

\newcommand{\HT}{\cH}

\newcommand{\per}{\cP}
\newcommand{\qper}{\cQ\cP}

\newcommand{\REM}{\varrho}

\newcommand{\ch}{\operatorname{ch}}
\newcommand{\ran}{\operatorname{Ran}}

\newcommand{\1}{\mathbf{1}}

\makeatletter
\newcommand{\doublewidetilde}[1]{{%
		\mathpalette\double@widetilde{#1}%
}}
\newcommand{\double@widetilde}[2]{%
	\sbox\z@{$\m@th#1\widetilde{#2}$}%
	\ht\z@=.9\ht\z@
	\widetilde{\box\z@}%
}
\makeatother

\newcommand{\wt}{\widetilde}

\newcommand{\La}{\langle }
\newcommand{\Ra}{\rangle }

\newcommand{\ba}{\mathbf{a}}
\newcommand{\bb}{\mathbf{b}}

\newcommand{\bw}{\mathbf{w}}
\newcommand{\bff}{\mathbf{f}}
\newcommand{\bv}{\mathbf{v}}
\renewcommand{\be}{\mathbf{e}}

\newcommand{\bx}{\mathbf{x}}

\newcommand{\bA}{\mathbf{A}}
\newcommand{\bE}{\mathbf{E}}
\newcommand{\bW}{\mathbf{W}}

\newcommand{\bg}{\mathbf{g}}

\newcommand{\bcQ}{\boldsymbol{\cQ}}

\newcommand{\la}{\langle}
\newcommand{\ra}{\rangle}

\newcommand{\fdot}{\,\cdot\,}

\newcommand{\Hdy}{\cH\ut{dy}}

\newenvironment{entry}
{\begin{list}{X}%
  {%
      \setlength{\labelwidth}{55pt}%
      \setlength{\leftmargin}{\labelwidth}
      \addtolength{\leftmargin}{\labelsep}%
   }%
}%
{\end{list}}      

\renewcommand{\labelenumi}{(\roman{enumi})}

\newcounter{vremennyj}

\newcommand\cond[1]{\setcounter{vremennyj}{\theenumi}\setcounter{enumi}{#1}\labelenumi\setcounter{enumi}{\thevremennyj}}

\begin{document}
\title[$3/2>1$]%
{
The matrix $A_2$ conjecture fails, i.e.~$3/2>1$.
}
\author[K.~Domelevo]{Komla Domelevo}
\address{Department of Mathematics, University of W\"{u}rzburg, W\"{u}rzburg, Germany}
\email{komla.domelevo@uni-wuerzburg.de \textrm{(K. Domelevo)}}
\author[S.~Petermichl]{Stefanie Petermichl}
\address{Department of Mathematics, University of W\"{u}rzburg, W\"{u}rzburg, Germany}
\email{stefanie.petermichl@uni-wuerzburg.de \textrm{(S. Petermichl)}}
\thanks{SP is partially supported by the Alexander von Humboldt foundation}
\author[S.~Treil]{Sergei Treil}
\address{Department of Mathematics, Brown University, USA}
\email{treil@math.brown.edu \textrm{(S. Treil)}}
\thanks{ST is partially supported by the NSF grant DMS-2154321}
\author[A.~Volberg]{Alexander Volberg}
\thanks{AV  
is partially supported  by the Oberwolfach Institute for Mathematics, Germany; AV is also supported 
by the NSF grants DMS-1900286, DMS-2154402}
\address{Department of Mathematics, Michigan Sate University, East Lansing, MI. 48823}
\email{volberg@math.msu.edu \textrm{(A. Volberg)}}
\makeatletter
\@namedef{subjclassname@2010}{
  \textup{2010} Mathematics Subject Classification}
\makeatother
\subjclass[2010]{42B20, 42B35, 47A30}
%
%
\keywords{matrix weights,
   martingale transform, Hilbert transform}
\begin{abstract}
We show that the famous matrix $\bA_2$ conjecture is false: the norm of the Hilbert Transform in 
the  
space $L^2(W)$ with matrix weight $W$ is estimated below by $C[W]\ci{\bA_2}^{3/2}$. 
\end{abstract}
\maketitle

\setcounter{tocdepth}{1}
\tableofcontents
\setcounter{tocdepth}{3}

\section*{Notation}

\begin{entry}

\item[$|I|$] denotes the $1$-dimensional Lebesgue measure of $I\subset \R$; 

\item[$\cD$] a dyadic lattice. 

\item[$\cD(I)$] for $I\in\cD$, $\cD(I)= \{J\in\cD: J\subset I\}$;  

\item[$\ch^k(I)$] dyadic descendants of order $k$ of the interval $I$;

\item[$I^0$] the unit interval $I^0=[0,1)$

\item[$\cD_k$] Abbreviation for $\cD_k=\ch^k(I^0)$ 

\item[$\cS$] $\cS=\bigcup_{n>0}\cS_n$ is the collection of stopping intervals;  

\item[$\cS_n$] the $n$th generation of stopping intervals; 

\item[$\la f\ra\ci I$] average of the function $f$ over $I$, $\la f \ra \ci I:= |I|^{-1} \int_I 
f(x) \dd x$;  

\item[$\bE\ci I$, $\bE_k$] averaging operators, $\bE\ci I f:= \La f \Ra\ci I \1\ci I$, $\bE_k:= 
\sum_{I\in\cD_k} \bE\ci I$; 

\item[$\Delta\ci I$, $\Delta_k$] martingale differences, $\Delta\ci I := -\bE\ci I + 
\sum_{J\in\ch(I)} \bE\ci{\!J}$, $\Delta_k = \bE_{k+1} - \bE_k = \sum_{I\in\cD_k} \Delta\ci I$;

\item[$\Delta\ci I^r$, $\Delta_k^r$] martingale differences of order $r$, $\Delta\ci I^r := -\bE\ci 
I + 
\sum_{J\in\ch^r(I)} \bE\ci{\!J}$, $\Delta_k^r = \bE_{k+r} - \bE_k = \sum_{I\in\cD_k} \Delta\ci 
I^r$; 
for $r=1$ we skip the index;

\item[$\wh h\ci I$] $L^\infty$ normalized Haar function, $\wh h\ci I = \1\ci{I_+} 
-\1\ci{I_-}$; 

\item[$h\ci I$] $L^2$ normalized Haar function, $h\ci I = |I|^{-1/2}\, \wh h\ci{I}$;  

\item[$\psi\ci{I,J}$] the unique orientation preserving affine bijection from the interval $J$ to 
the interval $I$;

\item[$x\underset{\delta_0<}{\asymp}y$] there exist $0<c <C<\infty$ such that 
$cx\le y \le Cx$   
for all 
sufficiently small $\delta_0$; 

\item[$x\underset{\delta_0<}{\lesssim}y$] there exists $C<\infty$ such that $ y \le Cx$   for all 
sufficiently small $\delta_0$;

\item[$x=O\ci{\delta_0<}(y)$] An alternative notation for the $x\underset{\delta_0<}{\lesssim}y$;

\item[$x\underset{\delta_0<}{\sim}y$]  for any $\e>0$ we have $(1+\e)^{-1} x\le y \le (1+\e)x$ for 
all sufficiently small $\delta_0$;  \\
in the above $4$ items the same constants work for all values of the variables involved and for all 
sufficiently small $\delta_0$. 

\end{entry}

In this paper we will use linear algebra notation, so for a (column) vector $\ba\in\F^d$ ($\F$ 
being either 
$\R$ or $\C$) we use $\ba^*$ for 
the functional $\bx\mapsto (\bx, \ba)\ci{\F^d}$, i.e.~for the Hermitian transpose of $\ba$, which 
is a row vector. Thus $\ba\ba^*$ denotes the rank one operator 
$\bx\mapsto (\bx, \ba)\ci{\F^d}\ba$.

\section{Definitions and main result}
\label{s:intro}

Recall that a ($d$-dimensional) matrix weight on $\R$  
is a locally integrable function on $\R$ 
with values in the set of positive-semidefinite $d\times d$ matrices%
\footnote{There are of course similar definitions on the unit circle $\mathbb{T}$ or 
$\mathbb{R}^N$.}. 
The weighted space $L^2(W)$ 
is defined as the space of all measurable functions $f:\R \to \F^d$, (here $\F=\R$, or $\F=\C$) 
for which
\[
\| f\|\ci{L^2(W)}^2:=\int (W(x)f(x), f(x))\ci{\F^d} \dd x <\infty\,;
\]
here $(\cdot, \cdot)\ci{\F^d}$ means the standard inner product in $\F^d$. 

A matrix weight $W$ is said to satisfy the matrix $\mathbf A_2$ condition (write $W\in\mathbf 
A_2$) if 
\begin{align}\label{e:A2}
[W]\ci{\bf A_2} := \sup_I \left\| \la W\ra\ci I^{1/2} \la W^{-1}\ra\ci I^{1/2} \right\|^2  < 
\infty\, ,
\end{align}
where $I$ runs over all intervals.
The quantity $[W]\ci{\mathbf A_2}$ is called the \emph{$\mathbf A_2$ characteristic of the weight} 
$W$.  In the scalar case, when $W$ is a scalar weight $w$, this coincides with the classical $A_2$ 
characteristic $[w]\ci{A_2}$.

Let $\cH$ denote the Hilbert transform, 
\begin{align*}
\HT f(s) =\frac1\pi \text{p.v.}\int_\R \frac{f(t)}{s-t} \dd t .
\end{align*}
In this paper, we disprove the famous \emph{matrix $\bA_2$ conjecture}, which  stated that for any 
$\bA_2$ weight $W$, $[W]\ci{\bA_2}\le \bcQ$
\begin{align*}
\| \cH g \|\ci{L^2(W)}\le C \bcQ \|g\|\ci{L^2(W)}  \qquad \forall g\in L^2(W). 
\end{align*}

More precisely, our main result is:

\begin{thm}\label{MainTheorem1}
There exists a constant $c>0$ such that for all sufficiently large 
$\bcQ$ there exist a $2\times 2$ matrix weight $W=W\ci{\!\!\bcQ}$ \tup{(}with real entries\tup{)}, 
$[W]\ci{\bA_2}\le \bcQ$ and 
a function $g\in L^2(W)$, $g:\R\to \R^2$, $g\ne 0$ such that
\begin{align*}
\| \cH g \|\ci{L^2(W)} \ge c   \bcQ^{3/2} \|g\|\ci{L^2(W)}. 
\end{align*}
\end{thm}

In fact, by picking a sufficiently small $c$ we can  state it for all $\bcQ\ge 1$. By a simple 
reduction, we can state it for all dimensions $d\ge 2 $ of matrices.

Theorem \ref{MainTheorem1} also shows that the upper bound 
\begin{align*}
\| T \|\ci{L^2(W)\to L^2(W)} \le C\ci{T} \bcQ^{3/2}
\end{align*}
obtained in \cite{NaPeTrVo} for general Calder\'on--Zygmund operators is sharp.

\section{Introduction}
In the scalar case, the $A_2$ conjecture, i.e.~the estimate $\|T\|\ci{L^2(w) \to L^2(w)} \le 
C\ci{T} 
[w]\ci{A_2} $ turned out to be true, so it is now known as the $A_2$ Theorem. It was a long 
standing open 
problem with a fascinating history that we describe briefly below.

\subsection{Motivation}

An important motivation for   sharp estimates of the Hilbert transform in $L^2(W)$ with 
matrix weight $W$ comes from probability theory, more precisely from the theory of stationary 
Gaussian processes. For  Gaussian processes all information is encoded 
in the means and correlations, so  the study of multivariate 
 stationary Gaussian processes 
(say with discrete time)  is reduced to the 
study of the subspaces $z^n \C^d$, $n\in\Z$, in the Hilbert space $L^2(\bW)$, where $\bW$ 
is the spectral measure of the process. For a multivariate process of dimension $d$ the spectral 
measure is a $d\times d$ matrix-valued measure. 

 The regularity properties of stationary stochastic 
processes in terms of their spectral measures $\bW$ is a classical area that attracted the 
attention of many mathematicians, and a huge bibliography can be found in \cite{Ib},
\cite{MW}, \cite{WM1}, \cite{WM2}. In the case of scalar processes many different types of 
regularities 
were studied and very detailed results were found in many papers, to name just a few \cite{HS}, 
\cite{HSa}, \cite{HMW}, \cite{Pell}, \cite{PeKh}. 

One of the questions about regularity was the question when the angle between past and future of 
the process is positive, which reduces to the question when the Riesz Projection $P_+$, or, 
equivalently the Hilbert Transform $\HT$ is bounded in the weighted space $L^2(\bW)$. 
This question goes back to N.~Wiener and P.~Masani, and was  the 
main motivation behind the famous Helson--Szeg\"{o} Theorem, which solved the problem in the 
one-dimensional case.

It was well known for a long time that even weaker regularity conditions than the positivity of the 
angle imply that the measure $\bW$ is absolutely continuous $\dd \bW = W \dd x$, so everything is 
reduced to estimates with matrix weights. 

Another motivation comes from the theory of Toeplitz operators (Riemann Hilbert Problem). A 
Toeplitz operator $T\ci{F}$ in the vector-valued Hardy space $H^2(\C^d)$ is defined as 
\begin{align*}
T\ci{F}f = P_+(Ff), \qquad f\in H^2(\C^d), 
\end{align*}
where $P_+$ is the Riesz Projection, i.e.~the orthogonal projection from $L^2$ onto the Hardy space 
$H^2$. 
The $d\times d$ matrix-valued function $F\in L^\infty(M_{d\times d})$ is called the symbol of the 
Toeplitz operator. It is well known that the operator $T\ci{F}$ is invertible if and only if it can
be factorized as
\begin{equation}
\label{e:factor}
F= G_1^* G_2, \qquad G_{1,2}^{\pm 1}\in H^2(M_{d\times d})   
\end{equation}
(the functions $G_{1,2}$ and their inverses are in the matrix-valued Hardy class $H^2$),  and the 
formal inverse 
\begin{align}\label{e:FormInv}
f\mapsto G_{2}^{-1} P_+((G_1^{-1})^* f)
\end{align}
of $T\ci{F}$ is bounded.  It is not hard to see that the operator \eqref{e:FormInv} is bounded if 
and only if the weighted estimate 
\begin{align}\label{e:WNI 01}
\| P_+ f \|\ci{L^2(V)} \le C \|f\|\ci{L^2(W)}\qquad \forall f\in L^2(W)
\end{align}
with weights $W=G_1 G_1^*$, $V=(G_2^{-1})^* G_2^{-1}$ holds, and the norm of the operator is the 
best constant $C$ in \eqref{e:WNI 01}. This looks like a two weight inequality, but in reality it 
reduces to the one weight case. Namely, the invertibility of $T\ci{F}$ implies that $F$ is 
invertible in $L^\infty$, and therefore it is not hard to check that the weights $V$ and $W$ are 
equivalent,  i.e.\ 
\begin{align*}
A^{-1} W \le V \le A W, \qquad \text{for some }A\in(0, \infty); 
\end{align*}
the inequality is understood as a matrix inequality. Therefore at the cost of the constant $A$, 
\eqref{e:WNI 01} is equivalent to the one weight estimate 
\begin{align}\label{e:WNI 02}
\| P_+ f \|\ci{L^2(W)} \le C \|f\|\ci{L^2(W)}\qquad \forall f\in L^2(W), 
\end{align}
i.e.~to the boundedness of $P_+$ in the weighted space $L^2(W)$. On the real line the Riesz 
Projection $P_+$ and the Hilbert Transform are related as $\HT = -i P_+ + i (\bI -P_+)$, so $P_+$ 
is bounded in $L^2(W)$ if and only if $\HT$ is, and the norms are equivalent in the sense of 
two-sided estimates.

\subsection{Some scalar history}

As we already mentioned above, the first solution  for scalar weights is given  by the famous 
Helson--Szeg\"{o} Theorem \cite{HS}. Later in \cite{HMW} a real variable characterization of such 
weights was obtained, and the Muckenhoupt $A_2$ condition came to prominence. 

Finding quantitative estimates of the weighted norms of singular integral operators in terms of 
$A_p$ characteristics was a natural problem to investigate. The interest in sharp estimates peaked 
after the paper \cite{PV}, where the sharp weighted estimate $\lesssim\bcQ$  for the 
Beurling--Ahlfors 
Transform was obtained and used to solve a long standing problem about regularity of the Beltrami 
PDE at the critical exponent (Conjecture 12.9 in \cite{IwM} in dimension $2$). Note that for other 
exponents \emph{any} weighted estimate works, see \cite{AIS},  but for the critical exponent the 
sharp estimate $\lesssim\bcQ$ is required. The proof of this sharp estimate in \cite{PV} featured a 
novel  \emph{implicit Bellman function} approach using the heat equation.

For the Hilbert Transform the scalar $A_2$ conjecture was resolved in \cite{P}. 
The main part of the proof was purely dyadic and provided the upper bound $\lesssim\bcQ$  for  the 
so-called Haar Shift $ \ShaUp $ 
\cite{Pe}. The estimate for the Hilbert Transform follows, since the Hilbert Transform is (up to a 
multiplicative constant) the average of the Haar Shifts over all translations and dilations of the 
dyadic lattice:  an idea that had great impact on the study of the Hilbert Transform.

The area generated much interest at one of the forefronts of analysis, which resulted in 
extraordinary development that goes far beyond the interest in weights, see e.g. the following 
articles and the literature cited therein: \cite{LPR}, \cite{Hy}, \cite{HyPTV}, \cite{LN}, 
\cite{La}, \cite{CR}, \cite{Le}.  

For general Calder\'on-Zygmund operators the $A_2$ conjecture was settled in \cite{Hy} via 
representation of such operators as the average of generalized  Haar Shifts and and sharp estimates 
of such shifts.

Later, as a result of all this interest, the highly influential and flexible sparse
domination technique appeared in pioneering works begun by A. Lerner that found its greatest impact
in the development of the pointwise sparse domination. Here such a vast literature exists by now
that we cite only a few early papers: \cite{Le}, \cite{LN}, \cite{CR}, \cite{La}.  Again, these
sequels included completely new, brilliant ideas and received much recognition.

\subsection{Vector history}

The matrix $\bA_2$ 
condition was introduced  
in \cite{Treil-OTAA-89}, where it was conjectured that it is necessary and sufficient for the 
boundedness of the Hilbert Transform. This conjecture was proved in \cite{TV}, which started the 
development of the theory of matrix-valued weights. 

After much development in the scalar case, with a number of proofs  giving the linear bound, 
interest spiked in the matrix case, but the matrix $\mathbf A_2$ conjecture remained an intriguing 
question for years. The conjectured linear bound in the matrix case was a puzzling, long standing 
open question, resisting the abundance of new techniques that had appeared in the scalar theory 
over the last twenty years. Still, there was a firm belief in scalar estimates holding in the 
matrix case.  For example for the Christ--Goldberg maximal function \cite{CG} this was confirmed 
early on in \cite{IsKwPo}. The belief in a linear estimate for $\cH$ was strengthened as the matrix 
weighted Carleson Embedding Theorem was finally proved \cite{CT} circumventing the 
non-commutativity of matrix multiplication with an elegant argument. Later, the dyadic square 
function with matrix weight was shown to have a linear dependence \cite{HyPeVo}, \cite{Tr}. 
All these results hinted that the Hilbert Transform and other Calder\'on--Zygmund operators 
should have an upper bound of the form $\lesssim \bcQ$. 

Turning to the Hilbert Transform, the constant one obtains by combing ideas from the scalar 
estimate \cite{PP} with \cite{TV} and tracking certain constants in \cite{TV}, was improved 
in \cite{BPW} to an estimate of the form $\lesssim \bcQ^{3/2}\log\bcQ$.  
The logarithmic term  was then dropped in  \cite{NaPeTrVo} using the so-called convex body sparse 
domination, but the $3/2$ power remained. Many alternative proofs were known to us, all with power 
$3/2$.

There were however some hints that the answer in not that simple. 
An interesting mile stone was the failure of the convex body Carleson Embedding Theorem as well as 
failure of the boundedness of the convex body Maximal Function as demonstrated in \cite{NaPeSkTr}. 
In addition, while the matrix weighted Carleson Embedding Theorem holds \cite{CT}, its bilinear 
version  (a crucial ingredient in some sharp scalar proofs) fails so miserably that a trivial 
counterexample shows not only that it fails, but that the $\mathbf A_2$ characteristic is not the 
right quantity at all \cite{DoPeSk}.

But on the other hand,   it was shown in \cite{NaPeTrVo} that a simple sparse operator, based on a 
sparse family where each 
sparse interval has at most one sparse descendant, had the expected linear bound in the matrix 
case. Dominating singular integral operators using sparse families does not generally yield 
such a simple collection. However, analyzing the construction of dominating sparse operators and 
using the $T(1)$-Theorem for matrix weights \cite{BCTW2019} one can show that for weights with one 
singularity the estimate is reduced to estimating simple sparse operators. 
This means that any counterexample should be 
quite complicated, and cannot be anything like a weight with one or finitely many singularities.

\subsection{The idea of our counterexample}
\label{s:counter}

In this paper, we first show the desired lower bound for some dyadic operators in terms of the 
\emph{dyadic} $\mathbf A_2$ condition, see \eqref{e:A2 01} for the definition.

The weight we construct is necessarily quite complicated, and we construct it using the martingale 
approach, i.e.~by tracking the averages of the weight $W$ and its inverse $W^{-1}$.  

In our top down approach, we apply subtle rotations of the eigenvectors of the averages in one step 
(which is only possible in the vector case) and in the 
next step we disbalance the weight, which is the action also giving the sharp lower bound in the 
scalar case.  The rotation, which is only possible in the vector case, adds some extra blow up. 

For the constructed weight we then prove the estimate  from below 
$\gtrsim([W]\ci{\bA_2}\ut{dy})^{3/2}$  for several dyadic operators relevant to our task.  

We start with a special sparse paraproduct, which will be the main building block for all the 
estimates. We also show that the classical Haar Shift $\ShaUp$ has the desired bound. 
We then prove the lower bound for the special dyadic operator $\Hdy$ which in our situation serves 
as a dyadic model for the Hilbert Transform.

It is well known that the average of the Haar Shift gives the Hilbert Transform -- as such, getting 
upper estimates for Hilbert Transform from the estimates of the Haar Shift is more  immediate, 
intuitive and obvious; this idea has been used for years.

The case of lower bounds is much more involved.   
In our passage from the dyadic operators $\Hdy$ to the Hilbert Transform we 
take inspiration from Bourgain \cite{Bo} in the use of separating frequencies (the 
periodizations). A modification of his construction by considering the second order martingale 
differences allowed us to obtain the lower bounds for 
the Hilbert Transform $\HT$ (with a modified weight) from  the lower bounds for the dyadic operator 
$\Hdy$ . The choice of the dyadic operator $\Hdy$ is forced upon us by the construction, we do not 
have any freedom here.

But the construction does not stop here. The weight we constructed for the dyadic operator and its 
periodization are in the dyadic $\bA_2$ 
class $\bA_2\ut{dy}$; however they both  have large jumps and do not satisfy the classical matrix
$\bA_2$ condition.

To overcome this difficulty we borrow from the technique of quasi-periodization,  developed by 
F.~Nazarov \cite{Nazarov} 
in his counterexample to the Sarason 
conjecture.  This technique allows us to get the classical $\bA_2$ condition, but only in a two 
weight situation. However in the present paper we are in the one weight 
situation, and thus have even less freedom. So, in the final step we use a modification of the 
so-called \emph{iterated remodeling} technique, developed in \cite{KakTre21} for the scalar case.

\section{Preliminaries}
\label{s:prelim}

\subsection{Some definitions}
Our construction starts in the dyadic martingale setting, so let us remind the reader of the main 
definitions. A \emph{dyadic interval} in this paper is an interval $I$ of the form
\begin{align*}
2^k\Bigl([0,1) + j\Bigr), \qquad j,k\in \Z. 
\end{align*}
For a dyadic interval $I$ we denote by $\cD(I)$ all dyadic subintervals of $I$ (including $I$). 

We will use the notation $I^0:=[0,1)$, and we often abbreviate $\cD:=\cD(I^0)$. For a 
dyadic  interval $I$ we denote by $\ch^k(I)$, $k\in\N$ its descendants of $k$th generation, 
i.e.~the the set of all dyadic subintervals of $I$ of length $2^{-k}|I|$. We will also use the
notation $\cD_k:=\ch^k(I^0)$. 

Finally, we denote by $I_+$ and $I_-$ the right and left child of $I$ respectively.

We say that a weight $W$ on $I^0$ satisfies the \emph{dyadic} matrix $\bA_2$ condition, and write 
$W\in \bA_2\ut{dy}$ if 
\begin{align}\label{e:A2 01}
[W]\ci{\bA_2}\ut{dy} := \sup_{I\in\cD} 
\left\| \la W\ra\ci I^{1/2} \la W^{-1}\ra\ci I^{1/2} \right\|^2  < \infty\,;
\end{align}
the difference from the classical matrix $\bA_2$ condition is that in the dyadic $\bA_2\ut{dy}$ 
condition we take the supremum only over dyadic intervals, not over all intervals as in 
\eqref{e:A2}. 

\subsection{Properties of \tp{$\bA_2$}{A2} weights}
For Hermitian matrices $A$, $B$ the notation $A\le B$ means that the matrix $B-A$ is positive 
semidefinite (denoted $B-A\ge0$). 
\begin{lm}
\label{l:A2_LMI_prelim}
Let $\bv$, $\bw$ be $d\times d$ positive definite matrices. Then the inequality 
$\|\bv^{1/2}\bw^{1/2}\|^2\le \bcQ$ is equivalent to 
\begin{align}\label{e:A2 to LMI}
\bv\le \bcQ \bw^{-1}, \qquad\text{or, equivalently}\qquad \bw\le \bcQ\bv^{-1}. 
\end{align}
\end{lm}
\begin{proof}
Taking the Hermitian square of $\bv^{1/2}\bw^{1/2}$ and recalling that $\|A\|^2 =\|A^*A\|$ we see 
that the condition $\|\bv^{1/2}\bw^{1/2}\|^2\le \bcQ$ can be rewritten as 
\begin{align*}
\left\| \bw^{1/2} \bv\bw^{1/2}\right\| & \le \bcQ
\intertext{or, equivalently}
\bw^{1/2} \bv\bw^{1/2} &\le \bcQ \bI. 
\end{align*}
Left and right multiplying the last inequality by $\bw^{-1/2}$ we get the first inequality in 
\eqref{e:A2 to LMI}; the second inequality in \eqref{e:A2 to LMI} is trivially equivalent to the 
first one. 
\end{proof}

\begin{lm}
\label{l:A2_LMI}
Let $W$ be a $d\times d$ matrix weight. Then
\begin{align}\label{e:weight LMI}
\la W \ra\ci I^{-1} \le \la W^{-1} \ra\ci I, \quad \text{or equivalently}\qquad 
\la W^{-1} \ra\ci I^{-1} \le \la W \ra\ci I. 
\end{align}
Moreover
 $\|\la W \ra\ci I^{1/2} \la W^{-1} \ra\ci I^{1/2} 
\|^2\le \bcQ$ if and only if 
\begin{align}\label{e:A2_LMI 01}
 \la W^{-1} \ra\ci I \le \bcQ \la W \ra\ci I^{-1}, \quad \text{or equivalently}\qquad 
  \la W \ra\ci I \le \bcQ \la W^{-1} \ra\ci I^{-1}. 
\end{align}
\end{lm}

\begin{proof}
Inequalities \eqref{e:A2_LMI 01} follow immediately from Lemma \ref{l:A2_LMI_prelim}. As for the 
inequalities \eqref{e:weight LMI}, they are just Property 1 of the $\cA_{p,q}$ weights in 
\cite[s.~2.5]{NaTr-Hunt}, see also \cite[Lemma 1.4]{LaTr2007}; in both cases $p=q=2$. 

To save the reader a trip to a library and the hurdle of translating from a different 
(mathematical) 
language, we present a proof here: it is the same proof as in \cite{NaTr-Hunt,LaTr2007} written in 
different notation.   Namely, we get for arbitrary $\be\in\F^d$
\begin{align*}
\left\La \left(\La W\Ra\ci{I}^{-1}\be, \be\right)\ci{\F^d}   \right\Ra_{I} &=
\left\La \left(W^{1/2}\La W\Ra\ci{I}^{-1}\be, W^{-1/2}\be\right)\ci{\F^d}   \right\Ra_{I} \\
&\le 
\left\La \left(\La W\Ra\ci{I}^{-1} W\La W\Ra\ci{I}^{-1}\be, \be\right)\ci{\F^d}   
\right\Ra_{I}^{1/2}
\left\La \left(W^{-1}\be, \be\right)\ci{\F^d}   \right\Ra_{I}^{1/2} \\
&=
\left(\La W\Ra\ci{I}^{-1} \be, \be\right)\ci{\F^d}^{1/2} 
\left(\La W^{-1}\Ra\ci{I} \be, \be\right)\ci{\F^d}^{1/2} . 
\end{align*}
Dividing both sides by $\left(\La W\Ra\ci{I}^{-1} \be, \be\right)\ci{\F^d}^{1/2}$ and squaring, we 
get the first inequality in \eqref{e:weight LMI}; the second inequality  is trivially equivalent to 
the first one. 
\end{proof}

\begin{rem}\label{r:matrix rescaling}
It is trivial and well-known that the matrix rescaling of the weight $W\mapsto AWA$, where 
$A=A^*>0$ is a constant matrix, does not change the $\bA_2$ characteristic, 
\[
[W]\ci{\bA_2} = [AWA]\ci{\bA_2}. 
\] 
It is also well-known and easy to see that for an operator  of the form $T\otimes \bI$, where $T$ 
is 
an operator acting on scalar-valued functions, its norms  in $L^2(W)$ and in $L^2(AWA) $ coincide. 
\end{rem}

\begin{lm}
\label{l:stop}
Let $\bw=\bw^*>0$ and $\bv=\bv^*>0$ be $d\times d$ matrices satisfying 
\begin{align}
\label{e:LMI 000}
\bv^{-1} \le \bw \le \bcQ\bv^{-1}. 
\end{align}
Then, given an interval $I$, there exists an $\bA_2$ weight $W$ on $I$, constant on the children of 
$I$, such that 
\begin{align*}
\la W \ra\ci{I} = \bw, \qquad \la  W^{-1}\ra\ci{I} =\bv, 
\end{align*}
and  $[W]\ci{\bA_2}\le \bcQ$. 
\end{lm}

\begin{rem}\label{r:need only dyadic}
In our construction we only need the fact that the resulting weight $W$ is a \emph{dyadic} $\bA_2$ 
weight satisfying $[W]\ci{\bA_2}\ut{dy} \le \bcQ$, so the reader can skip the part of the proof 
below dealing with the estimate $[W]\ci{\bA_2}\le \bcQ$. However this estimate is of independent 
interest, so we present its proof here. 
\end{rem}

\begin{proof}[Proof of Lemma \ref{l:stop}] 
As it was mentioned above in Remark \ref{r:matrix rescaling} the $\bA_2$ characteristic 
$[W]\ci{\bA_2}$ is invariant under the matrix rescaling $W\mapsto AWA$, $A=A^*>0$. 
The condition \eqref{e:LMI 000} is also invariant under the matrix rescaling $\bw\mapsto A \bw A$, 
$\bv \mapsto A^{-1} \bv A^{-1}$, $A=A^*>0$,  so without loss of generality we may assume $\bw=\bI$. 

Define $W:=(\bI+\Delta)\1\ci{I_+}+  (\bI-\Delta)\1\ci{I_-}$, where $\Delta =\Delta^*\ge 0$ is a 
constant matrix to be found. Through the 
use of  Neumann series or simple algebra one checks that 
$$
\left(\La W \Ra\ci{I_+} + \La W \Ra\ci{I_-} \right)/2= \left( (\bI+\Delta)^{-1}+(\bI-\Delta)^{-1} 
\right)/2=(\bI-\Delta^2)^{-1}.
$$
The above average should be equal to $\bv$, so  
$\Delta^2 = \bI-\bv^{-1} \ge 0$ (because $\bv^{-1}\le \bw$). Taking the unique positive root gives 
us the desired $\Delta$. 

Since the weight $W$ is constant on $I_\pm$, it is trivial that $[W]\ci{\bA_2}\ut{dy}\le\bcQ$.  The 
reader only interested in the construction of the counterexample can stop here, see Remark 
\ref{r:need only dyadic}.  

But for an inquisitive reader  let us show that in fact $[W]\ci{\bA_2}\le\bcQ$. Take an interval 
$J\subset I$, and let 
$J_1:=J\cap I_+$, $\alpha:= |J_1|/|J|$. 
Then, still assuming $\bw=\bI$ and recalling that $\bv=(\bI-\Delta^2)^{-1}$, we get 
\begin{align*}
\La W \Ra\ci{\!J} &= \alpha (\bI+\Delta) + (1-\alpha) (\bI-\Delta) = \bI + (2\alpha-1)\Delta 
=:A_\alpha 
\\
\La W^{-1} \Ra\ci{\!J} &= \frac{\alpha}{\bI+\Delta} + \frac{1-\alpha}{\bI-\Delta} 
=\frac{\bI - (2\alpha-1)\Delta}{\bI -\Delta^2} 
=: \frac{B_\alpha}{\bI -\Delta^2} 
\\
&=B_\alpha^{1/2} \left( \bI - \Delta^2 \right)^{-1} B_\alpha^{1/2}
=B_\alpha^{1/2} \bv B_\alpha^{1/2}. 
\end{align*}
The estimate $\bw\le\bcQ \bv^{-1}$ is equivalent to $\bv\le\bcQ\bw^{-1} = \bcQ\bI$, so 
$\|\bv^{1/2}\|\le \bcQ^{1/2}$. The trivial scalar inequality 
\[
1-x\le \frac{1}{1+x}
\]
implies the matrix version
\begin{align*}
B_\alpha=\bI - (2\alpha-1)\Delta \le (\bI + (2\alpha-1)\Delta)^{-1} =A_\alpha^{-1}, 
\end{align*}
and since $\|\bv^{1/2}\|\le\bcQ^{1/2}$ we conclude that 
\begin{align*}
\La W^{-1} \Ra\ci{\!J} = B_\alpha^{1/2} \bv B_\alpha^{1/2} \le \bcQ B_\alpha \le \bcQ A_\alpha^{-1} 
= \bcQ \La W \Ra\ci{\!J}^{-1}. 
\end{align*}
Since $J$ is an arbitrary interval, Lemma \ref{l:A2_LMI} implies that $[W]\ci{\bA_2}\le\bcQ$. 
\end{proof}

\subsection{Representation of weights as martingales} 

It is a common place in dyadic harmonic analysis to treat functions as martingales, i.e.~not to 
track  functions  themselves, but their averages over dyadic intervals. Namely, if we define 
the averaging operators (conditional expectations) $\bE\ci{I}$, $\bE_k$, 
\begin{align*}
\bE\ci{I}\vf =\La \vf \Ra\ci{I}\1\ci{I}, \qquad \bE_k:= \sum_{I\in\cD_k} \bE\ci{I}, 
\end{align*}
then given a function $\f\in L^1(I^0)$ we can see that the sequence $(\vf_n)_{n\ge0}$, 
$\vf_n:=\bE_n \vf$ is a martingale, meaning that  $\bE_n\vf_{n+1} =\vf_n$. 

From a na\"{i}ve point of view  one just tracks the average $\vf\ci{I}:=\La \vf \Ra\ci{I}$; the 
martingale property simply means that $\vf\ci{I}=(\vf\ci{I_+}+\vf\ci{I_-})/2$ for all $I\in\cD$. 
Of course not all martingales come from functions, the so-called \emph{uniform integrability} is 
the necessary and sufficient condition. But in our construction all martingales will stabilize 
after 
finitely many steps, so they will be trivially given by averages of some functions.

The following martingale difference decomposition is well known,
\[
\vf_n
= \sum_{I\in\cD_n}  \bE\ci{I} f
= \bE_0 \vf + \sum_{k=0}^{n-1}\sum_{I\in\cD_k} \Delta\ci{I}\vf
= \la \vf \ra_{I^0} + \sum_{k=0}^{n-1}\sum_{I\in\cD_k} (\vf,h\ci{I}) h_I;
\]
here
\begin{align*}
\Delta\ci{I}\vf := \bE\ci{I}\vf -\sum_{J\in\cH(I)} \bE\ci{\!J}\vf = (\vf, h\ci{I})h\ci{I},  
\end{align*}
and $h\ci{I}:= |I|^{-1/2} \left(\1\ci{I_+}-\1\ci{I_-}\right)$ is the normalized Haar function. 

In our construction of the (dyadic) $\bA_2$ weights we will track the averages $\bw\ci{I}=\La 
W\Ra\ci{I}$, $\bv\ci{I}=\La V\Ra\ci{I}$, satisfying the matrix inequalities
\begin{align}
\label{e:LMI 001}
\bv\ci{I}^{-1} \le \bw\ci{I} \le \bcQ\bv\ci{I}^{-1}. 
\end{align}
The martingales will terminate (i.e.~stabilize) after finitely many steps, so trivially they are 
obtained from weights $W$, $V$. Adding one more step and using Lemma \ref{l:stop} there, we can 
guarantee that $V=W^{-1}$. The fact that on each step the inequalities  \eqref{e:LMI 001} are 
satisfied implies that the resulting weight is a \emph{dyadic} $\bA_2$ weight and that 
$[W]\ci{\bA_2}\ut{dy}\le \bcQ$. 

\begin{rem}
We do not claim here that the weight $W$ is an $\bA_2$ weight, we only claim that it is a 
\emph{dyadic} $\bA_2$ weight. The constructed dyadic weight $W$ gives the lower bound 
$\gtrsim\bcQ^{3/2}$ for 
the norms of some dyadic operators in $L^2(W)$. 

To get the lower bound for the Hilbert Transform $\cH$ we perform the so-called \emph{iterated 
remodeling} on the weight $W$, to get a new weight $\wt W$. This new weight will satisfy the 
$\bA_2$ 
condition with $[W]\ci{\bA_2}\le 16 \bcQ$, and $\|\cH\|\ci{L^2(W)\to L^2(W)}\gtrsim \bcQ^{3/2}$. 
\end{rem}

\section{Construction of the dyadic weight}
\label{s:constr}

\subsection{Families of weights}
In this section we will construct a family of weighs $W=W\ci{\bcQ, \delta_0}$ depending 
on the dyadic $\bA_2$ characteristic $\bcQ=: [W]\ci{\bA_2}\ut{dy}$ and some auxiliary small 
parameter $\delta_0:= q (\beta_0/\alpha_0)^{1/2} = q E_0^{-1/2}$, $q=0.1$,  where 
$E_0:=\alpha_0/\beta_0$ is 
the initial eccentricity of the average $\La 
W^{-1}\Ra\ci{I^0}$, see \eqref{e:av01}.  We will then show in Sections \ref{s:est-para}, 
\ref{s:Lower Haar shifts}  that for 
any sufficiently large 
$\bcQ$ we will have 
for all sufficiently small $\delta_0$  the desired lower bound  $c \bcQ^{3/2}$ for relevant 
dyadic operators. 

For typographical reasons we will usually omit the indices $\bcQ$ and $\delta_0$, 
but the reader should keep in mind that $W$ is not one weight but the whole family. 

Let us introduce some notation, simplifying the write-up. The symbols $\lesssim$ and $\gtrsim$ mean
the inequalities with some absolute constants, and $\asymp$ means the two sided estimate with
absolute constants. 
We will use the notation
$\underset{\delta_0<}{\asymp}$, $\underset{\delta_0<}{\lesssim}$, $\underset{\delta_0<}{\gtrsim}$
to indicate 
that the estimates hold for sufficiently small $\delta_0$, and that  the same implied constants 
work for \emph{all} variables involved and for 
\emph{all} sufficiently small $\delta_0$. 
Similarly we will use the notation $O\ci{\delta_0<}(X)$ for a quantity whose absolute value  is 
bounded by $CX$ with an absolute constant $C$ for all sufficiently small $\delta_0$ (and for all 
variables involved). 

We 
will use $\underset{\delta_0<}{\le}$, $\underset{\delta_0<}{\ge}$ to indicate that the inequalities 
hold for all sufficiently small $\delta_0$ (with the constants explicitly presented).   

And finally, the symbol $\underset{\delta_0<}{\sim}$ means that the ratio of right and left hand 
side is as close to $1$ as we want, if $\delta_0$ is sufficiently small.

\subsection{Starting point}
We will start with the interval $I=I^0$, and on this interval the averages $\la 
W^{-1}\ra\ci{I}=\bv\ci{I^0}$ and $\la W\ra\ci{I^0}=\bw\ci{I^0}$ will be diagonal in the same  
orthonormal basis $\ba\ci{I^0}$, $\bb\ci{I^0}$, 
\begin{align}
\label{e:av01}
\la W^{-1}\ra\ci{I^0}=\bv\ci{I^0} =\alpha_0 \ba\ci{I^0}\ba\ci{I^0}^* + \beta_0 \bb\ci{I^0} 
\bb\ci{I^0}^*, \qquad 
\la W\ra\ci{I^0}=\bw\ci{I^0} = \beta_0^\# \ba\ci{I^0}\ba\ci{I^0}^* + \alpha_0^\# 
\bb\ci{I^0} \bb\ci{I^0}^*. 
\end{align}
We will require that $\alpha_0 \beta_0^\# = \beta_0\alpha_0^\#=\bcQ$, and 
 that $\alpha_0 \gg\beta_0$. 
 
\begin{rem}\label{r:rescaling}
The assumption that $\alpha_0 \gg\beta_0$ might look superfluous because of the rescaling 
invariance of the problem.  
Indeed, for the rescaling of a weight $W$, i.e.~for the weight $AWA$, where 
$A=A^*>0$ is a constant matrix, we have
\begin{align*}
[AWA]\ci{\bA_2} = [W]\ci{\bA_2}, 
\end{align*}
i.e.~the rescaling preserves the $\bA_2$ characteristics. 
 
It is also easy to see that for operators of the form $T\otimes \bI$, where $T$ is 
an operator acting on scalar-valued functions, its norms  in $L^2(W)$ and in $L^2(AWA) $ coincide.

Thus,  one can always rescale the weight and assume without loss of generality that 
$\alpha_0=\beta_0$. 

However, choosing a large 
eccentricity simplifies formulas by allowing us to easily identify terms we can ignore, and that is 
the  reason we assume that the initial eccentricity  
$E_0:=\alpha_0/\beta_0$ is as large as we want.  
\end{rem}

The initial eccentricity  $E_0 = \alpha_0/\beta_0 = \alpha_0^\# / \beta_0^\#$, or equivalently  
the quantity $\delta_0 : = q (\beta_0/\alpha_0)^{1/2}= q 
E_0^{-1/2}$ with $q=0.1$, will be the parameter of the construction.  We  will choose $\delta_0$ as 
small as we want, so $E_0$ will be  as large as we want. 

\subsubsection{About initial eigenvalues}
Now, what about the choice of the initial eigenvalues $\alpha_0$, $\beta_0$? Since in our 
construction we will 
not be taking the limit as $\delta_0\to0$ (equivalently as $E_0\to\infty$), but just take a 
sufficiently small $\delta_0$ that gives us the desired bound, we do not need to worry much about 
it. For each value of $\delta_0$ we just pick $\alpha_0$, $\beta_0$ satisfying 
$q\cdot(\beta_0/\alpha_0)^{1/2} = \delta_0$; the eigenvalues $\alpha_0^\#$, $\beta_0^\#$ are 
uniquely determined by the relations $\alpha_0 \beta_0^\# = \beta_0\alpha_0^\#=\bcQ$. 

The reader uncomfortable with such freedom can always assume that the eigenvalue $\alpha_0$ is 
fixed. 
Alternatively, given $\delta_0$ one can choose eigenvalues symmetrically, 
\begin{align*}
\alpha_0=\alpha_0^\# = q \delta_0^{-1}\bcQ^{1/2}, \qquad \beta_0=\beta_0^\# = q^{-1} \delta_0   
\bcQ^{1/2}  .
\end{align*}

\subsection{Rotations}
\label{s:rot}
On the next step of the construction we perform what we call the \emph{rotation}. 

Suppose we are given a positive self-adjoint operator $\bv$ in $\R^2$, 
\begin{align*}
\bv = \alpha \ba\ba^* +\beta \bb\bb^*, 
\end{align*}
where $\ba$, $\bb$ is an orthonormal basis in $\R^2$ and $\alpha, \beta>0$. 
Recall that we agreed that $\alpha \gg \beta$. 

Let us find matrices $\bv_+$ and $\bv_-$ that have a similar form to $\bv$ but in slightly 
rotated bases, such that $\bv= (\bv_+ +\bv_-)/2$. 

Namely, define 
\begin{align*}
\bv_\pm = \wt\alpha \ba_\pm \ba_\pm^* + \wt\beta \bb_\pm \bb_\pm^*, 
\end{align*}
where
\begin{align}
\label{e:ab_pm}
\ba_\pm = (1+\delta^2)^{-1/2} (\ba \pm \delta \bb), \qquad \bb_\pm = (1+\delta^2)^{-1/2} (\bb \mp 
\delta \ba), 
\end{align}
with appropriately chosen small parameter $\delta$.  Note, that here we just rotated the basis 
$\ba$, $\bb$ through the angle 
$\pm\theta$,  $\theta = \arctan\delta$. 

If $\delta$ is fixed, the eigenvalues $\wt\alpha$ and $\wt\beta$ (note that they are the same for 
$\bv_+$ and $\bv_-$) are uniquely defined by the martingale relations $\bv= (\bv_+ +\bv_-)/2$. 
Namely, noticing that
\begin{align*}
\frac{\ba_+\ba_+^* + \ba_-\ba_-^*}{2} &= \frac{\ba \ba^*}{1+\delta^2} + \frac{\delta^2\bb 
\bb^*}{1+\delta^2},
\intertext{and}
\frac{\bb_+\bb_+^* + \bb_-\bb_-^*}{2} &= \frac{\delta^2\ba \ba^*}{1+\delta^2} + \frac{\bb 
\bb^*}{1+\delta^2},
\end{align*}
we conclude that 
\begin{align*}
\alpha = \frac{\wt \alpha + \delta^2 \wt \beta}{1+\delta^2}, \qquad 
\beta = \frac{ \delta^2\wt\alpha +  \wt \beta}{1+\delta^2}. 
\end{align*}
Solving for $\wt\alpha$, $\wt\beta$ we get 
\begin{align}
\label{e:wt alpha beta}
\wt\alpha = \frac{\alpha - \delta^2\beta}{1-\delta^2}, \qquad 
\wt\beta = \frac{\beta - \delta^2\alpha}{1-\delta^2} . 
\end{align}
Note that $\wt\beta>0$ if and only if $\delta^2<\beta/\alpha$, and if we follow our assumption that 
$\alpha>\beta$, then in this case $\wt\alpha>\wt\beta$, so both eigenvalues are positive. It is 
also clear that in this case $\wt\alpha>\alpha$. 

In what follows we will set $\delta= q (\beta/\alpha)^{1/2}$, where $q$ is a 
small parameter (say 
$q=0.1$)  to be 
fixed later at the beginning of the construction, before all other parameters.%

\subsubsection{Applying the rotations}
In what follows we will apply the rotations every other generation, for example for some intervals 
in the even generations. On each such \emph{rotation} interval $I$ we will have  
\begin{align*}
\la W^{-1} \ra\ci{I} = \bv\ci{I} = \alpha\ci{I} \ba\ci{I} \ba\ci{I}^* + \beta\ci{I} \bb\ci{I} 
\bb\ci{I}^*, \qquad 
\la W \ra\ci{I} = \bw\ci{I} = \beta\ci{I}^\# \ba\ci{I} \ba\ci{I}^* + \alpha\ci{I}^\# 
\bb\ci{I} \bb\ci{I}^*
\end{align*}
where $\alpha\ci{I}\gg \beta\ci{I}$ and 
\begin{align*}
\alpha\ci{I}\beta\ci{I}^\# =\bcQ = \alpha\ci{I}^\#\beta\ci{I},
\end{align*}
hence $\alpha\ci{I}^\#/\beta\ci{I}^\# = \alpha\ci I /\beta\ci I \gg1$. 

In our construction the eigenvalues $\alpha\ci{I}$, $\alpha\ci{I}^\#$, $\beta\ci I$, $\beta\ci 
I^\#$ will depend only upon the generation of the rotation interval $I$ (i.e.~only on its size). 

Now we apply the rotation operation with the parameter $\delta=\delta\ci I = 
q(\beta\ci 
I/\alpha\ci I)^{1/2}=q(\beta\ci I^\#/\alpha\ci I^\#)^{1/2}$  to obtain the averages of $W^{-1}$ and 
$W$ on 
the 
children $I_+$, $I_-$ of $I$, 
\begin{align*}
\la W^{-1}\ra\ci{I_\pm} = \bv\ci{I_\pm} = \wt \alpha\ci I \ba\ci{I_\pm} \ba\ci{I_\pm}^* + 
\wt\beta\ci I \bb\ci{I_\pm} \bb\ci{I_\pm}^*, \qquad
\la W\ra\ci{I_\pm} = \bw\ci{I_\pm} = \wt \beta\ci I^\# \ba\ci{I_\pm} \ba\ci{I_\pm}^* + 
\wt\alpha\ci I^\# \bb\ci{I_\pm} \bb\ci{I_\pm}^*, 
\end{align*}
where the eigenvectors $\ba\ci{I_\pm}$, $\bb\ci{I_\pm}$ are given by \eqref{e:ab_pm}
\begin{align*}
\ba\ci{I_\pm} = (1+\delta^2)^{-1/2} (\ba\ci I \pm \delta \bb\ci I), \qquad \bb\ci{I_\pm} = 
(1+\delta^2)^{-1/2} (\bb\ci I \mp 
\delta \ba\ci I), 
\end{align*}
and the eigenvalues change according to \eqref{e:wt alpha beta}
\begin{align}
\label{e:wt alpha beta 01}
\wt\alpha\ci I = \frac{\alpha\ci I - \delta^2\beta\ci I}{1-\delta^2}, \quad 
\wt\beta\ci I = \frac{\beta\ci I - \delta^2\alpha\ci I}{1-\delta^2} , \qquad
\wt\alpha\ci I^\# = \frac{\alpha\ci I^\#  - \delta^2\beta\ci I^\#}{1-\delta^2}, \quad 
\wt\beta\ci I^\# = \frac{\beta\ci I^\# - \delta^2\alpha\ci I^\#}{1-\delta^2}. 
\end{align}
Note that the $\bA_2$ constant of the averages on $I_\pm$ will change a little, 
\begin{align}\label{e:s}
\wt\alpha\ci I \wt\beta\ci I^\# = \wt\alpha\ci I^\# \wt\beta\ci I = s\ci I \bcQ, 
\end{align}
where
\begin{align*}
s\ci I = \frac{(1-q^2)(1- q^2\beta\ci I^2/\alpha\ci I^2 )}
{(1- q^2\beta\ci I/\alpha\ci I)^2}  \underset{\delta_0<}{\sim}  1-q^2;
\end{align*}
the symbol $\underset{\delta<}{\sim}$ means that the ratio of right and left hand side is as 
close to $1$ as we want, if $\delta$ is sufficiently small. 

If $\alpha\ci I \gg \beta\ci I$, which will be the case in our construction, $s\ci I$ is close to 
$1-q^2$, so repeating this rotation we would get a  degradation (decreasing) of the $\bA_2$ 
constant 
as intervals get smaller.   

Fortunately, it is possible to return the $\bA_2$ constant to $\bcQ$ on the next step, 
\emph{stretching}. While the main purpose of the stretching is increasing the eigenvalues 
$\alpha\ci I$, $\alpha\ci I^\#$ as much as possible, the possibility of correcting the $\bA_2$ 
constant 
is a nice bonus, making the construction possible.     

Note that in our construction $\alpha\ci I>\beta\ci I$ so $s\ci I > 1-q^2=0.99$.

\begin{figure}[h]
\includegraphics[width=10cm]{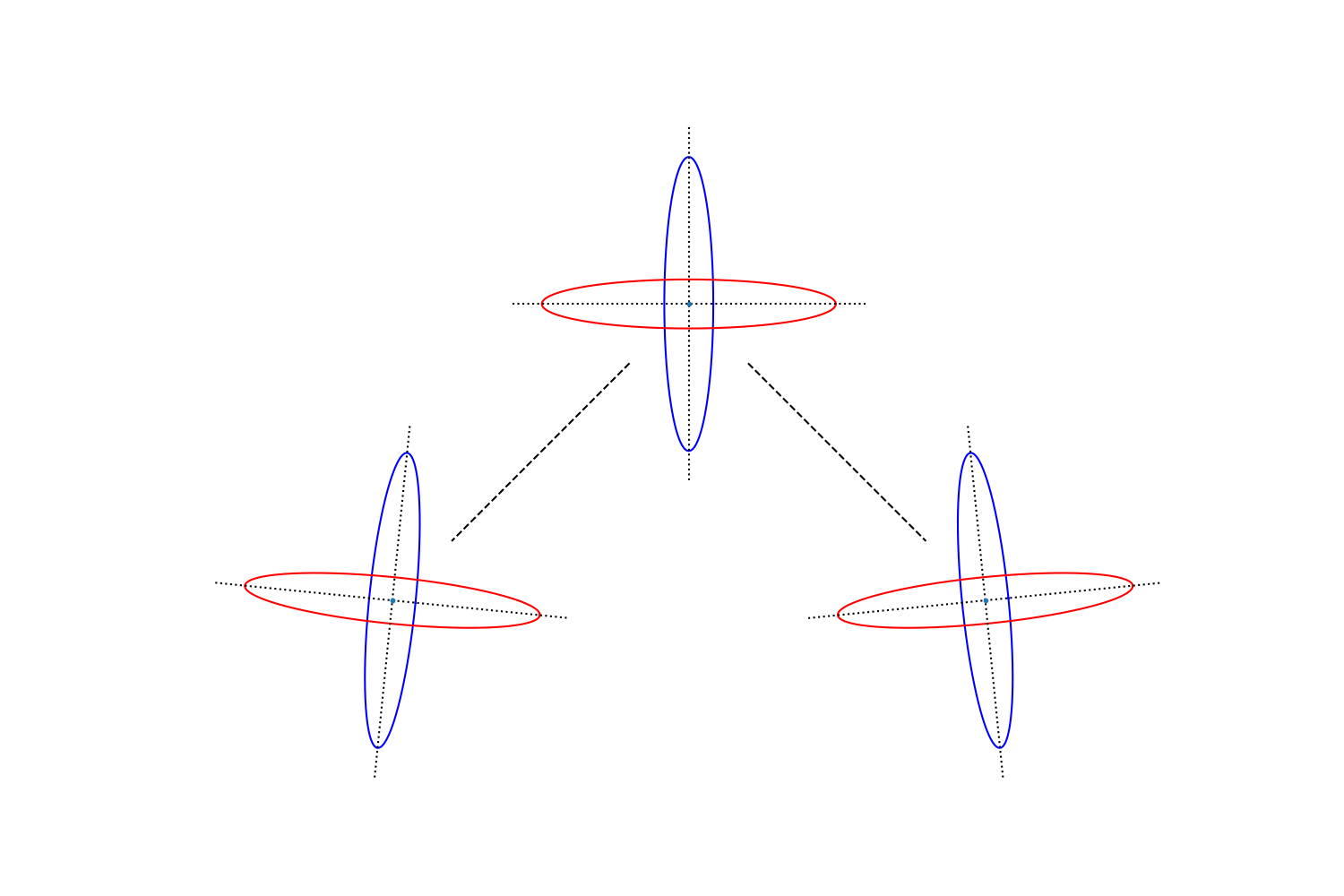}
\caption{Rotation: The picture shows the rotation step. The blue ellipses are the images of the 
unit ball under operators $\La W \Ra_I$ (the top line) and $\La W \Ra_{I_\pm}$ (bottom line). The 
red ellipses are the corresponding images under the averages $\La W^{-1} \Ra_I$ and $\La W^{-1} 
\Ra_{I_\pm}$ respectively. 
}
\end{figure}

\subsection{Stretching}
\label{s:stretch} 
The next step of the construction is the so-called \emph{stretching}. It will be performed 
individually in each eigenspace, so it is an operation on scalar objects (eigenvalues). 

Fix a \emph{stretching parameter} $r=2-1/\bcQ$. 
Consider a point $z=(x,y)\in\R^2$, $x,y>0$, $xy = s\bcQ$, where $0.9\le s\le 1$. Then for $z_+ = 
(x_+, y_+)$,  
\[
x_+ = rx, \qquad y_+ = y/(sr),
\]
we have $x_+y_+=\bcQ$.

A point $z_- = (x_-,y_-)$ such that $z=(z_+ + z_-)/2$ is given by 
\begin{align}\label{e:stretch -parts}
x_- = (2-r)x = x/\bcQ, \qquad y_- = (2-1/(sr))y. 
\end{align}
Then 
\begin{align*}
x_-y_- = (2-1/(sr))xy/\bcQ = (2-1/(sr))s . 
\end{align*}
As we are interested in asymptotics with large values of $\bcQ$, we can assume that $\bcQ>2$, so 
recalling that 
$s\ge 0.9$ we can see that 
\begin{align*}
1\le x_-y_- 
\le 2.
\end{align*}
In particular $z_-=(x_-, y_-)$ lies in the region $\left\{(x,y): x,y>0, \, 1\le xy \le \bcQ   
\right\}$ as required.
However, we keep in mind that $x_-y_- \ll x_+ y_+ = \bcQ$, as well as $x_-y_- \ll x y = s\bcQ$ for 
large values of $\bcQ$. We refer to Figure \ref{F:stretching} below for a sketch of the stretching 
step. 
\begin{figure}[ht]
 \includegraphics[width=15cm]{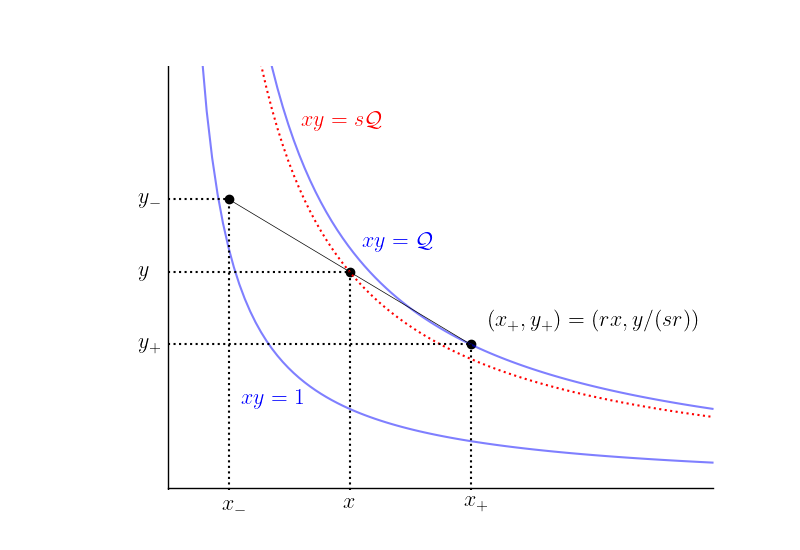}
 \caption{Stretching $(x,y) \mapsto (x_+,y_+)$}
 \label{F:stretching}
\end{figure}

\subsection{Complete construction}\label{ss:complete construction}
So, let us gather everything together. We start with the interval $I=I^0 =[0,1)$, which will give 
us the initial generation  
$\cS_0$ of \emph{stopping intervals}.   
As we are interested in asymptotics, $\bcQ$ can be as large as we want.

\subsubsection{Inductive construction and stopping intervals}\label{s:stopping int}
We will construct consequent generations $\cS_n$ of stopping intervals by induction.  Let $\cS_n$ 
be constructed. 
Take $I\in\cS_n$. 
The averages 
$\la W^{-1}\ra\ci{I}=\bv\ci{I}$ and $\la W\ra\ci{I}=\bw\ci{I}$ 
will be given by 
\begin{align*}
\La W^{-1}\Ra\ci{I} = \alpha_n \ba\ci{I} \ba\ci{I} ^* + \beta_n \bb\ci{I} \bb\ci{I}^* , \qquad 
\La W \Ra\ci{I}  = \beta_n^\# \ba\ci{I} \ba\ci{I} ^* + \alpha_n^\# \bb\ci{I} \bb\ci{I}^*, 
\end{align*}
where $\alpha_n \beta_n^\# = \alpha_n^\# \beta_n = \bcQ$.  Note that the eigenvalues $\alpha_n$, 
$\beta_n$, 
$\alpha_n^\#$, $\beta_n^\#$ depend only on the generation $\cS_n$ the interval $I$ is in, but not 
on the interval $I$ itself. 

We will see from the construction that  $\alpha_n\beta_n^\# = \alpha_n^\# \beta_n =\bcQ$ (true for 
$n=0$). 

We then apply the \emph{rotation}, described in Section \ref{s:rot} with parameter 
$\delta = \delta_n = q(\beta_n/\alpha_n)^{1/2} = q E_n^{-1/2}$  (recall 
that $E_n:=\alpha_n/\beta_n 
= \alpha_n^\#/\beta_n^\#$), to get the averages on the children $I_\pm$ of $I$. 
The averages are  
\begin{align}\label{e:Ave W}
\La W^{-1}\Ra\ci{I_\pm} = \wt\alpha_n \ba\ci{I_\pm} \ba\ci{I_\pm} ^* + \wt\beta_n \bb\ci{I_\pm} 
\bb\ci{I_\pm}^* , \qquad 
\La W \Ra\ci{I_\pm}  = \wt\beta_n^\# \ba\ci{I_\pm} \ba\ci{I_\pm} ^* + \wt\alpha_n^\# \bb\ci{I_\pm} 
\bb\ci{I_\pm}^*;   
\end{align}
the  eigenvalues $\wt\alpha_n$, $\wt\beta_n$, $\wt\alpha_n^\#$, $\wt\beta_n^\#$ 
are the same for both intervals, and are 
given by \eqref{e:wt alpha beta 01}
\begin{align}
\label{e:wt alpha beta 02}
\wt\alpha_n = \frac{\alpha_n - \delta_n^2\beta_n}{1-\delta_n^2}, \quad 
\wt\beta_n = \frac{\beta_n - \delta_n^2\alpha_n}{1-\delta_n^2} , \qquad
\wt\alpha_n^\# = \frac{\alpha_n^\#  - \delta_n^2\beta_n^\#}{1-\delta_n^2}, \quad 
\wt\beta_n^\# = \frac{\beta_n^\# - \delta_n^2\alpha_n^\#}{1-\delta_n^2}; 
\end{align}
note that since $\alpha_n>\beta_n$ we have that $\wt\alpha_n>\alpha_n$.  The eigenvectors 
$\ba\ci{I_\pm}$, $\bb\ci{I_\pm}$ 
are given by \eqref{e:ab_pm}, 
\begin{align}
\label{e:ab_pm 01}
\ba\ci{I_\pm} = (1+\delta_n^2)^{-1/2} (\ba\ci{I}  \pm \delta_n \bb\ci I), \qquad 
\bb\ci{I_\pm} = (1+\delta_n^2)^{-1/2} (\bb\ci I \mp 
\delta_n \ba\ci I), 
\end{align}

Then we apply the stretching operation independently in each of the principal axes
$\ba\ci{I_\pm}$, $\bb\ci{I_\pm}$; pairs $(\wt\alpha_n, \wt\beta^\#_n)$, $(\wt\alpha_n^\#, 
\wt\beta_n)$ will play the role of the pair $(x,y)$, and pairs $(\alpha_{n+1}, \beta^\#_{n+1})$, 
$(\alpha^\#_{n+1}, \beta_{n+1})$ will play the role of $(x_+,y_+)$. 
We put the results corresponding to $x_+$, $y_+$ to the right child of the
intervals $I_\pm$. Thus, for $J=I_+$ or $J=I_-$ we will have
\begin{align*}
\La W^{-1} \ra\ci{\!J_+} = \alpha_{n+1} \ba\ci{\!J} \ba\ci{\!J}^* + \beta_{n+1} \bb\ci{\!J}  
\bb\ci{\!J}^*, \qquad \La W \ra\ci{\!J_+} 
= \beta_{n+1}^\# \ba\ci{\!J} \ba\ci{\!J}^* + \alpha_{n+1}^\# \bb\ci{\!J}  \bb\ci{\!J}^*,
\end{align*}
where 
\begin{align}\label{e:wt alpha to alpha}
\alpha_{n+1} = r \wt\alpha_n, \quad \beta_{n+1}^\# = \wt\beta_n^\#/(s_n r), \qquad 
\alpha_{n+1}^\# 
= r \wt\alpha_n^\#, \quad \beta_{n+1}= \wt\beta_n/(s_nr),
\end{align}
and $s_n :=\wt\alpha_n \wt\beta_n^\# /\bcQ= \wt\alpha_n^\# \wt\beta_n/\bcQ$ is given by 
\eqref{e:s} (with $I$ replaced by $J$). 

As we mentioned above, $\wt\alpha_{n}>\alpha_n$, so
\begin{align}
\label{e:alpha_n growth}
\alpha_{n+1}>r\alpha_n, \qquad \alpha_{n+1}^\# > r \alpha_n^\#.
\end{align}
Since $\alpha_{n+1}\beta_{n+1}^\# = \alpha_{n+1}^\# \beta_{n+1} =\bcQ = \alpha_n\beta_n^\# = 
\alpha_n^\# \beta_n$, we conclude that 
\begin{align}\label{e:quot alphas betas}
{\alpha_{n+1}}/{\alpha_n} = {\beta_n^\#}/{\beta_{n+1}^\#}, \qquad 
{\alpha_{n+1}^\#}/{\alpha_n^\#} = {\beta_n}/{\beta_{n+1}}, 
\end{align}
so 
\begin{align}
\label{e:beta_n decrease}
\beta_{n+1} < \beta_n/r, \qquad \beta_{n+1}^\# < \beta_n^\#/r .
\end{align}
Recalling that $\delta_n = q (\beta_n/\alpha_n)^{1/2} = q (\beta_n^\#/\alpha_n^\#)^{1/2}$ we can 
see that 
\begin{align}\label{e:theta quotient}
\delta_{n+1} < \delta_n / r .
\end{align}

\begin{rem}\label{r:asympt qoutients}
One can easily see  that $\wt\alpha_n \underset{\delta_0<}{\sim}\alpha_n$, $\wt\alpha_n^\# 
\underset{\delta_0<}{\sim}\alpha_n^\#$, so
\begin{align*}
\alpha_{n+1} \underset{\delta_0<}{\sim} r \alpha_n, \qquad 
\alpha_{n+1}^\# \underset{\delta_0<}{\sim} r \alpha_n^\# , 
\end{align*}
so by \eqref{e:quot alphas betas}
\begin{align*}
\beta_{n+1} \underset{\delta_0<}{\sim} \beta_n/r, \qquad
\beta_{n+1}^\# \underset{\delta_0<}{\sim} \beta_n^\#/r. 
\end{align*}
Therefore, we can conclude that 
\begin{align*}
\delta_{n+1}  \underset{\delta_0<}{\sim} \delta_n/r. 
\end{align*}

\end{rem}

\begin{rem}\label{r:alpha/alpha}
One can see from the construction that 
\begin{align}\label{e:alpha/alpha}
\alpha_n^\# = c\alpha_n, \quad \beta_n = c\beta_n^\#, \qquad \text{where}\quad c=
\alpha_0^\#/\alpha_0. 
\end{align}
Indeed, since $\alpha_0 \beta_0^\# = \alpha_0^\# \beta_0 =\bcQ$,  this trivially holds for $n=0$. 
Assuming that \eqref{e:alpha/alpha} holds for $n$, we see from \eqref{e:wt alpha beta 02} that 
\begin{align*}
\wt\alpha_n^\# = c\wt\alpha_n, \quad \wt\beta_n = c\wt\beta_n^\#,
\end{align*}
and \eqref{e:wt alpha to alpha} implies that \eqref{e:alpha/alpha} holds for $n+1$. 

We do not need this fact to prove the main result, but it is used later in Section 
\ref{s:HaarShift} to prove the lower bound for the Haar Shift $\ShaUp$. 
\end{rem}

The intervals $I_{++}$ and $I_{-+}$, with $I\in\cS_n$ will give us the next generation $\cS_{n+1}$ 
of stopping intervals. Finally, let us denote $\cS =\cS(I^0):=\bigcup_{n\ge0}\cS_n$. Denote also  
$\cS_{<}=\cS_{<}(I^0):=\cS\setminus \{I^0\} = \bigcup_{n>0}\cS_n$. 

\subsubsection{Terminal intervals}\label{s:terminal int}
We do not care what happens on the intervals $I_{+-}$ and $I_{--}$, with $I\in\cS_n$, we 
just terminate our 
martingale 
there. 
Namely,  we know from the 
construction of \emph{stretching}, see Section \ref{s:stretch},  that for 
$J=I_\pm$ we have for the averages
\begin{align*}
\La W^{-1}\Ra\ci{\!J_-}^{-1} \le \La W\Ra\ci{\!J_-} \le 2 \La W^{-1}\Ra\ci{\!J_-}^{-1}. 
\end{align*}
Thus, by Lemma \ref{l:stop} there exists an $\bA_2$ weight $W$ on $J_-$, which is constant on the  
children of $J_-$ 
and has the prescribed averages, and we define 
$W$ on $J_-$ to be this weight. 

We will call such intervals $I_{+-}$, $I_{--}$, $I\in\cS$  \emph{terminal 
intervals}, and denote by $\cT$ the collection of all terminal intervals.

\begin{figure}[h]
 \includegraphics[width=15cm]{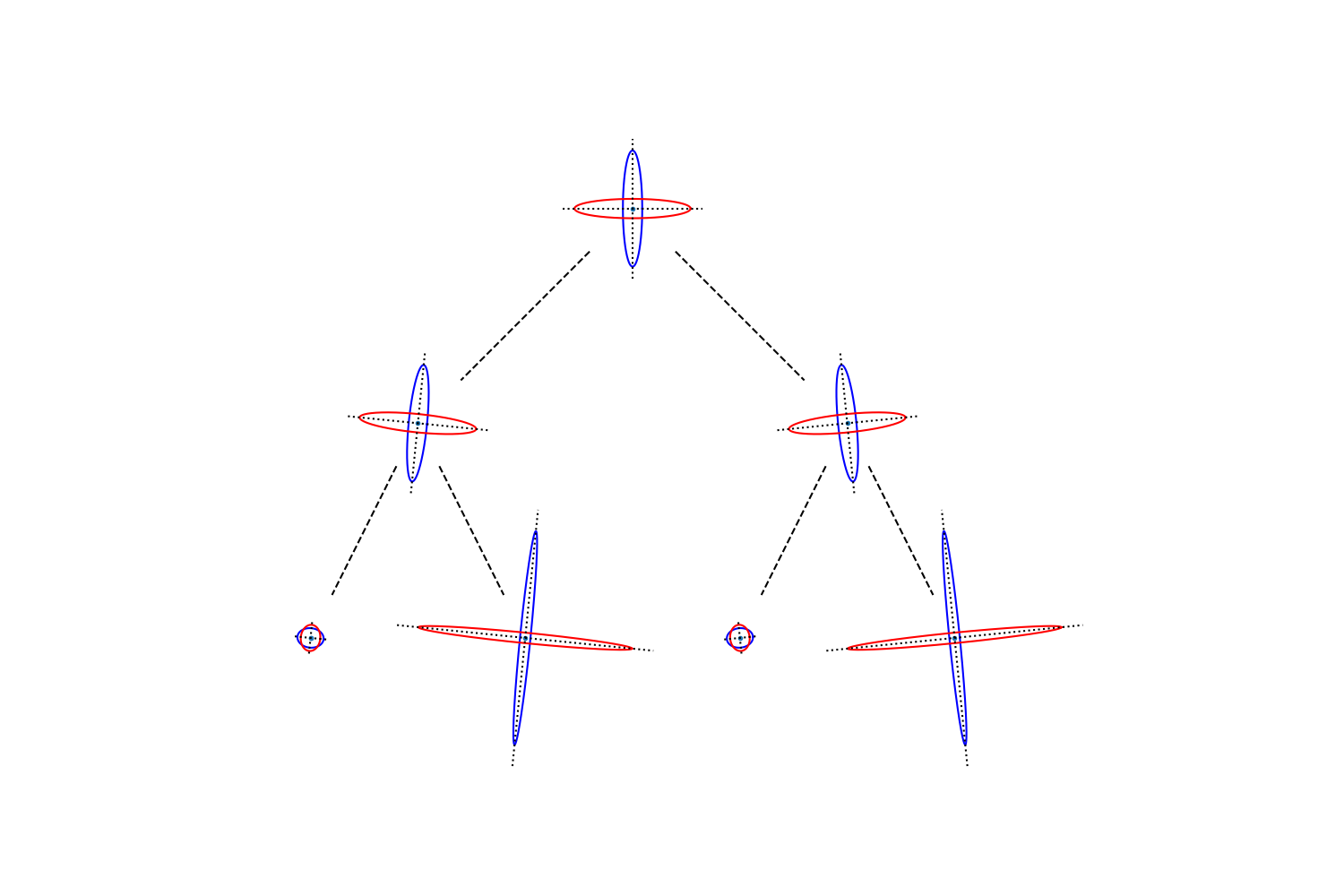}
 \caption{Dyadic tree: the picture shows two dyadic time steps: a rotation followed by stretching.}
\end{figure}

Thus we constructed martingales given by the averages $\bv\ci I$ and $\bw\ci I$. While it is 
possible (and not hard) to show 
that the martingales are uniformly integrable, and so are the averages of (matrix-valued) functions 
$V$ and  $W$ and that $V=W^{-1}$ (on terminal intervals we already have $V=W^{-1}$ by the 
construction), we chose an easier way and terminate the process, using again Lemma \ref{l:stop} 
after finitely many steps. We will write all the sums as infinite ones, but will keep finitely many 
terms that give us the desired lower bound. 

\section{Lower bounds for paraproducts} 
\label{s:est-para}

In this section we will prove the lower bound $c\bcQ^{3/2}$ for a special paraproduct, which will 
be a crucial part in the proof of the main result. 
In this section $\bA_2$ always means the \emph{dyadic} $\bA_2$ class $\bA_2\ut{dy}$, and  we will 
be using the notation $\bcQ:=[W]\ci{\bA_1}\ut{dy}$. 

Recall that $\cS_{<} =\cS_{<}(I^0) :=\bigcup_{n>0}\cS_n$, where $\cS_n$ are the generations of 
stopping intervals 
constructed above in Section \ref{s:constr}, and that $\cS=\cS(I^0) :=\bigcup_{n\ge 0}\cS_n $ is 
the collection of all stopping intervals.

Define the paraproduct $\Pi$ by 
\begin{align*}
\Pi f = \sum_{I\in\cS_{<}} \La f \Ra\ci{I}   \wh h\ci{I} = \sum_{I\in\cS_{<}} |I|^{1/2} \La f 
\Ra\ci{I}   
h\ci{I} , 
\end{align*}
where $\wh h \ci{I}$ is the $L^\infty$ normalized Haar function, $\wh h\ci I = \1\ci{I_+} - 
\1\ci{I_-}$, and $h\ci{I} = |I|^{-1/2} \wh h\ci{I}$ is the standard $L^2$-normalized Haar 
function.  

Since $\cS_{<}$ is a sparse family of intervals ($2$-Carleson, see the Definition \ref{df:la 
carleson} in Section \ref{s:ess part} below), the 
paraproduct is bounded in the unweighted 
$L^2$ space, or, equivalently its symbol $b:= \sum_{I\in\cS} \wh h\ci I$ is in the \emph{dyadic} 
BMO 
space. 

We claim that for the weight $W$ constructed above we have for all sufficiently large $\bcQ$ that  
$\|\Pi\|\ci{L^2(W)}\underset{\delta_0<}{\gtrsim}
([W]\ci{\bA_2}\ut{dy})^{3/2}$.

More precisely, the following lemma holds

\begin{lm}
\label{l:para lower bound}
For the weight $W$ constructed above and all sufficiently large $\bcQ$, we have for a function 
$f=\1\ci{I^0}W^{-1}\be$, where $\be = \ba\ci{I^0}$ or $\be= \ba\ci{I^0}+\bb\ci{I^0}$
\begin{align*}
\| \Pi f \|\ci{L^2(W)} \underset{\delta_0<}{\gtrsim}\bcQ^{3/2} \|f\|\ci{L^2(W)} \asymp 
\bcQ^{3/2}\alpha_0^{1/2}. 
\end{align*}
\end{lm}

\begin{rem}
For $\be = \ba\ci{I^0}$ we have $\|f\|\ci{L^2(W)} = \alpha_0^{1/2}$, and for $\be= 
\ba\ci{I^0}+\bb\ci{I^0}$ we only have the two sided estimates. 
\end{rem}

\begin{rem}
It looks more natural to consider the paraproduct  where the sum is taken over the whole family 
$\cS$, but for what follows later in Section \ref{s:Lower Haar shifts} it is more convenient to 
consider the sum only over $\cS_{<}$.  
However, since both paraproducts differ by one  term of the norm at most $\bcQ^{1/2}$ (with 
constant $1$), it really does not matter which paraproduct is used in Lemma \ref{l:para lower 
bound}: the bound for one immediately implies the bound for the other. 
\end{rem}

\subsection{Preliminaries and main idea}
For an interval $I\in\cS$ let us denote $\cS(I):=\{J\in\cS:J\subset I\}$,   
$\cS_{<}(I):=\cS(I)\setminus\{I\}$. Recall also that $\cS_{<} :=\cS_{<}(I^0)$.

We can write
\begin{align*}
\| \Pi (W^{-1}\1\ci{I^0}\be) \|\ci{L^2(W)}^2 & =  \sum_{I\in\cS_{<}}   \sum_{J \in\cS_{<}} 
\left( \La W^{-1} \Ra\ci{I}  \wh h\ci{I}  \be ,  \La W^{-1} \Ra\ci{\!J} \wh h\ci{\!J}  \be 
\right)\ci{L^2(W)}     \\
&=  \sum_{I\in\cS_{<}}   \left\|  \La W\Ra\ci{I}^{1/2}  \La W^{-1}\Ra\ci{I}  \be 
\right\|_{\R^2}^2    |I|
\\
& \quad  + 
2 \sum_{I\in\cS_{<}}  \sum_{J\in\cS_{<}(I) }   
\int _J \left( \La W^{-1} \Ra\ci{I} \be,  W(s) \wh h\ci{I} (s)\wh h\ci{\!J} 
(s)\La W^{-1} \Ra\ci{\!J}\be  \right) \ci{\R^2}  \dd s. 
\end{align*}
It is known that the first sum is always estimated above by $C\bcQ$, $\bcQ:=[W]\ci{\bA_2}\ut{dy}$.  
It will not be used here, but it explains why we do not bother with the first sum, and only will be 
estimating the second one below.
Formally, since the first sum is trivially non-negative, a lower bound for the second sum (the sum, 
not just its absolute value) gives us the same lower bound for $\| \Pi (W^{-1}\1\ci{I^0}\be) 
\|\ci{L^2(W)}^2$.

Denote by $\cS_+(I)$ (respectively $\cS_-(I)$) the intervals $J\in\cS(I)$ such that $J\subset I_+$ 
(respectively $J\subset 
I_-$). Since for $J\in\cS_\pm(I)$
\begin{align*}
|J|^{-1} \int_J  W \wh h \ci I \wh h\ci{\!J} \dd s = \pm \La W \wh h\ci{\!J} \Ra\ci{\!J}, 
\end{align*}
we can rewrite the second sum as 
\begin{align}\notag
 \sum_{I\in\cS_{<}} & \left( \sum_{J\in\cS_+(I) }    \left( \La W^{-1} \Ra\ci{I}  \be,  \La W \wh 
 h\ci{\!J}\Ra\ci{\!J}  \La W^{-1} \Ra\ci{\!J} 
 \be 
 \right)\ci{\R^2} |J|\right.\\ \label{e:S2}
 &\qquad\qquad \left. - \sum_{J\in\cS_-(I) }    \left( \La W^{-1} \Ra\ci{I} \be,  \La W \wh 
 h\ci{\!J} \Ra\ci{\!J}  \La W^{-1} \Ra\ci{\!J} \be 
 \right)\ci{\R^2}  |J| \right).
\end{align}

In the scalar case each term in \eqref{e:S2} is estimated  above by $\La w^{-1} \Ra\ci{I} \bcQ 
|J|$. For $I\in\cS_n$ summing over 
$J\in\cS_k(I) $, $k> n$ one gets the bound 
\begin{align}\label{e:S3}
\sum_{J\in\cS_k(I)} \La w^{-1} \Ra\ci{I} \La w \Ra\ci{\!J} \La w^{-1} \Ra\ci{\!J} |J| \le
 2^{-n-k}\bcQ\La w^{-1} \Ra\ci{I} = 2^{n-k} \bcQ\La w^{-1} \Ra\ci{I} |I|
\end{align}
(there are $2^{k-n}$ intervals $J\in\cS_k(I)$ and $|J|=2^{-2k}$).  
The summation over $k>n$ then gives us 
\begin{align}
\label{e:S4scalar}
\sum_{J\in\cS_{<}(I)} \La w^{-1} \Ra\ci{I} \La w \Ra\ci{\!J} \La w^{-1} \Ra\ci{\!J} |J| \le 
\sum_{k>n} 
 2^{n-k}\bcQ\La w^{-1} \Ra\ci{I} |I| \le \bcQ\La w^{-1} \Ra\ci{I} |I| .
\end{align}
It is well known that for a sparse sequence $\cS =\cS(I^0)$ and an $A_2$ weight $w$ with $A_2$
characteristic $\bcQ$ 
\begin{align*}
\sum_{I\in\cS(I^0)} \La w^{-1} \Ra\ci{I} |I| \le C\bcQ \La w^{-1} \Ra\ci{I^0} |I^0|,
\end{align*}
where the constant depends only on the sparseness characteristic of $\cS$. Thus, in the scalar case 
the sum \eqref{e:S2} is estimated above by $C\bcQ^2 \La w^{-1} \Ra\ci{I^0} |I^0|$, which gives, as 
one would expect, the linear in $\bcQ$ upper bound for the norm.  

In the matrix-valued case the $\bA_2$ condition does not imply any bounds on the norm of $\La 
W^{-1} \Ra\ci{\!J} \La W \Ra\ci{\!J}$. In our case we are able to estimate below \emph{most} of the 
terms 
in \eqref{e:S2}, $I\in\cS_n$, $J\in\cS_k(I)$ (with sign) by something like 
\begin{align*}
c \bcQ  r^{k-n}\left( \La W^{-1} \Ra\ci{I}  \be,  \be \right)\ci{\R^2} |J|, 
\end{align*}
i.e.~to gain the extra factor $r^{k-n}$ (recall that $r=2-1/\bcQ$). Thus, instead of summing the 
geometric series $\sum_1^\infty 2^{-k}=1$ as in \eqref{e:S4scalar}, we end up summing a bigger 
geometric series
\begin{align*}
\sum_{k\ge 1} (r/2)^k =2\bcQ-1, 
\end{align*}
which give us an extra factor $\bcQ$ in the sum. Since the sum gives us the square of the norm, 
this translates to the extra factor $\bcQ^{1/2}$ in the lower bound for the norm. Of course, there 
are a lot of technical details, that need to be carefully followed.

\subsection{Preliminary estimates}

 Let $I\in\cS_n$,  $J\in\cS(I)$, $J\in\cS_k$. Then it is 
easy to see that in the basis 
$\ba\ci{\!J}$, 
$\bb\ci{\!J}$ the 
operators $\ba\ci{\!J_\pm} \ba\ci{\!J_\pm}^*$ and $\bb\ci{\!J_\pm} \bb\ci{\!J_\pm}^*$ can be 
written as 
\begin{align}\label{e:aa* pm}
\ba\ci{\!J_\pm} \ba\ci{\!J_\pm}^* = \frac{1}{1+\delta_k^2} 
\begin{pmatrix}
1 & \pm \delta_k \\ \pm \delta_k & \delta_k^2
\end{pmatrix} ,  \qquad 
\bb\ci{\!J_\pm} \bb\ci{\!J_\pm}^* = \frac{1}{1+\delta_k^2} 
\begin{pmatrix}
\delta_k^2 & \mp \delta_k \\ \mp \delta_k & 1
\end{pmatrix} . 
\end{align}
Note that  $2 \La W \wh h\ci{\!J} \Ra\ci{\!J} = \La W \Ra\ci{\!J_+} - \La W \Ra\ci{\!J_-}$. Then 
using 
\eqref{e:aa* pm} and formulas \eqref{e:Ave W} for the averages we get 
\begin{align}\label{e:Av W h_J}
 \La W \wh h\ci{\!J} \Ra\ci{\!J}  =  - \delta_k \frac{\wt\alpha_k^\# - \wt 
 \beta_k^\#}{1+\delta_k^2} 
 \left(\ba\ci{\!J} \bb\ci{\!J} ^*   + 
 \bb\ci{\!J} 
 \ba\ci{\!J} ^* \right) . 
\end{align}
Then for $\be = \ba_0$ (or for $\be = \ba_0 + \bb_0$) we get
\begin{align*}
\La W^{-1} \Ra\ci{\!J} \be = (\be, \ba\ci{\!J})\ci{\R^2} \alpha_{k}\ba\ci{\!J} +
(\be, \bb\ci{\!J})\ci{\R^2} \beta_{k}\bb\ci{\!J},
\end{align*}
so
\begin{align}\label{e:sa-tb}
 \La W \wh h\ci{\!J} \Ra\ci{\!J}  \La W^{-1} \Ra\ci{\!J} \be & = s\ci{\!J}\ba\ci{\!J} - t\ci{\!J} 
 \bb\ci{\!J}, 
\end{align}
where  
\[
s\ci{\!J}=-\delta_k \frac{\wt \alpha_k^\# - \wt 
\beta_k^\#}{1+\delta_k^2}\beta_k(\be,\bb\ci{\!J})\ci{\R^2},
\qquad
t\ci{\!J}=\delta_k \frac{\wt \alpha_k^\# - \wt 
\beta_k^\#}{1+\delta_k^2}\alpha_k(\be,\ba\ci{\!J})\ci{\R^2}.
\]

We estimate 
\begin{align}\notag
|s\ci{\!J}| & \underset{\delta_0<}\lesssim 
\beta_k\alpha_k^\# \delta_k = \bcQ \delta_k,     
\intertext{and write}   \label{e:t_k}
t\ci{\!J} = (\be,  \ba\ci{\!J})\ci{\R^2}  t_k,  \qquad t_k &= \delta_k 
\frac{\wt\alpha_k^\# - \wt 
\beta_k^\#}{1+\delta_k^2}  \alpha_k \underset{\delta_0<}\sim \delta_k \alpha_k \alpha_k^\#  =  q 
\sqrt{\bcQ 
\alpha_k^\#\alpha_k} =: \wt t_k .  
\end{align}
 The sign ``$-$'' in front of $t\ci{\!J}$ is of 
essence here.  

We also have 
\begin{align*}
\La W^{-1}  \Ra\ci I \be & = \sigma\ci{I} \ba\ci I + \tau\ci{I} \bb\ci I  
\intertext{where}
\sigma\ci{I} =  \alpha_n ( \ba\ci{I}, \be)\ci{\R^2}  , \qquad &\tau\ci{I}= \beta_n ( \bb\ci{I}, 
\be)\ci{\R^2} . 
\end{align*}

Let us estimate $ \left( \La W^{-1} \Ra\ci{I} \be,  \La W \wh h\ci{\!J} \Ra\ci{\!J}  \La W^{-1} 
\Ra\ci{\!J} \be  
\right)\ci{\R^2}$.  We claim and we will show below 
that the main contribution comes from the inner product $-(\sigma\ci{I} \ba\ci{I}, 
t\ci{\!J} \bb\ci{\!J})\ci{\R^2}$.  
Writing  this inner product we get
\begin{align}\label{e:main term 01}
\sigma\ci{I} t\ci{\!J} (\ba\ci{I}, \bb\ci{\!J})\ci{\R^2} = \alpha_n t_k  
 (\ba\ci{I}, \be)\ci{\R^2} (\ba\ci{\!J}, \be)\ci{\R^2} (\ba\ci{I}, \bb\ci{\!J})\ci{\R^2} .
\end{align}

We will need the following simple lemma. 
\begin{lm}
\label{l:thetas}
Let $J\in\cS(I)$, $I\in\cS_n$, $J\in\cS_k(I)$, and let $\bcQ\ge 4$.  Then the angle 
$\theta\ci{I,J}$ 
between $\ba\ci{I}$ and 
$\ba\ci{\!J}$ can be estimated as 
\begin{align*}
|\theta\ci{I,J}| \le 3 \,\theta_n , \qquad \theta_n=\arctan \delta_n.
\end{align*}
\end{lm}
\begin{proof} The angle is trivially estimated by 
\begin{align*}
\sum_{j=n}^\infty \theta_j \le \sum_{j=n}^\infty \delta_j \le\frac{\delta_n}{1-1/r} = \delta_n 
\frac{2\bcQ-1}{\bcQ-1} \le (7/3)\delta_n . 
\end{align*}
By the construction $\delta_n\le q=0.1$, so using concavity of $\arctan$ on $\R_+$ we get 
\begin{align*}
\theta_n =\arctan\delta_n = \delta_n (\arctan \delta_n) /\delta_n \ge \delta_n (\arctan 0.1) /0.1 
\ge 0.99 \delta_n.  
\end{align*}
Thus the angle is bounded by $(7/2)\delta_n/0.99 \le 3 \theta_n$, 
which completes the estimate. 
\end{proof}

Assume that $J\in\cS(I)$, $I\in\cS_n$, $J\in\cS_k$, and that $\bcQ\ge 10$. We get from Lemma 
\ref{l:thetas} that 
\begin{align*}
(\ba\ci{\!J}, \be)\ci{\R^2} = (\ba\ci{I}, \be)\ci{\R^2} + O\ci{\delta_0<}(\delta_n). 
\end{align*}
Since $\bb\ci{\!J}\perp \ba\ci{\!J}$, Lemma \ref{l:thetas} implies that $|(\ba\ci{I}, 
\bb\ci{\!J})\ci{\R^2}| \le 3 \theta_n\le 3\delta_n$, so
\begin{align*}
\sigma\ci{I} t\ci{\!J} (\ba\ci{I}, \bb\ci{\!J})\ci{\R^2} & = \alpha_n t_k  
\left[ (\ba\ci{I}, \be)\ci{\R^2}^2 + O\ci{\delta_0<}(\delta_n)\right] (\ba\ci{I}, 
\bb\ci{\!J})\ci{\R^2} 
\\
& = \alpha_n t_k \left[ (\ba\ci{I}, \be)\ci{\R^2}^2 (\ba\ci{I}, \bb\ci{\!J})\ci{\R^2}  + 
O\ci{\delta_0<}(\delta_n^2)\right] .
\end{align*}

\begin{lm}
\label{l:main term}
Let $J\in\cS(I)$, $I\in\cS_n$, $J\in\cS_k$, and let $\bcQ\ge 10$. 
\begin{align}
\label{e:term}
\left( \La W^{-1} \Ra\ci{I} \be,  \La W \wh h\ci{\!J} \Ra\ci{\!J}  \La W^{-1} \Ra\ci{\!J} \be  
\right)\ci{\R^2}    
= - \alpha_n t_k \left[ (\ba\ci{I}, \be)\ci{\R^2}^2 (\ba\ci{I}, \bb\ci{\!J})\ci{\R^2}  + 
O\ci{\delta_0<}(\delta_n^2)\right]  
\end{align}
\tup{(}with implied constants not depending on $n$, $k$, $I$, $J$, $\bcQ$\tup{)}.
\end{lm}

\begin{proof}
We have already shown above that $\sigma\ci{I} t\ci{\!J} (\ba\ci{I}, \bb\ci{\!J})\ci{\R^2}$ gives 
us 
the desired form (the minus sign  
here is because of the minus in front of $t\ci{\!J}$ in \eqref{e:sa-tb}). As we claimed before, 
that is 
the main term, and we just need to show that the sum of the 
other three inner products are estimated by $C\alpha_n t_k \delta_n^2$. 

We have
\[
|\sigma\ci{I} s\ci{\!J}(\ba\ci{I}, \ba\ci{\!J})\ci{\R^2}| \le 
\bcQ\alpha_n 
\delta_k \underset{\delta_0<} 
\asymp \alpha_n t_k \delta_k^2 
\le \alpha_n t_k \delta_n^2 ; 
\]
here in the last inequality we used the fact that $\delta_k$ is decreasing,   and the equivalence 
follows from   the fact  that 
\[
t_k \delta_k^2 \underset{\delta_0<}{\sim} \delta_k  \alpha_k^\# \alpha_k \delta_k^2  =  \delta_k 
\alpha_k\alpha_k^\# q^2\beta_k /\alpha_k = 
q^2\delta_k\bcQ \underset{\delta_0<} \asymp \delta_k\bcQ. 
\]

On the other hand, trivially  
\[
| \tau\ci{I} t\ci{\!J} (\bb\ci{I}, \bb\ci{\!J})\ci{\R^2} | \le \beta_n t_k =q^{-2}
\alpha_n \delta_n^2 t_k.
\]
Finally $|s\ci{\!J}| \underset{\delta_0<}{\lesssim}   t_k$ (because $\beta_k < \alpha_k$), so we 
can 
estimate 
\[
\left| \tau\ci{I} s\ci{\!J} (\bb\ci{I}, \ba\ci{\!J})\ci{\R^2}  \right| 
\underset{\delta_0<}{\lesssim} \beta_n |s\ci{\!J} | \le \beta_n t_k =\alpha_n \delta_n^2 t_k. 
\]
Summing everything up, we get the desired estimate. 
\end{proof}

We claim that the term $\alpha_n t_k (\ba\ci{I}, \be)\ci{\R^2}^2 (\ba\ci{I}, 
\bb\ci{\!J})\ci{\R^2}$  is the main part of the term 
\begin{align*}
\left( \La W^{-1} \Ra\ci{I}  \be,  \La W \wh 
 h\ci{\!J}\Ra\ci{\!J}  \La W^{-1} \Ra\ci{\!J} \right)\ci{\R^2}, 
\end{align*}
 meaning that after summation (with 
 weights $|J|$) 
it gives us the lower bound $\underset{\delta_0<}{\gtrsim} \bcQ^3$, while the rest gives a smaller 
 contribution $\underset{\delta_0<}{\lesssim}\bcQ^2$.

\subsection{Final estimates}\label{finalestimates} Let us estimate the inner  sum in \eqref{e:S2}. 
Consider the sum 
over the set $\cS_+^k(I) = \cS_+(I)\cap \cS_k$, 
\begin{align*}
\sum_{J\in\cS_+^k(I) }  & \left( \La W^{-1} \Ra\ci{I}  \be,  \La W \wh 
 h\ci{\!J}\Ra\ci{\!J}  \La W^{-1} \Ra\ci{\!J} \be \right)\ci{\R^2} |J| \\
&= \sum_{J\in\cS_+^k(I) } - \alpha_n t_k \left[ (\ba\ci{I}, \be)\ci{\R^2}^2 (\ba\ci{I}, 
\bb\ci{\!J})\ci{\R^2}  + O\ci{\delta_0<}(\delta_n^2) \right] |J| .
\end{align*}
Let us first estimate the main part 
\begin{align}\label{e:main part est 01}
\sum_{J\in\cS_+^k(I) } - \alpha_n t_k  (\ba\ci{I}, \be)\ci{\R^2}^2 (\ba\ci{I}, 
\bb\ci{\!J})\ci{\R^2} |J| .
\end{align}

Note that only the  factor 
$(\ba\ci{I}, \bb\ci{\!J})\ci{\R^2}$ depends on $J$, so the total sum 
equals
\begin{align*}
- 2^{-n-1-k} \alpha_n t_k (\ba\ci{I}, \be)\ci{\R^2}^2 (\ba\ci{I}, \overline\bb\ci{I}^k)\ci{\R^2},
\end{align*}
where $\overline\bb\ci{I}^k$ is the average of the vectors $\bb\ci{\!J}$, $J\in\cS_+^k(I)$. There 
are $2^{k-n-1}$ intervals in $\cS_+^k(I)$ and the length of each interval $4^{-k}$, 
which explains the factor $2^{-n-1-k}$.

To find the average $\overline\bb\ci{I}^k$  let us recall how the vectors $\bb\ci{\!J}$, 
$J\in\cS_+^k(I)$ are constructed.  The set $\cS_+^{n+1}(I)$ consists of one interval $J=I_{++}$ and 
the vector $\bb\ci{\!J}$ is the rotation of the vector $\bb\ci{I}$ through the angle $\theta_n$. 
The set $\cS_+^{n+2}(I)$ consists of the  grandchildren $J_{++}$ and $J_{-+}$ of $J=I_{++}$, and 
the 
corresponding vectors are the rotations of $\bb\ci{\!J}$ through the angles $\pm\theta_{n+1}$ 
respectively, and so on. 

Therefore, it is clear that the average $\overline\bb\ci{I}^k$ must be a multiple of the vector 
$\bb\ci{I_{++}}$. By Lemma \ref{l:thetas}  the angle between $\bb\ci{I}$ and $\bb\ci{\!J}$ cannot 
be more than $3\theta_n<3\theta_0$, so for sufficiently small $\delta_0$ the vectors $\bb\ci{\!J}$, 
$J\in\cS_+^k(I)$  are almost parallel. Thus we can conclude that $\overline\bb\ci{I}^k = a_k 
\bb\ci{I_{++}}$, where $0.9\le a_k \le 1$. 
By the construction 
\begin{align*}
\bb\ci{I_{++}} = \bb\ci{I_+}= -(\sin \theta_n) \ba\ci{I} + (\cos \theta_n) \bb\ci{I}, 
\end{align*}
so
\begin{align}\label{e:pos angle -01}
- (\ba\ci{I}, \overline\bb\ci{I}^k)\ci{\R^2} & = a_k \sin \theta_n \underset{\delta_0<}{\ge} 0.8 
\delta_n
\end{align}
(note that the signs worked out). Therefore, for all sufficiently small $\delta_0$, we can estimate
\begin{align}\notag
\sum_{J\in\cS_+^k(I) } - \alpha_n t_k  (\ba\ci{I}, \be)\ci{\R^2}^2 & (\ba\ci{I}, 
\bb\ci{\!J})\ci{\R^2} |J| 
\\ \label{e:LB 02}
& \ge 0.8 \cdot 2^{-n-1-k} \alpha_n t_k (\ba\ci{I}, \be)\ci{\R^2}^2 
\delta_n 
\ge  0.7 \cdot 2^{-n-1-k} \alpha_n t_k \delta_n ;
\end{align}
here in the last inequality we used the fact that $(\ba\ci{I}, \be) \underset{\delta_0<}{\ge} 0.9 
(\ba\ci{I^0}, \be) = 0.9 $. 

\begin{rem}
One can see from the above reasoning that for half of the intervals  $J\in \cS_+^k(I)$
there holds 
$-(\ba\ci{I}, \bb\ci{\!J})\ci{\R^2}\ge  -(\ba\ci{I}, \bb\ci{I_{++}})\ci{\R^2}$. 

If all inner products $(\ba\ci{I}, \bb\ci{\!J})\ci{\R^2}$ had the same sign ($\le0$) for all $J\in 
\cS_+^k(I)$,  
one could use just this half of the intervals to get the estimate \eqref{e:LB 02} (with a smaller 
constant, which does not really matter).

Unfortunately, for $J\in\cS_{+}^k$ with sufficiently large $k$, not all inner products $(\ba\ci{I}, 
\bb\ci{\!J})\ci{\R^2}$ have the same sign. Namely, see \eqref{e:theta quotient} we can only 
guarantee that
$\theta_{j+1} \le \theta_j/r$ and we know that $1/r>1/2$.  Moreover we know, see Remark 
\ref{r:asympt qoutients}, that $\delta_{n+1}  \underset{\delta_0<}{\sim} \delta_n/r$, and since 
$\theta_n\underset{\delta_0<}{\sim}\delta$, we conclude that there exists $q'>0$, $1/r> q'>1/2$ 
such 
that for all sufficiently small $\delta_0$ 
\begin{align*}
\theta_{j+1}\ge q' \theta_j. 
\end{align*}
Therefore, since $q'>1/2$, we have for all sufficiently small $\delta_0$ that
\begin{align*}
\sum_{j>n} \theta_j >\theta_n.  
\end{align*}
Recall that a 
vector $\bb\ci{\!J}$ is obtained from $\bb\ci{I}$ by rotating it trough angle $\theta_n$, and then 
through angles $\pm\theta_j$, $n<j<k$. 
So, by picking sufficiently large $k$ and all the rotations except the first one to be through 
$-\theta_j$, $j>n$, we get that for all sufficiently small $\delta_0$ the corresponding vector 
$\bb\ci{\!J}$ is obtained from $\bb\ci{I}$ by a \emph{negative} rotation, so $-(\ba\ci{I}, 
\bb\ci{\!J})\ci{\R^2}<0$. 

Thus, a trivial estimate does not work, and we need a bit more involved reasoning; one possible 
reasoning using the average of the vectors $\bb\ci{\!J}$ 
was presented above. 
\end{rem}

Continuing with estimating \eqref{e:main part est 01},  let us estimate the non-essential parts. 
Trivially 
\begin{align*}
\left|\sum_{J\in\cS_+^k(I) } \alpha_n t_k  O\ci{\delta_0<}(\delta_n^2) |J| \right| \le C
2^{-n-k} \alpha_n t_k \delta_n^2   
\end{align*}
for all sufficiently small $\delta_0$.

Since $\delta_n\le \delta_0$, we have that  for all sufficiently small 
$\delta_0$
\begin{align*}
\left|\sum_{J\in\cS_+^k(I) } \alpha_n t_k  O\ci{\delta_0<}(\delta_n^2) |J| \right| \le 0.2 
\cdot 2^{-n-1-k} \alpha_n t_k \delta_n , 
\end{align*}
and we can conclude that 
\begin{align*}
\sum_{J\in\cS_+^k(I) }  & \left( \La W^{-1} \Ra\ci{I}  \be,  \La W \wh 
 h\ci{\!J}\Ra\ci{\!J}  \La W^{-1} \Ra\ci{\!J} \be \right)\ci{\R^2} |J| \ge
0.5\cdot 2^{-n-1-k} \alpha_n t_k \delta_n. 
\end{align*}

The sum over $\cS_-^k (I) = \cS_-(I) \cap \cS_k $ is estimated exactly the same. The difference is 
that in this case the average $\overline\bb\ci{I}^k=a_k \bb\ci{I_{-+}}$, and $\bb\ci{I_{-+}}$ is 
the vector  
$\bb\ci{I}$ rotated through the angle $-\theta_n$, so 
\begin{align}\label{e:pos angle}
(\ba\ci{I}, \overline\bb\ci{I}^k)\ci{\R^2} & = a_k \sin \theta_n \ge 0.7 \delta_n
\end{align}
(for sufficiently small $\delta_0$), i.e.~the inner product is positive, and not negative as 
before. But we also have ``$-$'' in front of the sum, so we have the same estimate
\begin{align*}
-\sum_{J\in\cS_-^k(I) }  & \left( \La W^{-1} \Ra\ci{I}  \be,  \La W \wh 
 h\ci{\!J}\Ra\ci{\!J}  \La W^{-1} \Ra\ci{\!J} \be \right)\ci{\R^2} |J| \ge
0.5\cdot 2^{-n-1-k} \alpha_n t_k \delta_n   \underset{\delta_0<}{\asymp}  2^{-n-k} \alpha_n \wt t_k 
\delta_n  . 
\end{align*}
Thus for the sum over $\cS_k(I):= \cS(I)\cap \cS_k$ we get 
\begin{align}\label{e:sum4}
\sum_{J\in\cS_k(I) }  & \left( \La W^{-1} \Ra\ci{I}  \be,  \La W \wh 
 h\ci{\!J}\Ra\ci{\!J}  \La W^{-1} \Ra\ci{\!J} \be \right)\ci{\R^2} |J|  
 \underset{\delta_0<}{\gtrsim}
 2^{-n-k} \alpha_n \wt t_k \delta_n. 
\end{align}
Using \eqref{e:alpha_n growth} and formula \eqref{e:t_k} for $\wt t_k$, we see that $\wt t_k \ge 
r^{k-n} \wt t_n$. 
Therefore  
\begin{align*}
\sum_{k> n} 2^{n-k} \wt t_k \ge \sum_{k> n} 2^{n-k} \wt t_n \wt r^{n-k} = \frac{r/2}{1-r/2} \wt t_n 
= (2\bcQ-1)\wt t_n,  
\ge 1.5 \bcQ \wt t_n,  
\end{align*}
so we get from \eqref{e:sum4} that 
\begin{align*}
2\sum_{J\in\cS_{<}(I) } \left( \La W^{-1} \Ra\ci{I}  \be,  \La W \wh 
 h\ci{\!J}\Ra\ci{\!J}  \La W^{-1} \Ra\ci{\!J} \be \right)\ci{\R^2} |J| 
&\underset{\delta_0<}{\gtrsim} 2^{-2n } \bcQ \alpha_n \wt t_n  \delta_n   \\
& 
\ge  2^{-2n } \bcQ \alpha_n q ({\bcQ \alpha_n^\#\alpha_n})^{1/2} \delta_n  \\
& =  2^{-2n} \bcQ q^2   \alpha_n   ({\bcQ \beta_n^\#\alpha_n})^{1/2} \\
& =  2^{-2n } \bcQ^2 q^2   \alpha_n       .  
\end{align*}

Now, summing the above estimate  over all $I\in\cS_{<}$ (i.e.~over all $n\ge 1$), and recalling 
that 
$\alpha_n\ge r^n \alpha_0$ and that the set $\cS_n$ contains exactly $2^n$ intervals, we get that 
\begin{align*}
\| \Pi (W^{-1}\1\ci{I^0}\be) \|\ci{L^2(W)}^2 \ge {q^2\bcQ^2\alpha_0}  \sum_{n\ge1} (r/2)^n 
=  q^2\bcQ^2(2\bcQ-1)\alpha_0 \ge q^2 \bcQ^3 \alpha_0;
\end{align*}
here we used the fact that $\sum_{n\ge 1} (r/2)^n = 2\bcQ-1$. Recalling that 
\begin{align*}
\|f\|\ci{L^2(W)}^2 = \|W^{-1} f \|\ci{L^2(W)}^2 \asymp \alpha_0
\end{align*}
completes the proof. \hfill \qed

\section{Lower bounds for Haar Shifts}    \label{s:Lower Haar shifts}
In this section we prove the lower bound for the special Haar Shift $\Hdy$,  which will be later 
used to get the lower bound for the Hilbert Transform, i.e.~to prove the main theorem (Theorem 
\ref{MainTheorem1}). 

This special operator $\Hdy$ is defined as a linear combination of classical Haar Shifts, so let us 
recall some definitions. 
The classical Haar Shift $\ShaUp$
is  defined by 
\begin{align*}
\ShaUp f = \sum_{I\in\cD} (f, h\ci I)\ci{L^2} \left[h\ci{I_+}-h\ci{I_-}\right].
\end{align*}
The importance of the Haar Shift comes from the fact that its average over dyadic grids yields a 
non-zero multiple of the Hilbert transform. 

In our construction we will use the \emph{odd} Haar Shift $\sh$, where the summation is taken over 
the set 
$\cD\ti{odd}= \bigcup_{k} \cD_{2k+1}$: 
\begin{align*}
\sh f = \sum_{I\in\cD\ti{odd}} (f, h\ci I)\ci{L^2} \left[h\ci{I_+}-h\ci{I_-}\right].
\end{align*}

We will also need the \emph{odd} modification $\sh_0$ of the operator $\ShaUp_0$  introduced in 
\cite{DPlower} 
\begin{align*}
 \ShaUp_0 f 
&= \sum_{I\in\cD} \left[ (f, h\ci{I_+})\ci{L^2} h\ci{I_-} - (f, 
h\ci{I_-})\ci{L^2} h\ci{I_+}\right], \\
\sh_0 f &= \sum_{I\in\cD\ti{odd}} \left[ (f, h\ci{I_+})\ci{L^2} h\ci{I_-} - (f, 
h\ci{I_-})\ci{L^2} h\ci{I_+}\right].
\end{align*}

On vector-valued functions these operators are defined by the same formulas, if we, sightly abusing 
notation, treat  $(f, h\ci I)\ci{L^2}$ as the vector 
\begin{align}\label{e:abuse IP}
(f, h\ci I)\ci{L^2} :=  |I|^{1/2} \La f h\ci{I} \Ra\ci{I}. 
\end{align}

After the remodeling described below in Section \ref{s:remodeling}, the lower bound for the Hilbert 
transform (for the remodeled weight) will be obtained  
from the lower bound for the dyadic operator
\[
\Hdy = c_1 (\sh - \sh^*) + c_2\shst; \qquad c_1
\ne 0.
\]
with $c_1$, $c_2$ given by the identities \eqref{e:c12} below.

We will use the lemma below with particular $c_1$, $c_2$ given by  \eqref{e:c12}, but it  
remains true for any $c_1\ne0$, $c_2$ (with implied constants depending on $c_1$, $c_2$). 

\begin{lm}\label{l:Lower bound for T}
For  the matrix weight $W$ we constructed in Section \ref{s:constr} and $f = 
W^{-1}\1\ci{I^0}\be $  there holds 
\[
\|\Hdy f\|\ci{L^2(W)}   \underset{\delta_0<}{\gtrsim} \bcQ^{3/2} 
\|f\|\ci{L^2(W)}.
\]
\end{lm}

\subsection{Proving Lemma \ref{l:Lower bound for T}} 
In this section we always denote $f:=\1\ci{I^0}W^{-1}\be$.

\subsubsection{Introducing sparse counterparts}
We first notice that for each stopping interval $I\in\cS$ its 
left sibling $L(I)$ is a terminal interval, so the weight $W$ is constant on children of $L(I)$, 
and so is the function $f$. Therefore, in the representation $\Hdy f$ a lot of terms are $0$, so we 
can replace $\Hdy$ in  $\Hdy f$ by its sparse counterpart $\Hdy_{\cS}$. 

Namely, for $I\in\cS_{<}$ let $\wh I$ be its dyadic parent. Then, defining sparse counterparts 
$\sh\ci{\cS}$, $\sh\ci{\cS}^0$ and $\Hdy_\cS$ of $\sh$, $\sh_0$ and $\Hdy$  respectively by 
\begin{align*}
\sh\ci{\cS} f  &:= \sum_{I\in\cS_{<}} (f, h\ci{\wh I})\ci{L^2} \left[h\ci{\wh I_+}-h\ci{\wh 
I_-}\right], 
\\  
\sh\ci{\cS}^0 f &:= \sum_{I\in\cS_{<}} \left[ (f, h\ci{\wh I_+})\ci{L^2} h\ci{\wh I_-} - (f, 
h\ci{\wh I_-})\ci{L^2} h\ci{\wh I_+}\right], \\
\Hdy_\cS & := c_1 (\sh\ci{\cS} - \sh\ci{\cS}^*) + c_2\sh\ci{\cS}^0, 
\end{align*}
we can see that for $f=W^{-1}\1\ci{I^0}\be$ we have
\begin{align*}
\Hdy f & = \Hdy_\cS f. 
\end{align*}

The operator $\sh\ci{\cS}$ can be represented as a linear combination of paraproducts
\begin{align*}
\sh\ci{\cS} = 2^{-1/2}\left(\Pi  + \Pi_1  - \Pi_2 - \Pi_3\right) 
\end{align*}
where $\Pi$ is the paraproduct introduced before in Section \ref{s:est-para},  and 
\begin{align*}
\Pi_1 g  = \sum_{I\in\cS_{<}} \La g\Ra\ci{L(I)} \wh h\ci{L(I)}, \qquad 
\Pi_2 g  = \sum_{I\in\cS_{<}} \La g\Ra\ci{L(I)} \wh h\ci{I} , \qquad 
\Pi_3 g  = \sum_{I\in\cS_{<}} \La g\Ra\ci{I} \wh h\ci{L(I)},
\end{align*}
where, recall, $L(I)$ is the left sibling of $I$.

\subsubsection{Essential part of \tp{$\Hdy f = \Hdy_\cS f$}{Hdy f}}\label{s:ess part}
The following lemma shows that $2^{-1/2}(\Pi-\Pi^*)f$ is the main part of $\Hdy f = \Hdy_\cS f$, 
and the rest can be ignored.   That means the estimate 
\begin{align*}
\|(\Pi -\Pi^*) f\|\ci{L^2(W)}\gtrsim \bcQ^{3/2}\|f\|\ci{L^2(W)}
\end{align*}
would prove Lemma \ref{l:Lower bound for T}.   

\begin{lm}\label{l:ignored parts}
The norms of the operators $\Pi_1$,  $\Pi_2$, $\Pi_3$, $\sh^0_{\cS}$ and of their adjoints  
are estimated above by $C\bcQ$ in $L^2(W)$ { \tup{(}and so in $L^2(W^{-1})$ as 
well\tup{)}}.  
\end{lm}

\begin{rem}
\label{r:S lower bound} The above lemma implies immediately that for $f= \1\ci{I^0} W^{-1}\be$ and 
sufficiently large $\bcQ$ one has the lower bound 
\begin{align*}
\| \sh f\|\ci{L^2(W)} \underset{\delta_0<}{\gtrsim} \bcQ^{3/2} \|  f\|\ci{L^2(W)}. 
\end{align*}

Indeed, for this particular $f$  
\begin{align*}
\sh f = \sh\ci{\cS} f = 2^{-1/2}\left(\Pi  + \Pi_1  - \Pi_2 - \Pi_3\right) f, 
\end{align*}
and we have the desired lower bound for $\|\Pi f\|\ci{L^2(W)}$. By Lemma \ref{l:ignored parts} the 
rest is small in comparison, and can be ignored. 
\end{rem}
To prove Lemma \ref{l:ignored parts} we need some results from \cite{NaPeSkTr}. 

\begin{df}\label{df:la carleson}
Recall that a collection $\cC$ of dyadic intervals is called \emph{$\lambda$-Carleson}, 
$\lambda>0$, if for any 
dyadic interval $J$
\begin{align*}
\sum_{I\in \cC(J)} |I| \le \lambda |J|;
\end{align*}
here we use the notation $\cC(J):= \{I\in \cC: I\subset J\}$.
\end{df} 
Note that the collection $\cS$ of 
stopping intervals is trivially $2$-Carleson, and therefore the collection $\wh\cS:=\{\wh I: 
I\in\cS\} $ is $4$-Carleson.

We now restate Lemma 5.4 from \cite{NaPeSkTr} in the way adapted to our purposes. For a collection 
$\cC$ of dyadic intervals define the square function  $S_1 = S_{1,\cC, W}$ as 
\begin{align}\label{e:S_1}
S_1 g (x):= \left(\sum_{I\in\cC}  \left\La  \|W(x)^{1/2} W^{-1/2} g \|\ci{\R^d} \right\Ra_{I}^2 
\1\ci{I} (x) \right)^{1/2}. 
\end{align}
Here $g$ is a function with values in $\R^d$ and $W$ is a $d\times d$ matrix weight.

\begin{lm}\label{l:SqFnEst}
Let $\cC$ be a $\lambda$-Carleson collection of dyadic intervals, and let $W\in \bA_2\ut{dy}$,   
$[W]\ci{\bA_2}\ut{dy} \le\bcQ$. Then for $S_1 = S\ci{1,\cC, W}$ there holds 
\begin{align*}
\| S_1 \|\ci{L^2\to L^2} \le C \lambda d \bcQ
\end{align*}
with an absolute constant $C$. 
\end{lm}
It is essentially Lemma 5.4 from \cite{NaPeSkTr} applied to the weights $W^{-1}$ and $W$ for  $W$ 
and $V$ in  \cite{NaPeSkTr}.  Lemma 5.4 from \cite{NaPeSkTr} was stated for the operator $\wt S_1$ 
that dominates $S_1$.  
The estimate there would translate to $C d \lambda 
\left([W^{-1}]\ci{\bA_2}\ut{dy}\right)^{1/2}\left([W]\ci{\bA_\infty}\ut{dy}\right)^{1/2}$, where 
$[W]\ci{\bA_\infty}\ut{dy}$ is the dyadic $\bA_\infty$ characteristic of the weight $W$. We are not 
giving the definition here, we only need to know that $[W]\ci{\bA_\infty}\ut{dy}\le 4 
[W]\ci{\bA_2}\ut{dy}$, see \cite[Remark 4.4]{NaPeSkTr}.  
Noticing that trivially 
$[W^{-1}]\ci{\bA_2}\ut{dy} = [W]\ci{\bA_2}\ut{dy}$ we get the estimate stated in Lemma  
\ref{l:SqFnEst}. 

\begin{proof}[Proof of Lemma \ref{l:ignored parts}]
Trivially, for an operator $T$ its norm in $L^2(W)$ is the norm of the operator $g\mapsto W^{1/2} T 
W^{-1/2}g $ (slightly abusing the notation  we will say ``of the operator  $W^{1/2} T 
W^{-1/2}$'') in the non-weighted space $L^2$.  Since the intervals $L(I)$, 
$I\in\cS_{<}$ are disjoint, 
we can estimate 
\begin{align*}
\|W^{1/2} \Pi_3 W^{-1/2}g \|\ci{L^2}^2 & =\sum_{I\in\cS_{<}} \int_{L(I)} \|W(x)^{1/2}    
\La W^{-1/2} g \Ra\ci{I}\|^2 \dd x \\
&\le \sum_{I\in\cS_{<}} \int_{L(I)} \La \|W(x)^{1/2}  W^{-1/2} g \| \Ra\ci{I}^2 \dd x \\
&\lesssim \sum_{I\in\cS_{<}} \int_{\wh I} \La \|W(x)^{1/2}    
W^{-1/2} g \|\Ra\ci{\wh I}^2 \dd x \\
&= \left\| S\ci{1,\wh \cS_{<}, W} g \right\|_{L^2}^2
\lesssim \bcQ^2 \| g \|\ci{L^2}^2 , 
\end{align*}
where $\wh\cS_{<}:=\{\wh I: I\in\cS_{<}\}$. 
The crucial moment here is that the intervals $L(I)$, $I\in\cS_{<}$ are disjoint, so we can 
represent 
the square of the norm as a sum of squares.

The estimate for $\Pi_1$ works absolutely the same way: the only difference is that in the first 
and the second lines one should take the averages over $L(I)$ instead ot $I$. But the final 
estimate will be the same.  

As for $\Pi_2$, we can get the estimate by duality. Namely, replacing $W$ by $W^{-1}$ and repeating 
for $\Pi_2^*$ the same calculations as above with obvious 
modifications we get that 
\begin{align*}
\|W^{-1/2} \Pi_2^* W^{1/2}\|\ci{L^2\to L^2}^2 \lesssim \|S\ci{1,\wh\cS_{<}, W^{-1}} \|\ci{L^2\to 
L^2}^2     
 \lesssim \bcQ^2  . 
\end{align*}
Taking the adjoint we get that 
\begin{align*}
\|W^{1/2} \Pi_2 W^{-1/2}\|\ci{L^2\to L^2}  \lesssim \bcQ .
\end{align*}

Finally, to estimate $\sh\ci{\cS}^0$ we represent it as 
\begin{align*}
\sh\ci{\cS}^0 = \sh\ti{L} - \sh\ti{L}^*, 
\end{align*}
where $\sh\ti{L}$ is the \emph{sparse left shift}, 
\begin{align*}
\sh\ti{L} g = \sum_{I\in \cS_{<}}  (g, h\ci{I})\ci{L^2} h\ci{L(I)} . 
\end{align*}
Since the intervals $L(I)$, $I\in\cS_{<}$ are disjoint, using the same estimates as above, we get 
that 
\begin{align*}
\|\sh\ti{L}\|\ci{L^2(W)\to L^2(W)} \lesssim \bcQ. 
\end{align*}
Replacing $W$ by $W^{-1}$ we get that
\begin{align*}
\|\sh\ti{L}\|\ci{L^2(W^{-1})\to L^2(W^{-1})} \lesssim \bcQ,  
\end{align*}
which by duality implies that $\|\sh\ti{L}^*\|\ci{L^2(W)\to L^2(W)} \lesssim \bcQ$. 
\end{proof}

As it was discussed above,  the lower bound for $\Hdy_\cS f$ follows from the lower bound for 
\[
 \| (\Pi 
-\Pi^*)  f \|\ci{L^2(W)}.
\] 

We have the correct lower bound for $\| \Pi  f \|\ci{L^2(W)}$. 
To see that $\| (\Pi-\Pi^{\star})  f \|\ci{L^2(W)}$ also has the correct lower bound, it suffices 
to show that for any sufficiently large $\bcQ$ 
\[
\left( \Pi  f, \Pi^*  f\right)_{L^2(W)} \le 0  
\]
for all sufficiently small $\delta_0$. 

We calculate 
\begin{align}\label{e:Pi f, Pi*f 01}
-\left( \Pi  f, \Pi^*  f\right)_{L^2(W)}=&
-\sum_{I\in \cS_{<}}\left( \La W \wh h\ci{I} \Ra\ci{I}  \La W^{-1} \wh h\ci{I} \Ra\ci{I} \be , \La 
W^{-1} \Ra\ci{I}\be\right)\ci{\R^2}|I|\\ \label{e:Pi f, Pi*f 02}
&+\sum_{I\in \cS_{<}}\sum_{J\in \cS_-(I)}\left( \La W^{-1} \Ra \ci{I}\be,\La W \Ra \ci{\!J} \La 
W^{-1} 
\wh h\ci{\!J} \Ra\ci{\!J}\be \right)\ci{\R^2}|J|\\ \label{e:Pi f, Pi*f 03}
&-\sum_{I\in \cS_{<}}\sum_{J\in \cS_+(I)}\left(\La W^{-1} \Ra \ci{I}\be,\La W \Ra \ci{\!J} \La 
W^{-1} 
\wh h\ci{\!J} \Ra\ci{\!J}\be \right)\ci{\R^2}|J|\\ \label{e:Pi f, Pi*f 04}
&-\sum_{I\in \cS_{<}}\sum_{J\in \cS_{<}(I)}\left(\La W^{-1} \wh h\ci{I} \Ra\ci{I}\be, \La W \wh 
h\ci{\!J} 
\Ra\ci{\!J} \La W^{-1} \Ra \ci{\!J} \be \right)\ci{\R^2}|J|.
\end{align}
Here we expanded $\Pi f$ indexing by $I$ and $\Pi^* f$ by $J$: the first term is the diagonal 
term, the middle two terms correspond to the sum over $J\subsetneq I$ and the last term corresponds 
to $I\subsetneq J$ but we swapped the roles of $I$ and $J$. To address the middle two terms, we 
compute for $J\in\cS_k$ (compare with \eqref{e:Av W h_J})
\begin{align}\label{e:Av W^-1 h_J}
\La W^{-1} \wh h\ci{\!J} \Ra\ci{\!J} = 
\delta_k \frac{\wt\alpha_k - \wt \beta_k}{1+\delta_k^2} \left(\ba\ci{\!J} \bb\ci{\!J} ^*   + 
 \bb\ci{\!J} 
 \ba\ci{\!J} ^* \right) ; 
\end{align}
it is essentially the same calculation as for \eqref{e:Av W h_J}, but with $\alpha_k^\#$, 
$\beta_k^\#$ replaced by $\alpha_k$ $\beta_k$ respectively, and the opposite sign because the 
rotation is through $\theta_k$ instead of $-\theta_k$ in \eqref{e:Av W h_J}. 
Therefore, for $I\in\cS_n$, $J\in\cS_k$, $k>n$,
\[
\La W \Ra \ci{\!J} \La W^{-1} \wh h\ci{\!J} \Ra\ci{\!J}\be = 
s\ci{\!J}'\ba\ci{\!J}+t\ci{\!J}'\bb\ci{\!J},
\]
{where}
\[
s\ci{\!J}'=\beta^\#_k \delta_k 
\frac{\wt\alpha_k-\wt\beta_k}{1+\delta^2_k}(\bb\ci{\!J},\be)\ci{\R^2}
\qquad  
t\ci{\!J}'=\alpha^\#_k \delta_k 
\frac{\wt\alpha_k-\wt\beta_k}{1+\delta^2_k}(\ba\ci{\!J},\be)\ci{\R^2}
=: t_k' (\ba\ci{\!J},\be)\ci{\R^2}. 
\]
We also can easily see that 
\[
\La W^{-1} \Ra \ci{I}\be =\sigma\ci{I}' \ba \ci{I}+\tau\ci{I}'\bb\ci{I}
\]
{where}
\[
\sigma\ci{I}'=\alpha_n (\ba\ci{I},\be)\ci{\R^2} , \qquad \tau\ci{I}'=\beta_n 
(\bb\ci{I},\be)\ci{\R^2}.
\]

We will need the following analogue of Lemma \ref{l:main term}. 
\begin{lm}
\label{l:main term 01}
Let $J\in\cS_<(I)$, $I\in\cS_n$, $J\in\cS_k$, and let $\bcQ\ge 10$. 
\begin{align}
\label{e:term 01}
\left( \La W^{-1} \Ra \ci{I}\be,\La W \Ra \ci{\!J} \La W^{-1} 
\wh h\ci{\!J} \Ra\ci{\!J}\be \right)\ci{\R^2}   
=  \alpha_n t_k' \left[ (\ba\ci{I}, \be)\ci{\R^2}^2 (\ba\ci{I}, \bb\ci{\!J})\ci{\R^2}  + 
O\ci{\delta_0<}(\delta_n^2)\right] 
\end{align}
\tup{(}with implied constants not depending on $n$, $k$, $I$, $J$, $\bcQ$\tup{)}.
\end{lm}
The proof is exactly the same as the proof of Lemma \ref{l:main term}, so we skip it.

Let us use Lemma \ref{l:main term 01} to check the signs of lines \eqref{e:Pi f, Pi*f 02}, 
\eqref{e:Pi f, Pi*f 03}.  Let us fix $k>n$, $I\in\cS_n$ and let us average the right hand side of 
\eqref{l:main term 01} over all $J\in \cS_-^k(I)=\cS_-(I)\cap \cS_k$. We will get 
\begin{align}\label{e:ave RHS}
 \alpha_n t_k' \left[  (\ba\ci{I}, \be)\ci{\R^2}^2 (\ba\ci{I}, \overline\bb\ci{I}^k)\ci{\R^2} + 
 O\ci{\delta_0<}(\delta_n^2) \right], 
\end{align}
where $\overline\bb\ci{I}^k$ is the average of the vectors $\bb\ci{\!J}$ over all $J\in\cS_-^k(I)$
(again with implied constants not depending on $n$, $k$, $I$, $\bcQ$). But we already computed the 
average $\overline\bb\ci{I}^k$ in Section \ref{finalestimates}, and we know that
\begin{align*}
(\ba\ci{I}, \overline\bb\ci{I}^k)\ci{\R^2} & = a_k \sin \theta_n \ge 0.7 \delta_n, 
\end{align*}
see \eqref{e:pos angle}. Therefore the expression in brackets in \eqref{e:ave RHS} is bounded below 
by $0.5\delta_n>0$ (for $k>n$). 
Summing up over all $n$, $k>n$, and recalling that $\alpha_n t_k'\ge 0$ we get that \eqref{e:Pi f, 
Pi*f 02} is non-negative.

The treatment of the line \eqref{e:Pi f, Pi*f 03} is exactly the same. Taking the average of the 
right 
hand side of 
\eqref{e:Pi f, Pi*f 03} over all $J\in \cS_+^k(I)=\cS_+(I)\cap \cS_k$, we again get 
\eqref{e:ave RHS}, but now $\overline\bb\ci{I}^k$ is the average of the vectors $\bb\ci{\!J}$ over 
all $J\in\cS_+^k(I)$. 

But again, this average was already computed in Section \ref{finalestimates}, 
and we know that in this case
\begin{align*}
-(\ba\ci{I}, \overline\bb\ci{I}^k)\ci{\R^2} & = a_k \sin \theta_n \ge 0.7 \delta_n, 
\end{align*}
see \eqref{e:pos angle -01}. The ``$-$'' sign in the above inequality is compensated by the minus 
in front of the sum in \eqref{e:Pi f, Pi*f 03}, so using the same reasoning as above we get that 
the contribution of \eqref{e:Pi f, Pi*f 03} is also positive.

Turning to the last sum \eqref{e:Pi f, Pi*f 04}, we first calculate using formula \eqref{e:Av W 
h_J} for $\La W \wh h\ci{\!J} \Ra\ci{\!J}$ that 
\[
\La W \wh h\ci{\!J} \Ra\ci{\!J} \La W^{-1} \Ra \ci{\!J} \be=
s\ci{\!J}''\ba\ci{J}+t\ci{J}''\bb\ci{J},
\]
{where} (for $I\in\cS_n$, $J\in\cS_k$) 
\[
s\ci{J}''=-\delta_k\beta_k\frac{\wt\alpha_k^\#-\wt\beta_k^\#}{1+\delta_k^2}(\bb\ci{J},\be)\ci{\R^2}
 \quad
t\ci{J}''=-\delta_k\alpha_k\frac{\wt\alpha_k^\#-\wt\beta_k^\#}{1+\delta_k^2}(\ba\ci{J},\be)\ci{\R^2}
 =: t_k (\ba\ci{J},\be)\ci{\R^2}
\]
($t_k$ here is exactly $t_k$ from \eqref{e:t_k}), 
{and}
\[
\La W^{-1} \wh h\ci{I} \Ra\ci{I}\be=
\sigma\ci{I}''\ba\ci{I}+ \tau\ci{I}''\bb\ci{I}, 
\]
\noindent{where}
\[
\sigma\ci{I}''=\delta_n \frac{\wt \alpha_n-\wt \beta_n}{1+\delta_n^2} (\bb\ci{I},\be)\ci{\R^2} 
\qquad
\tau\ci{I}''=\delta_n \frac{ \wt \alpha_n -\wt \beta_n}{1+\delta_n^2} (\ba\ci{I},\be)\ci{\R^2}.
\]

Therefore
\begin{align}\notag
&\left(\La W^{-1} \wh h\ci{I} \Ra\ci{I}\be, \La W \wh h\ci{\!J} 
\Ra\ci{\!J} \La W^{-1} \Ra \ci{\!J} \be \right)\ci{\R^2} \\ \label{e:last term}
& \qquad \qquad =
\sigma\ci{I}''s\ci{\!J}''(\ba\ci{I},\ba\ci{J})\ci{\R^2}  + 
\sigma\ci{I}''t\ci{\!J}''(\ba\ci{I}, \bb\ci{J})\ci{\R^2} 
+
\tau\ci{I}''t\ci{J}''(\bb\ci{I},\bb\ci{J})\ci{\R^2} + \tau\ci{I}''s\ci{J}'' 
(\bb\ci{I},\ba\ci{J})\ci{\R^2} .
\end{align}

We observe that $(\ba\ci{I},\ba\ci{J})\ci{\R^2}>0$ and $(\bb\ci{I},\bb\ci{J})\ci{\R^2}>0$ and 
recall that $(\ba\ci{J},\be)\ci{\R^2} (\ba\ci{I},\be)\ci{\R^2}>0$ and $(\bb\ci{J},\be)\ci{\R^2} 
(\bb\ci{I},\be)\ci{\R^2}>0$. Then, $\wt\alpha_k-\wt\beta_k>0$ and $\wt\alpha_k^\#-\wt\beta_k^\#>0$ 
gives
\[
\sigma\ci{I}''s\ci{J}''(\ba\ci{I},\ba\ci{J})\ci{\R^2}<0, \qquad
\tau\ci{I}''t\ci{J}''(\bb\ci{I},\bb\ci{J})\ci{\R^2}<0. 
\]

We roughly estimate the remaining two terms in \eqref{e:last term}, noticing that 
\begin{align*}
0\le \sigma\ci{I}'' \underset{\delta_0<}{\lesssim} \tau\ci{I}'', \qquad 
0\le - s\ci{\!J}'' \underset{\delta_0<}{\lesssim} - t\ci{\!J}'' . 
\end{align*}
Recalling that 
$|(\ba\ci{I},\bb\ci{J})\ci{\R^2}|, \, |(\bb\ci{I},\ba\ci{J})\ci{\R^2}|\le 3\delta_n 
\underset{\delta_0<}{\lesssim} \delta_0 (\bb\ci{I},\bb\ci{J})\ci{\R^2} $, we then conclude that 
\begin{align*}
| \sigma\ci{I}''t\ci{\!J}''(\ba\ci{I}, \bb\ci{J})\ci{\R^2} |, \, 
| \tau\ci{I}''s\ci{J}'' (\bb\ci{I},\ba\ci{J})\ci{\R^2} |  \underset{\delta_0<}{\lesssim}
- \delta_0 \tau\ci{I}''t\ci{J}''(\bb\ci{I},\bb\ci{J})\ci{\R^2}, 
\end{align*}
implying that for sufficiently small $\delta_0$ the sum \eqref{e:last term} is negative. 
The minus sign in front of the sum \eqref{e:Pi f, Pi*f 04} gives us a positive contribution there.

For the diagonal terms \eqref{e:Pi f, Pi*f 01}, let us figure out the sign of $\left( \La W \wh
h\ci{I} \Ra\ci{I} \La W^{-1} \wh h\ci{I} \Ra\ci{I} \be , \La W^{-1} \Ra\ci{I}\be\right)\ci{\R^2}$.
The operators $\La W \wh h\ci{I} \Ra\ci{I}$ and $\La W^{-1} \wh h\ci{I} \Ra\ci{I}$ are given by
\eqref{e:Av W h_J} and \eqref{e:Av W^-1 h_J} respectively, so their product is a negative multiple
of the identity $\bI$. Therefore the sign of $\left( \La W \wh
h\ci{I} \Ra\ci{I} \La W^{-1} \wh h\ci{I} \Ra\ci{I} \be , \La W^{-1} \Ra\ci{I}\be\right)\ci{\R^2}$ 
is the same as the sign of 
\begin{align*}
\left( \be , \La W^{-1} \Ra\ci{I}\be\right)\ci{\R^2}  < 0 , 
\end{align*}
{which} implies a positive contribution of the diagonal term. 
\hfill\qed

\subsection{The lower bound  for the classical Haar Shift \tp{$\ShaUp$}{Sha}}\label{s:HaarShift}
Even though we will not need the lower estimate for the full Haar Shift for our main result, it is 
of independent interest and it is not hard to complete the proof.

\begin{thm}\label{MainTheorem2}
For the weight $W$  constructed above, and for the function $f=\1\ci{I^0}W^{-1}\be$ 
\tup{(}$\be=\ba\ci{I^0}$, or $\be = \ba\ci{I^0} + \bb\ci{I^0}$\tup{)}
\[
\|  \ShaUp f \|\ci{L^2(W)} \ge c \bcQ^{3/2} \|f\|\ci{L^2(W)}.
\]
\end{thm}

\begin{proof}
It was proved above, see Remark \ref{r:S lower bound}, that for the odd Haar Shift $\sh$ and 
$f=\1\ci{I^0}W^{-1}\be$ we have the desired estimate $\| \sh f\|\ci{L^2(W)}\gtrsim \bcQ^{3/2}$. 
Therefore, to prove Theorem \ref{MainTheorem2} 
it suffices to show that the contribution of the  \emph{even} Haar Shift $\sh'$ (where summation is 
over the intervals in 
$\cD\ti{even}=(\cD_{2k})_{k\ge 0}$),
\begin{align*}
 \sh' f = \sum_{I\in\cD\ti{even}} (f, h\ci I)\ci{L^2} \left[h\ci{I_+}-h\ci{I_-}\right],
\end{align*}
is small in comparison.   
Below we will show that for
the weight $W$ and
our test function $f=\1\ci{I^0}W^{-1}\be$ 
\begin{align}\label{e:S'f}
\| \sh' f \|\ci{L^2(W)}\lesssim \bcQ \|f\|\ci{L^2(W)}, 
\end{align}
which will prove the theorem.

The weight $W^{-1}$, and so the test function $f=\1\ci{I^0}W^{-1}\be$, are constant on children of 
terminal intervals, see Section \ref{s:terminal int}, so the descendants of terminal intervals do 
not contribute to $\sh' f$. 
The terminal intervals do contribute to $\sh'f$, so expanding $ \| \sh' f \|\ci{L^2(W)}^2$ we get
\begin{align}\notag
&  \| \sh' f \|\ci{L^2(W)}^2 \\ \label{e:S'f 01}
&=\sum_{I\in \cS\cup\cT}\sum_{J\in \cS \cup \cT} 
   \left( 
   \left[\int W(\wh h\ci{I_+}-\wh h\ci{I_-})(\wh h\ci{J_+}-\wh h\ci{J_-})\right]
   \La W^{-1} \wh h\ci{I}\Ra\ci{I}\be,  
   \La W^{-1} \wh h\ci{J}\Ra\ci{J}\be
   \right)\ci{\R^2}, 
\end{align}
where, recall, 
$\cT$ is the collection of all 
terminal intervals. 

The estimate for the sum of the diagonal terms $I=J$ can be easily reduced to Lemma 
\ref{l:SqFnEst}. Namely, we can estimate this sum as 
\begin{align*}
\sum_{I\in\cS\cup \cT} \left( \La W \Ra\ci{I} \La W^{-1}\wh h \be \Ra\ci{I}, \La W^{-1}\wh h \be 
\Ra\ci{I}  \right)\ci{\R^2} |I| 
& \le \int_\R \sum_{I\in\cS\cup \cT} \left\La \| W(x)^{1/2} f \|\right\Ra_{I}^2 \dd x \\
& = \int_\R \sum_{I\in\cS\cup \cT} \left\La \| W(x)^{1/2} W^{-1/2} W^{1/2}f 
\|\ci{\R^2}\right\Ra_{I}^2 \dd x \\
& = \| S_1 W^{1/2} f \|\ci{L^2}^2, 
\end{align*}
where $S_1 = S\ci{1,\cS\cup\cT, W}$ is the square function \eqref{e:S_1} with $\cC= 
\cS\cup\cT$.

As we discussed above in Section \ref{s:ess part} the collection $\cS$ is $2$-Carleson, therefore
the collection $\cS\cup \cT$ is $4$-Carleson. Therefore, by Lemma \ref{l:SqFnEst} 
\begin{align*}
\| S_1 W^{1/2} f \|\ci{L^2}^2 \lesssim \bcQ^2 \|W^{1/2} f\|\ci{L^2}^2 = \bcQ^2 \|f\|\ci{L^2(W)}^2 .
\end{align*}

To estimate the rest of the sum \eqref{e:S'f 01}, let us thus consider one of the (symmetric) off 
diagonal terms 
with $J\subsetneq I$. If $I\in \cT$ 
then there is no contribution. If $J\in \cT$ then $\int W(\wh h\ci{I_+}-\wh h\ci{I_-})(\wh 
h\ci{J_+}-\wh h\ci{J_-}) =\pm \int W(\wh h\ci{J_+}-\wh h\ci{J_-})= 0$ since $W$ is constant on 
$J_\pm$, so again there is no contribution. Thus we only need to consider the case $I, J\in\cS$. 
Let us assume that $I\in\cS_n$, $J\in\cS_k$, $k>n$.  

We calculate 
\[
\La W(\wh h \ci{J_+}-\wh h \ci{J_-}) \Ra\ci{J}=\frac{1}{4}
\left[ \La W \Ra \ci{J_{++}}-\La W \Ra \ci{J_{+-}} -\La W \Ra \ci{J_{-+}}+\La W \Ra \ci{J_{--}}
\right].
\]
Let us regroup the terms into differences corresponding to ``neighboring'' stopping and  terminal 
intervals.  For the stopping intervals $J_{-+}$ and $J_{++}$ we get
\begin{align}\label{e:diff stop}
\frac{1}{4}
\left[ \La W \Ra \ci{J_{++}}-\La W \Ra 
\ci{J_{-+}}\right]=\frac{1}{2}\cdot\frac{\delta_k}{1+\delta_k^2}(\beta_{k+1}^\#-\alpha_{k+1}^\#)
(\ba\ci{J}\bb\ci{J}^*+\bb\ci{J}\ba\ci{J}^*);
\end{align}
we use the same calculation as for \eqref{e:Av W h_J} with $\wt\alpha_k^\#$, $\wt\beta_k^\#$ 
replaced by 
their ``stretched'' counterparts $\alpha_{k+1}^\# = r \wt\alpha_k^\# $, $\beta_{k+1}^\# =\wt 
\beta_{k}^\#/(s_k r)$.  

As for the terminal 
intervals we get using the formulas \eqref{e:stretch -parts} that
\begin{align}\label{e:diff terminal}
\frac{1}{4}
\left[ \La W \Ra \ci{J_{+-}}-\La W \Ra \ci{J_{--}}\right]
=\frac{1}{2}\cdot\frac{\delta_k}{1+\delta_k^2}\left( \left(2-\frac{1}{s_k r}\right)\wt 
\beta_{k}^\#-(2-r)\wt 
\alpha_{k}^\#\right)
(\ba\ci{\!J}\bb\ci{\!J}^*+\bb\ci{\!J}\ba\ci{\!J}^*).
\end{align}
The exact coefficients are not essential here; what is essential is that in both cases we have 
scalar multiples of the operator $\ba\ci{\!J}\bb\ci{\!J}^*+\bb\ci{\!J}\ba\ci{\!J}^*$, and the 
absolute values of the coefficients are estimated as $\lesssim \delta_k\alpha_{k}^\#$.%
\footnote{Of course, to get to this conclusion we do not really need the tedious calculations 
above: everything can be obtained just by simple analysis of the corresponding terms.}
The same, of course, holds for the difference, so taking into account that $\| 
\ba\ci{\!J}\bb\ci{\!J}^*+\bb\ci{\!J}\ba\ci{\!J}^*\|=1$ we get a trivial rough estimate
\begin{align*}
\left\| \La W(\wh h \ci{J_+}-\wh h \ci{J_-}) \Ra\ci{J} \right\|
\lesssim 
\delta_k\alpha_{k}^\# .
\end{align*}

Recalling  formula \eqref{e:Av W^-1 h_J} for $\La W^{-1} \wh h \ci{\!J}\Ra \ci{\!J}$ 
we get 
\begin{align*}
\left\|  \La W^{-1} \wh h \ci{\!J}\Ra \ci{\!J} \right\| \lesssim \delta_k \alpha_k, \qquad
\left\|  \La W^{-1} \wh h \ci{I}\Ra \ci{I} \right\| \lesssim \delta_n \alpha_n. 
\end{align*}
So a rough estimate gives us 
\begin{align*}
A\ci{I,J}:=\left|\left(
 \La W(\wh h \ci{J_+}-\wh h \ci{J_-}) \Ra\ci{J}\La W^{-1} \wh h \ci{J}\Ra \ci{J}\be , \La W^{-1} 
 \wh h \ci{I}\Ra \ci{I}\be
\right)\ci{\R^2} \right|  
&\lesssim \delta_k^2 \alpha_k \alpha_k^\# \delta_n\alpha_n \\
&= q^2 \bcQ \delta_n\alpha_n \\ 
&= q^3\bcQ \beta_n^{1/2}\alpha_n^{1/2} \\
&= q^3\bcQ^{3/2}(\alpha_n/\alpha_n^\#)^{1/2}. 
\end{align*}
Using \eqref{e:alpha/alpha} and identities $\bcQ = \alpha_n^\# \beta_n$, $\delta_n = 
q(\beta_n/\alpha_n)^{1/2}$ with $n=0$ and the fact that $\alpha_n/\alpha^\#_n = 
\alpha_0/\alpha^\#_0$, see Remark 
\ref{r:alpha/alpha}, we get 
\begin{align*}
q \bcQ^{1/2}(\alpha_n/\alpha_n^\#)^{1/2} = q \bcQ^{1/2} (\alpha_0/\alpha_0^\#)^{1/2} 
=\delta_0\alpha_0 ,   
\end{align*}
so
\begin{align*}
A\ci{I,J} \lesssim q^2 \bcQ \delta_0\alpha_0.  
\end{align*}
Multiplying by $|J|$ and summing over all intervals we get 
\begin{align*}
\sum_{I\in\cS}\sum_{J\in\cS, J\subsetneq I} A\ci{I,J}|J| &\lesssim 
\sum_{I\in\cS} q^2 \bcQ \delta_0\alpha_0 |I|  = 2 q^2 \bcQ \delta_0\alpha_0  \asymp q^2 \bcQ 
\delta_0 \| f\|\ci{L^2(W)}^2;
\end{align*}
here we used the fact that for a stopping interval $I$ the total length of the stopping intervals 
$J\subset I$ of the next generation is $|I|/2$.  

By picking sufficiently small $\delta_0$ we can make the sum of off diagonal terms as small as we 
want, so it can be ignored. Thus \eqref{e:S'f} holds, and Theorem \ref{MainTheorem2} is proved. 
\end{proof}

\section{Towards the lower bound for the Hilbert Transform: remodeling}  
\label{s:remodeling}
Above we constructed a dyadic $\bA_2$ weight $W=W\ci{\!\bcQ}$ on $I^0=[0,1)$, 
$[W]\ci{\bA_2}\ut{dy}=\bcQ$ such that the dyadic operators, namely the paraproduct $\Pi$, the 
\emph{odd} Haar Shift $\sh$ and  the linear combination $\Hdy=c_1(\sh-\sh^*) + c_2 \shst$ had lower 
bounds $c\bcQ^{3/2}$.

In this section we will use the weight $W$ to construct a weight $\wt W$ such that $[\wt 
W]\ci{\bA_2}\asymp [W]\ci{\bA_2}\ut{dy}=\bcQ$ to prove later in Section \ref{s:Lower bound Hilbert} 
that the norm of  the Hilbert 
Transform $\HT$ in  $L^2(\wt W)$ is bounded below by  $c\bcQ^{3/2}$.  

Since from Lemma \ref{l:Lower bound for T} we have a lower bound for the operator $\Hdy$,  
$\Hdy=c_1(\sh-\sh^*) + c_2 \shst$, there exist functions 
$f\in L^2(W)$, $g\in L^2(W)$, $\|f\|\ci{L^2(W)} = \|g\|\ci{L^2(W)} =1$ such that 
\begin{align}\label{e:biest01}
\left(\Hdy   f, g\right)\ci{L^2(W)} \gtrsim \bcQ^{3/2}. 
\end{align}
In our construction the functions $f$, $g$ were the normalized in $L^2(W)$
 functions 
$\1\ci{I^0}W^{-1}\be$ and 
$\Hdy f =\Hdy W^{-1}\1\ci{I^0}\be$ respectively, but in what follows we do not care what the 
functions $f$ and 
$g$ are. Denoting 
$\bg:= Wg$, and using $\bff=f$ for consistency sake, we can see that $\|\bff\|\ci{L^2(W)} = 
\|\bg\|\ci{L^2(W^{-1})} =1$, and \eqref{e:biest01} can be rewritten as 
\begin{align}\label{e:biest02}
\left(\Hdy \bff, \bg\right)\ci{L^2} = \left(\Hdy f, g\right)\ci{L^2(W)} \gtrsim \bcQ^{3/2}.
\end{align}
And using the construction presented below, we will show that \eqref{e:biest02} implies the lower 
bound for the Hilbert Transform $\cH$ in $L^2(\wt W)$   for a weight $\wt W$ remodeled from $W$.

Throughout this section we denote for intervals $I$, $J$ by $\psi\ci{I, J}$ the unique 
orientation preserving affine bijection from $J$ to $I$.

\subsection{(Iterated) periodization} 
\label{s:IterPer}

Given a function $f$ on an interval $I$ define its \emph{periodization} $\per_{I}^N f$ with 
``period'' $N$ as the unique function on $I$ with period $2^{-N}|I|$, consisting of repeated $2^N$ 
``shrunken copies'' of the function $f$, i.e.\ $\per_{I}^N f = f \circ \psi\ci{I}^N$, where the 
restriction of $\psi\ci{I}^N$ onto each  $J\in\ch^N(I)$ coincides with $\psi\ci{I, J}$. Note that 
$\psi\ci{I}^N$ is a measure preserving transformation of the interval  $I$.

Given a sequence $\vec N = (N_1, N_2, N_3, \ldots)$ of frequencies  $N_k$ and a function $f$ on 
$I^0$ let us define the \emph{iterated periodization} $\per^{\vec N} f$ by induction as follows.

\begin{enumerate}
\item We start with the children $I^0_\pm$ of $I^0$ as the  \emph{starting intervals} of the first 
order, $\Start_1$ and define $f_0:=f$. 

\item Knowing the function $f_{k-1}$ define the function $f_k$ by replacing the function 
$f_{k-1}|_I$ 
on every starting interval $I$ of order $k$    by its periodization $\per_{I}^{N_k} 
(f_{k-1}|_{I})$. We will call the descendants of order $N_k$ of the starting intervals of order $k$ 
the \emph{stopping intervals} of order $k$. 
\item Define the starting intervals of order $k+1$ as the grandchildren ($4$ intervals in 
$\ch^2(I)$) of the stopping intervals of order $k$, and return to step \cond2. 
\end{enumerate}

To avoid technical details with convergence we will only consider functions $f\in L^1(I^0)$ with 
finite Haar expansion; in this case the process stabilizes after finitely many steps. 

Since a periodization is obtained by composing a function with a measure  preserving 
transformation, the iterated periodization $\per^{\vec N} f$ is also given by $\per^{\vec N} f = 
f\circ \Psi^{\vec N}$, where $\Psi^{\vec N}$ is a measure preserving transformation.  

\subsubsection{An alternative description of iterated periodization}\label{s:alt per}
Analyzing the process described above, we can see that the iterated periodization associates with a 
starting interval $I$ of order $k$ 
a unique interval $J=\Sinv(I)\in \cD_{2k-1}$, such that $f_{k-1}|_I=f\circ \psi\ci{J,I}$. Then, on 
$I$ we perform the periodization of the 
function $f\circ \psi\ci{J,I}$. We will also associate the same interval $J=\Sinv(I)\in\cD_{2k-1}$ 
with all 
stopping subintervals $I'\subset I$ of order $k$, 
so $\Sinv(I')=\Sinv(I)$ for such intervals. 

The function $\Sinv$  maps each starting (and stopping) interval $I$ of order $k$ to a 
unique interval $\Sinv(I)\in\cD_{2k-1}$. Note that the function $\Sinv$ depends (only) on the 
sequence $\vec N$ of indices, so $\Sinv=\Sinv^{\vec N}$; however, to simplify our typography, we 
will 
often omit index $\vec N$. 

Now let us see what happens when we apply the periodization to the martingale difference 
decomposition of $f$, 
\begin{align}\label{e:MartDiffDecomp2}
f = \bE_1 f + \sum_{J\in\cD\ti{odd}} \Delta\ci{J}^2 f . 
\end{align}
Since the periodization does not change constant functions, we can see that $\per^{\vec N} \bE_1
f=\bE_1 f$, and   the periodization  $\per^{\vec N}\Delta\ci{J}^2 f$ with $J\in \cD_{2k-1}$ depends
only on indices $N_j$, $j\le k$. It is also easy to see that the starting intervals of order $k$,
and action of the function $\Sinv$ on these intervals depends only on indices $N_j$, $j<k$. Thus 
for $J\in\cD_{2k-1}$ we can write
\begin{align}\label{e:P Delta f}
\per^{\vec N}\Delta\ci{J}^2 f = \sum_{\substack{I\in \Start_k:\\ \cF(I) = J}} \per\ci{I}^{N_k}
\left( \left(\Delta\ci{J}^2 f\right) \circ \psi\ci{J,I}  \right)   .
\end{align}

Note that the position of the starting intervals $I\in\Start_k$, $\Sinv(I) = J$ depends only on 
indices $N_j$, $j<k$. 

Also, since the iterated periodization is given by a composition with a measure preserving 
transformation, we can see that
\begin{align}\label{e:sum |I|}
\sum_{\substack{I\in \Start:\\ \cF(I) = J}} |I| = |J|. 
\end{align}
\begin{figure}[ht]
    \vspace*{-0.5cm}\hspace*{-2cm}\includegraphics[width=18cm]{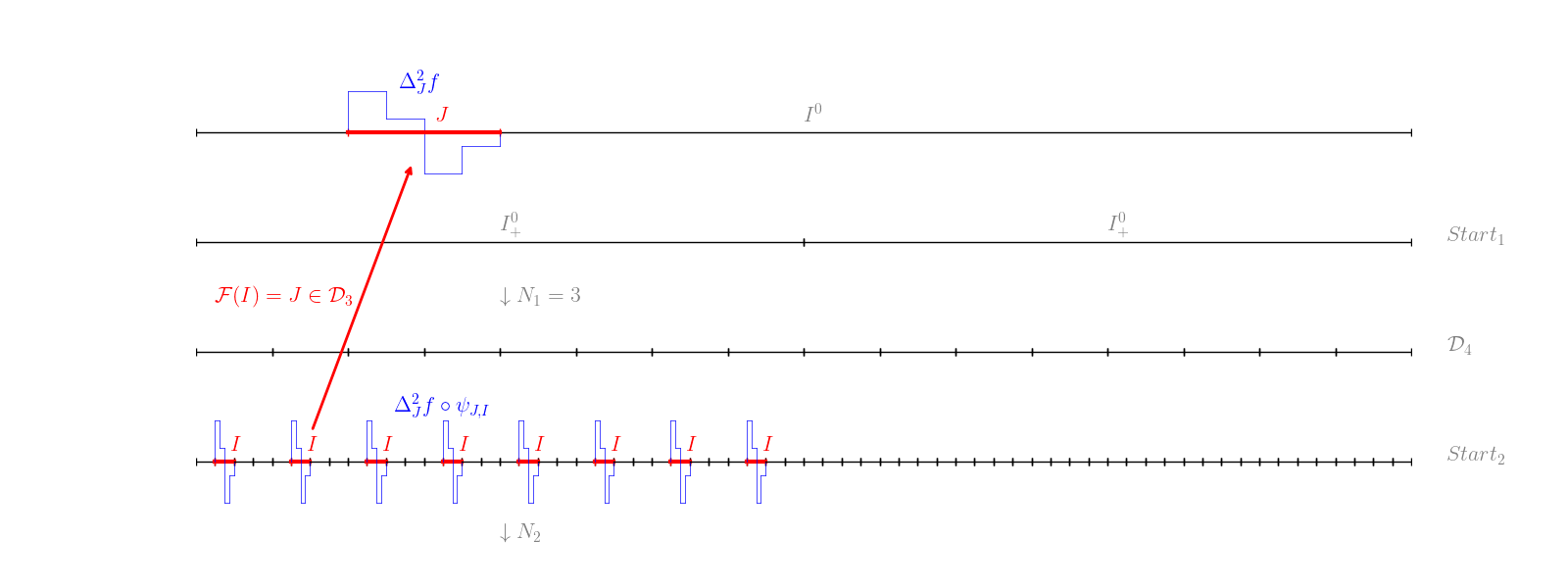}
    \caption{The first intervals on which periodisation is applied are
    $\Start_1=\{ I^0_-,I^0_+\}$. The function $f|_{I^0_-}$ 
    is copied $2^{N_1}=8$ times onto the $8$ intervals in $\mathcal{D}_4\cap{I^0_-}$. The grandchildren $\mathcal{D}_6\cap{I^0_-}$ of those $8$ intervals constitute the new starting intervals $\Start_2\cap{I^0_-}$ below $I^0_-$. Each of these grandchildren carries a copy of one of the grandchildren of $I^0_-$. For example the interval $J\in\mathcal{D_3}$ marked in red is the second grandchild (counting from the left) of $I^0_-$. Correspondingly the intervals $I\in\Start_2\cap{I^0_-}$ marked in red are the second grandchildren of intervals in $\mathcal{D}_4\cap{I^0_-}$. As such, the second difference $\Delta^2_I f$ on those intervals is a compressed version of the second difference $\Delta^2_J f$ through the map $\Psi_{J,I}$, with $J=\mathcal{F}(I)$.}
    \label{fig: Periodisation}
\end{figure}

Finally, the measure preserving property implies the preservation  of products,  norms,  and inner 
products:
\begin{lm}\label{l:P preserves norms}
For $f, \,g\in L^2(\R)$  with support on $I^0$ we have
\[
\per^{\vec N} (f g) = (\per^{\vec N} f)(\per^{\vec N} g),
\quad
\left(\per^{\vec N} f, \per^{\vec N} g \right)\ci{L^2(\R)} = \left(f, g \right)\ci{L^2(\R)},
\quad
\| \per^{\vec N} f \|\ci{L^2(\R)} = \| f \|\ci{L^2(\R)}.
\]
\end{lm}

\subsection{Remodeling} The iterated periodization will give us the desired lower bound for 
\linebreak $\left(\HT^\R \per^{\vec N}\bff, \per^{\vec N}\bg \right)\ci{L^2(\R)}$. Unfortunately, 
while the weight $\per^{\vec N} W$ has the same dyadic $\bA_2$ characteristic as $W$, its $\bA_2$ 
characteristic can blow up. 

Fortunately, a  simple idea, introduced 
by F.~Nazarov in \cite{Nazarov} and further 
developed later in \cite{KakTre21}, allows one to pass from dyadic characteristics to the classical 
ones. 

\subsubsection{\tup{(}Iterated\tup{)} quasi-periodizations}
\label{s:QPer}   We start with a dyadic $\bA_2$ weight, and we want to get a 
regular $\bA_2$ weight (with comparable constant). Our weight $W$ is generally not an $\bA_2$ 
weight, it is only a dyadic $\bA_2$ weight, and the periodization does not make the $\bA_2$ 
characteristic better. Fortunately, a slight modification of the iterated periodization, which we 
will call the \emph{quasi-periodization} allows us to get the $\bA_2$ condition on all intervals. 

Namely, for a function $f$ on $I^0$ with finite Haar expansion we get its quasi-periodization as 
follows. 

Recall that in the periodization a starting interval $I$ of order $k$ is divided into $2^{N_k}$ 
dyadic subintervals, which we called \emph{stopping intervals} of order $k$. Let us call the  
stopping intervals touching the boundary  of $I$ the \emph{exceptional} stopping intervals (there 
are $2$ of them), and the rest are the \emph{regular} stopping intervals. We use the notation 
$\cE(I)$ and $\cR(I)$ for exceptional and regular stopping subintervals respectively.

To get the function $f_{k}$ from $f_{k-1}$ on the starting interval $I$, we, as in the case of 
iterated periodization,  compute $\per\ci{I}^{N_k} f_{k-1}$, 
but  on the exceptional stopping subintervals $I'\in\cE(I)$ we put the values of the function $f_k$ 
to 
be $\La f_{k-1} \Ra\ci{I} = \La \per\ci{I}^{N_k} f_{k-1} \Ra\ci{I'}$; on the regular stopping 
intervals 
$I'\in \cR(I)$ we still have $f_k = \per\ci{I}^{N_k} f_{k-1}$. The result on the interval $I$ will 
be denoted by $\qper\ci{I}^{N_k} f_{k-1}$; note that it differs from $\per\ci{I}^{N_k} f_{k-1}$ 
only on two exceptional 
stopping subintervals of $I$. 

We then will do nothing on the exceptional stopping intervals, and continue on the regular 
ones: take their grandchildren as the starting intervals, and so on\ldots. 

The process will stop after finitely many steps, and we denote the result as $\qper^{\vec N} f$.

Note that the quasi-periodization is not given by a composition of the original function with a 
measure 
preserving transformation. 
However, exactly as in the case of iterated periodization, we will assign to each starting and 
each stopping interval $I$ of order $k$  a unique interval  $\Sinv(I)\in\cD_{2k-1}$, see end of 
Section \ref{s:IterPer}. Formally, the function $\Sinv$ is the  restriction of the function 
described above in Section \ref{s:IterPer}, because it is not defined on any of the descendants of 
the exceptional stopping intervals. However, to simplify typography again (and because we will not 
be using this function for periodization anymore) we will be using the same 
symbol $\Sinv$ for it. 

We can compare the iterated periodization and the iterated quasi-periodization quantitatively.
\begin{lm}\label{l:PversusQP}
Let $f$ be a function with finite Haar expansion, we have
\begin{align*}
\left\| \per^{\vec N} f - \qper^{\vec N} f \right\|_{L^2(\R)} \longrightarrow 0 \qquad \text{as } 
\vec 
N\to\infty.
\end{align*}
\end{lm}
\begin{proof}
    Since $f$ with finite Haar expansion are bounded, the difference above is controlled by the 
    total size of the exceptional intervals, which is arbitrarily small as $\vec N\to\infty$.
\end{proof}

\subsubsection{An alternative description of the iterated quasi-periodization} 
\label{s: alt QP}
Let us analyze the action of the quasi-periodization on martingale differences and write an 
analogue of
formula \eqref{e:P Delta f} for the (iterated) quasi-periodization. Applying the iterated 
quasi-periodization to the martingale difference decomposition \eqref{e:MartDiffDecomp2}, we get
almost the same formula as \eqref{e:P Delta f} with a slight modification. Namely, denote by
$\Startq\subset\Start$ the collection of all starting intervals that are not a subset of any
exceptional stopping interval. Then it is not hard to see that for $J\in\cD_{2k-1}$
\begin{align}\label{e:P Delta f 01}
\qper^{\vec N}\Delta\ci{J}^2 f = \sum_{\substack{I\in \Startq:\\ \cF(I) = J}} \qper\ci{I}^{N_k}
\left( \left(\Delta\ci{J}^2 f\right) \circ \psi\ci{J,I}  \right)   .
\end{align}

Looking at this formula and the martingale difference decomposition \eqref{e:MartDiffDecomp2}, we
can see that for a function $f$ with finite Haar expansion its quasi-periodization $\qper^{\vec N}
f$ can be obtained by starting with the function $\bE_1 f$ (which is constant on $I^0_\pm$) and then 
applying recursively on  starting intervals the following procedure finitely many times:

\begin{proc}\label{proc:QP step}
On an interval $I$ where the previously constructed function is constant, add to this 
function a quasi-periodization $\qper^{N(I)} \Delta\ci{I}^2\vf\ci{I}$ of an appropriate second 
order martingale difference.
\end{proc}

The order of starting intervals where we apply the above procedure is not really important: it 
just needs to agree with the ordering by inclusion of the intervals. Namely, we can apply the 
procedure to a starting interval only after we are done with all starting intervals containing it. 

\subsubsection{Remodeling} In our case we apply the 
quasi-periodization to the weights $W$, $V=W^{-1}$ and the functions $\bff$, $\bg$; recall that we
assumed that all the objects have finite Haar expansions. We get the quasi-periodization of all 
objects,
but while we do not need to worry about $\wt\bff:=\qper^{\vec N} \bff$, $\wt\bg:= \qper^{\vec N}
\bg$, we have a small problem with the quasi-periodizations of the weights: on the exceptional 
stopping
intervals the weights $\qper^{\vec N} V$ and $\qper^{\vec N} W$ are not inverses of each other
(but they are everywhere else).

Fortunately, this problem is very easy to fix. After performing the (iterated) quasi-periodization, 
we got a
disjoint collection of exceptional stopping intervals, where we have a problem. So, on each such 
exceptional stopping interval 
$I$ of order $k$ we replace constants by appropriately chosen weights 
and perform the iterated quasi-periodization of these  
weights.   
We will get a new collection of exceptional stopping intervals, perform iterated 
quasi-periodization of appropriate weights on these intervals, and so on\ldots

To write it down formally, recall that during the quasi-periodization,
as it was discussed above at the end of Section \ref{s:QPer}, we associate with each stopping 
interval $I$  of order $k$ a unique interval $J=\Sinv(I)\in\cD_{2k-1}$. And on the exceptional 
stopping intervals  we perform 
the iterated quasi-periodization $\qper^{\vec 2}$ of the weights  $V\circ\psi\ci{J,I}$, 
$W\circ\psi\ci{J,I}$, $J=\Sinv(I)$; here $\vec 2 := (2, 2, 2, \ldots)$.

We get a new set of exceptional stopping intervals, and their total length will be $1/2$ of the 
total length of the previous stopping intervals.  On these new stopping intervals%
\footnote{the map $\Sinv$ is naturally extended to new stopping and starting intervals.} 
we again perform 
the iterated quasi-periodization $\qper^{\vec 2}$ of the corresponding weights, again treating 
these stopping intervals, and not their grandchildren, as starting ones (of the same order), get 
new stopping intervals, and so on. 

In every step of the remodeling iterations, we will get new 
exceptional stopping intervals, but their total measure will be  half of the measure of the 
previous stopping intervals.   

Therefore, inductively repeating this procedure
and taking the limit,%
\footnote{By the construction the limit exists a.e., and since $W$ and $V=W^{-1}$ have finite Haar 
expansion, the limit exists in (matrix-valued) $L^1$.} 
we will end up with the weights $\wt V$, $\wt W$ such that $\wt V=\wt 
W^{-1}$ a.e.%

We will denote the new remodeled weights as $\wt V := \REM^{\vec N} V$, $\wt W := \REM^{\vec N} W$.

\begin{rem}
We would like to emphasize that in the remodeling we only change the weights; the functions
$\wt\bff$, $\wt\bg$  will be the quasi-periodizations $\wt\bff=\qper^{\vec N} \bff$,
$\wt\bg=\qper^{\vec N} \bg$, and will not be changed any further.  We will use Lemma
\ref{l:HilbPer01a} below to estimate $\left(\HT \wt\bff, \wt\bg\right)\ci{L^2}$, so we will need to
pick appropriately large frequencies $N_k$ later.
As for the behavior of the weight, all the modifications after the first quasi-periodization do not 
influence estimates of $\left(\HT \wt\bff, \wt\bg\right)\ci{L^2}$, and we will not need arbitrarily 
large frequencies. We are performing the remodeling to make the weights $\wt V$ and $\wt W$ to be 
the inverses of each other, and for this purpose any frequency $N\ge2$ (the same for all 
exceptional intervals) will work. So above we just picked $N=2$. 
\end{rem}

\begin{rem}\label{r:remod}
{ As we discussed above, the remodeling in our construction was done by performing consecutively 
quasi-periodizations on some intervals on some functions with finite Haar expansion, and 
quasi-periodization is done applying the above Procedure \ref{proc:QP step}  finitely many times.   

Further analyzing the remodeling, we can see that the remodeling of the function $f$ can be 
obtained by 
starting with the function $\bE_1 f :=\La f \Ra\ci{I^0_+}\1\ci{I^0_+} + \La f 
\Ra\ci{I^0_-}\1\ci{I^0_-}$ (which is constant on intervals $I_{\pm}^0$), then 
applying inductively Procedure  \ref{proc:QP step} and taking the limit (a.e.~and/or in $L^1$).

Again, the order of starting intervals, where we apply the above procedure is not really 
important and it
only needs to agree with the ordering by inclusion of the intervals. Namely, we can apply the 
procedure to a starting interval only after we are done with all starting intervals containing it. }
\end{rem}

The picture below illustrates the iterated quasi-periodisation process and should be compared with the iterated periodisation process from Fig. \ref{fig: Periodisation}.

\begin{figure}[ht]
    \centering
    \vspace*{-0.4cm}\hspace*{-2cm}\includegraphics[width=18cm]{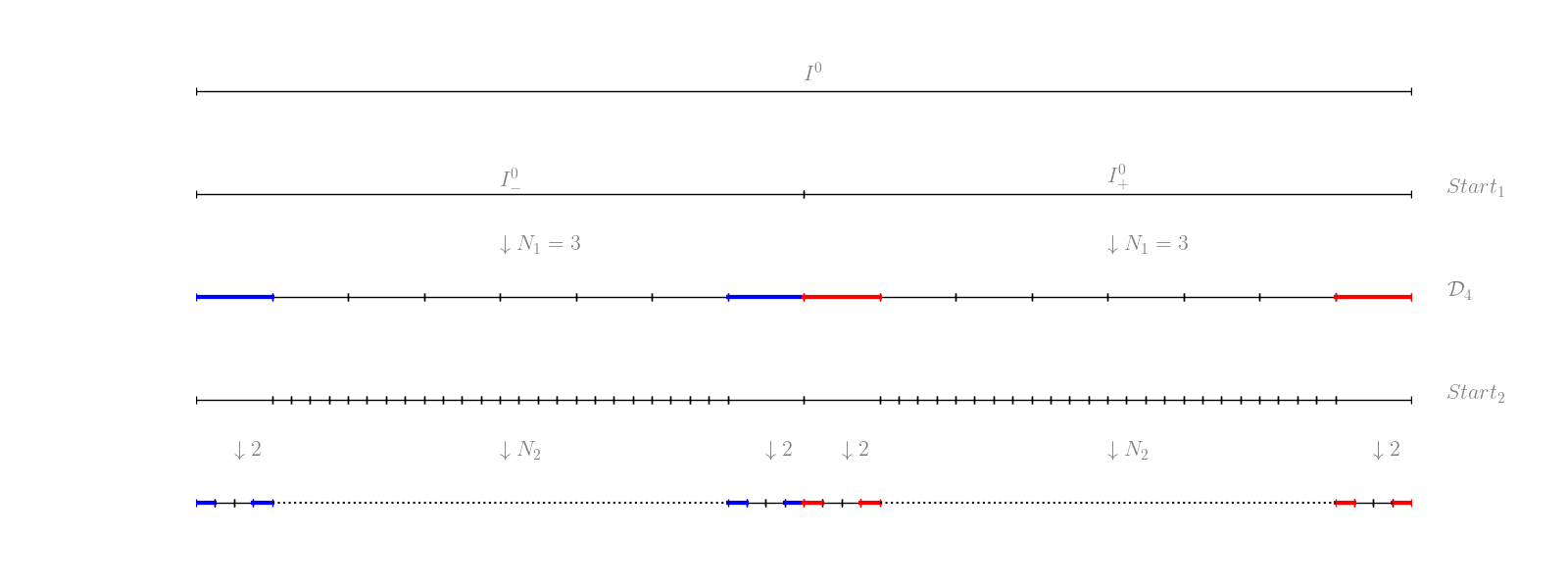}
    \caption{
    The first intervals on which quasi-periodisation is applied are again
    $\Start_1=\{ I^0_-,I^0_+\}$.
    The function $f|_{I^0_-}$ 
    is copied $2^{N_1}-2=6$ times on the $6$ regular stopping intervals $\mathcal{R}(I^0_-)$ in $\mathcal{D}_4\cap{I^0_-}$ that are \emph{not} touching the boundary of $I^0_-$. The two blue intervals from $\mathcal{D}_4\cap{I^0_-}$ that are touching the boundary of $I^0_-$
    are the two exceptional stopping intervals $\mathcal{E}(I^0_-)$ below $I^0_-$,
    the red intervals are $\mathcal{E}(I^0_+)$.
    }
    \label{fig: QuasiPeriodisation}
\end{figure}

\begin{lm}\label{l:QP norms limit}
Let $\bff$, $\bg$, $W$, $W^{-1}$ be functions with finite Haar expansions, and let $W\ci{\!\!\vec 
N} :=\REM^{\vec 
N} W$ be the remodeling of the weight $W$. Then 
\begin{align*}
\left\| \qper^{\vec{N}} \bff \right\|\ci{ L^2(\REM^{\vec N} W)}   \longrightarrow 
\|\bff\|\ci{L^2(W)}, \qquad
\left\| \qper^{\vec{N}} \bg \right\|\ci{ L^2((\REM^{\vec N} W)^{-1})}   \longrightarrow 
\|\bg\|\ci{L^2(W^{-1})}.
\end{align*}
\end{lm}
\begin{proof}
Since the function $\qper^{\vec{N}} \bff$ is constant on the exceptional stopping intervals, the 
contribution of these intervals to the norm $\left\| \qper^{\vec{N}} \bff \right\|\ci{ 
L^2(\REM^{\vec N} W)}$ depends only on the averages of the weight $\REM^{\vec N} W$ there. Therefore
\begin{align*}
\left\| \qper^{\vec{N}} \bff \right\|\ci{ L^2(\REM^{\vec N} W)} =
\left\| \qper^{\vec{N}} \bff \right\|\ci{ L^2(\qper^{\vec N} W)} =
\left\| \qper^{\vec{N}} \bff \right\|\ci{ L^2(\per^{\vec N} W)}. 
\end{align*}
We also trivially have $ \left\| \per^{\vec{N}} \bff \right\|\ci{ L^2(\per^{\vec N} W)} = 
\|\bff\|\ci{L^2(W)}$. Because of the finite Haar expansion the weight $W$ is bounded, and applying 
Lemma \ref{l:PversusQP} we get the first limit. 

Noticing that $(\REM^{\vec N}W)^{-1}= \REM^{\vec N} (W^{-1})$ we immediately get the second limit. 
\end{proof}

\subsection{Hilbert Transforms and periodization}

Recall that the Hilbert Transform on the real line $\HT=\HT^\R$ is defined as 
\begin{align*}
\HT^\R f(s) =\frac1\pi \text{p.v.}\int_\R \frac{f(t)}{s-t} \dd t
\end{align*}
and the Hilbert Transform $\HT^\T$ on the torus $\T :=\R\slash \Z$ (the periodic Hilbert Transform) 
is given by 
\begin{align*}
\HT^\T f (s) = \text{p.v.} \int_0^1 f(t) \cot\left(\pi\cdot (s-t)\right) \dd t; 
\end{align*}
here we naturally identify $\T$ with $I^0=[0,1)$ with the Lebesgue measure $\dd x$. We can also 
identify $\T$ with an arbitrary interval $I$ with the normalized Lebesgue measure $|I|^{-1} \dd x$. 
In this case the Hilbert Transform on the torus $\T$ is represented as  
\begin{align*}
\HT\ci{I}^\T f (s) = 
|I|^{-1}\text{p.v.} \int_I f(t) \cot\left(|I|^{-1}\pi\cdot (s-t)\right) \dd t, \qquad s\in I.  
\end{align*}

If we denote $L^2(I)= L^2(I, |I|^{-1}\dd x)$, then the representations $\HT\ci{I}^\T$ in $L^2(I)$  
and $\HT\ci{J}^\T$ in $L^2(J)$ are unitarily equivalent, 
\begin{align}\label{e:scale HT^T}
\HT\ci{I}^\T U\ci{I,J} = U\ci{I,J} \HT\ci{J}^\T,  
\intertext{where the unitary operator $U\ci{I,J}: L^2(J)\to L^2(I)$ is given by }
\label{e:scale HT^T 01}
U\ci{I,J} f = f\circ \psi\ci{J,I}. 
\end{align}

Define
\begin{align}\label{e:c12}
c_1  := \left(\HT^\T h\ci{I^0}, h\ci{I^0_+}\right)\ci{L^2(\T)}  , \qquad
c_2 := \left(\HT^\T h\ci{I^0_+}, h\ci{I^0_-}\right)\ci{L^2(\T)}.  
\end{align}
Recall, that extending a function $f\in L^2(I^0)$ periodically  to the 
real line $\R$, we naturally identify $I^0$ with the torus $\T=\R/\Z$ and the space 
$L^2(I^0)$ with $L^2(\T)$. 
The rotation of the torus $\T$ through $\pi$, i.e.~map $f(\fdot)\mapsto f(\fdot + 1/2)$ 
in $L^2(\R/\Z)) = L^2(I^0)$ is a unitary transformation in $L^2(I^0)$, mapping $h\ci{I^0}$ to 
$-h\ci{I^0}$, and $h\ci{I^0_+}$ to $h\ci{I^0_-}$, so we can write 
\begin{align}\label{e:c1a}
c_1 = - \left(\HT^\T  h\ci{I^0},  h\ci{I^0_-}\right)\ci{L^2(I^0)}.
\end{align}

Let the dyadic operator $\Hdy$ be defined as 
\begin{align}\label{e:mod-shift}
\Hdy:= c_1 (\sh - \sh^*) + c_2\shst . 
\end{align}
Let also
\begin{align}
\label{e:c0}
c_0:= \left(\HT^\R \1\ci{I^0_+}, \1\ci{I^0_-}\right)\ci{L^2(\T)}.
\end{align}

\begin{lm}\label{l:HilbPer01}
Let $\bff$, $\bg$ be functions on $I^0 =[0,1)$ with finite Haar expansion. Then
\begin{align*}
 \left(\HT^\R \per^{\vec N} \bff, \per^{\vec N} \bg\right)_{L^2(\R)}
\longrightarrow  \left(\Hdy\bff, \bg\right)\ci{L^2(I^0)}  + 
c_0 \left[\left( \La \bff\Ra\ci{I^0_+}, \La \bg\Ra\ci{I^0_-} \right)\ci{\R^2}  - 
\left( \La \bff\Ra\ci{I^0_-}, \La \bg\Ra\ci{I^0_+} \right)\ci{\R^2}\right]   
\end{align*}
along some family of $\vec N\to\infty$ \tup{(}all $N_k\to\infty$\tup{)}. 
\end{lm}

\begin{rem}\label{r:HilbPer convergence}
It is possible to show that in fact we have the convergence as $\vec N\to \infty$, but for our 
purposes convergence along some family is sufficient. 
\end{rem}

Using Lemma  \ref{l:PversusQP} we can trivially get from  
the above Lemma \ref{l:HilbPer01} 
\begin{lm}\label{l:HilbPer01a}
Let $\bff$, $\bg$ be functions on $I^0 =[0,1)$ with finite Haar expansion. Then
\begin{align*}
 \left(\HT^\R \qper^{\vec N} \bff, \qper^{\vec N} \bg\right)_{L^2(\R)}
\longrightarrow \left(\Hdy \bff, \bg\right)\ci{L^2(I^0)} +
c_0 \left[\left( \La \bff\Ra\ci{I^0_+}, \La \bg\Ra\ci{I^0_-} \right)\ci{\R^2}  - 
\left( \La \bff\Ra\ci{I^0_-}, \La \bg\Ra\ci{I^0_+} \right)\ci{\R^2}\right]  . 
\end{align*}
along some family of $\vec N\to\infty$ \tup{(}all $N_k\to\infty$\tup{)}. 
\end{lm}

In both lemmas above the first term $\left(\Hdy \bff, \bg\right)\ci{L^2(I^0)}$ is the essential one 
that will be used, and the second term is small in comparison.

To prove Lemma \ref{l:HilbPer01} we will need some additional results. The first one explains the 
connection between the operator $\Hdy $ and the periodic Hilbert Transform $\HT\ci{I}^\T$. 

\begin{lm}\label{l:HTvsT} Let $I\in\cD$. 
For all $f, g\in \ran\Delta\ci{I}^2$,
\begin{align}
\left(\HT\ci{I}^\T f, g\right)\ci{L^2(I)} = \left(\Hdy  f, g\right)\ci{L^2(I)}.
\end{align}
\end{lm}
\begin{proof}
First of all notice that the functions $|I|^{1/2} h\ci{I}$, $|I|^{1/2} h\ci{I_-}$, $|I|^{1/2} 
h\ci{I_+}$ form an orthonormal basis for $\ran \Delta\ci{I}^2$ in the norm of $L^2(I)$, and that
\begin{align*}
h\ci{I^0}= |I|^{1/2} h\ci{I}\circ \psi\ci{I,I^0}, \qquad
h\ci{I^0_\pm}= |I|^{1/2} h\ci{I_\pm}\circ \psi\ci{I,I^0}. 
\end{align*}

Identities 
\eqref{e:scale HT^T} mean that $\HT\ci{I}^\T$ in $L^2(I)$ and $\HT\ci{I^0}^\T$ in $L^2(I^0)$ are 
unitarily equivalent
\begin{align*}
\HT\ci{I^0}^\T U\ci{I^0, I}   = U\ci{I^0,I} \HT\ci{I}^\T, 
\end{align*}
where, recall,
$U\ci{I^0, I} f = f \circ \psi\ci{I, I^0}$; note that $U\ci{I^0, I}$ is a unitary 
operator from $L^2(I)$ to $L^2(I^0)$.
Therefore the definitions \eqref{e:c12}, \eqref{e:c1a} of 
$c_1$, $c_2$ can be rewritten as
\begin{align}\label{e:c1}
c_1 & := \left(\HT^\T h\ci{I^0}, h\ci{I^0_+}\right)\ci{L^2(\T)}  
 =  |I|\left(\HT\ci{I}^\T  h\ci{I},  h\ci{I_+}\right)\ci{L^2(I)}
= -\left(\HT^\T h\ci{I^0}, h\ci{I^0_-}\right)\ci{L^2(\T)}   
, \\   \label{e:c2}
c_2 & := \left(\HT^\T h\ci{I^0_+}, h\ci{I^0_-}\right)\ci{L^2(\T)}  
 = |I| \left(\HT\ci{I}^\T  h\ci{I_+},  h\ci{I_-}\right)\ci{L^2(I)}.
\end{align}

Decompose $f, g\in\ran\Delta\ci{I}^2$ as
\begin{align*}
f & = x\ci{I} |I|^{1/2}h\ci{I} + x\ci{I_-} |I|^{1/2}h\ci{I_-} + x\ci{I_+} |I|^{1/2}h\ci{I_+} ,\\
g & = y\ci{I} |I|^{1/2}h\ci{I} + y\ci{I_-} |I|^{1/2}h\ci{I_-} + y\ci{I_+} |I|^{1/2}h\ci{I_+} .
\end{align*}
Then
\begin{align*}
\left(\HT\ci{I}^\T f, g\right)\ci{L^2(I)} &= 
|I|\left(\HT\ci{I}^\T \left(x\ci{I}  h\ci{I} + x\ci{I_-}  h\ci{I_-} + x\ci{I_+} 
 h\ci{I_+}\right) , y\ci{I}  h\ci{I} + y\ci{I_-}  h\ci{I_-} + y\ci{I_+} 
 h\ci{I_+}  \right)\ci{L^2(I)} \\
&=
|I|\begin{pmatrix}
y\ci{I}\\ y\ci{I_-} \\ y\ci{I_+} 
\end{pmatrix}^T
\begin{pmatrix}
\left( \HT\ci{I}^\T h\ci{I} , h\ci{I} \right)\ci{L^2(I)} & \left( \HT\ci{I}^\T h\ci{I_-} , h\ci{I} 
\right)\ci{L^2(I)} & \left( \HT\ci{I}^\T h\ci{I_+} , h\ci{I} \right)\ci{L^2(I)}    \\
\left( \HT\ci{I}^\T h\ci{I} , h\ci{I_-} \right)\ci{L^2(I)} & \left( \HT\ci{I}^\T h\ci{I_-} , 
h\ci{I_-} \right)\ci{L^2(I)} & \left( \HT\ci{I}^\T h\ci{I_+} , h\ci{I_-} \right)\ci{L^2(I)} 
\\
\left( \HT\ci{I}^\T h\ci{I} , h\ci{I_+} \right)\ci{L^2(I)} & \left( \HT\ci{I}^\T h\ci{I_-} , 
h\ci{I_+} \right)\ci{L^2(I)} & \left( \HT\ci{I}^\T h\ci{I_+} , h\ci{I_+} \right)\ci{L^2(I)}
\end{pmatrix}
\begin{pmatrix}
x\ci{I}\\ x\ci{I_-} \\ x\ci{I_+} 
\end{pmatrix}
\\ & = 
\phantom{|I|}
\begin{pmatrix}
y\ci{I}\\ y\ci{I_-} \\ y\ci{I_+} 
\end{pmatrix}^T
\begin{pmatrix}
0 & c_1 & -c_1 \\
-c_1 & 0 & c_2 \\
c_1 &  -c_2 & 0
\end{pmatrix}
\begin{pmatrix}
x\ci{I}\\ x\ci{I_-} \\ x\ci{I_+} 
\end{pmatrix}.
\end{align*}
Applying similar computations to $\left(\sh f, g\right)\ci{L^2(I)}$, $\left(\sh^* f, 
g\right)\ci{L^2(I)}$ and $\left(\shst f, g\right)\ci{L^2(I)}$ we get 
\begin{align*}
\left(\sh f, g\right)\ci{L^2(I)}
& = 
\begin{pmatrix}
y\ci{I}\\ y\ci{I_-} \\ y\ci{I_+} 
\end{pmatrix}^T
\begin{pmatrix}
0 & 0 & 0 \\
-1& 0 & 0 \\
1 &  0 & 0
\end{pmatrix}
\begin{pmatrix}
x\ci{I}\\ x\ci{I_-} \\ x\ci{I_+} 
\end{pmatrix},
\\
\left(\sh^* f, g\right)\ci{L^2(I)} 
&= 
\begin{pmatrix}
y\ci{I}\\ y\ci{I_-} \\ y\ci{I_+} 
\end{pmatrix}^T
\begin{pmatrix}
0 & -1 & 1 \\
0& 0 & 0 \\
0 &  0 & 0
\end{pmatrix}
\begin{pmatrix}
x\ci{I}\\ x\ci{I_-} \\ x\ci{I_+} 
\end{pmatrix},
\\
\left(\shst f, g\right)\ci{L^2(I)} 
&= 
\begin{pmatrix}
y\ci{I}\\ y\ci{I_-} \\ y\ci{I_+} 
\end{pmatrix}^T
\begin{pmatrix}
0 & 0  & 0 \\
0 & 0  & 1 \\
0 & -1 & 0
\end{pmatrix}
\begin{pmatrix}
x\ci{I}\\ x\ci{I_-} \\ x\ci{I_+} 
\end{pmatrix},
\end{align*}
which gives the desired conclusion. 
\end{proof}

\begin{lm}\label{l:HilbPer02}
Let $f$ be a function on an interval $I$, $\La f\Ra\ci{I}=0$ with finite Haar expansion on $I$, 
extended by $0$ outside of $I$. Then 
\begin{align}\label{e:Per HT 01}
\lim_{ N\to \infty}\left\| \HT^\R  \per^{N}_{I} f - \per^{N}_{I} \HT\ci{I}^\T f  
\right\|_{L^2(I)} 
= 0, 
\qquad 
\lim_{ N\to \infty}\left\|\1\ci{\R\setminus I} \HT^\R \per^{N}_{I} f  \right\|_{L^2(\R)} = 0.
\end{align}
\end{lm}

\begin{proof}
Since the Hilbert Transforms $\HT^\R$ and $\HT\ci{I}^\T$ commute with rescaling and translations, 
it is sufficient to prove the lemma for $I=I^0=[0,1)$. 

Let $g^N$ be the function $\per^{N}_{I^0} f$ extended by $1$-periodicity to $\R$. Then for any 
$f\in L^2(I^0)$ we have on $I^0$
\begin{align*}
\HT^\R g^N = \per^{N}_{I^0} \HT^\T f;  
\end{align*}
the identity is trivial for the exponentials $e^{2\pi i k x}$, $k\in \Z$, and extends to 
$L^2(I^0)$ by continuity. We then need to show that 
\begin{align}\label{e:cross int 01}
\left\| \1\ci{I^0} \HT^\R (\1\ci{\R\setminus I^0})  \right\| \to 0 \qquad \text{as } N\to\infty. 
\end{align}
But this is an easy exercise in basic analysis. Namely, denote $I^0_m:= [-2^{-m}, 1+ 2^{-m}]$. Then 
using the fact that $\La g^N \Ra\ci{I} = 0$ for any interval $I$ of length $2^{-N}$, we can easily 
see that for a fixed $m\in\N$
\begin{align*}
\HT^\R (\1\ci{\R\setminus I^0_m})  \rightrightarrows 0 \quad \text{on } I^0,  
\qquad \text{as }N\to\infty . 
\end{align*}
But $\| 1\ci{I^0_m\setminus I^0} f \|\ci{L^2}$ can be made arbitrarily small, and together with the 
boundedness of $\HT^\R$ this implies \eqref{e:cross int 01}. Thus, the first identity in \eqref{e:Per 
HT 01} is proved. 

The second identity in \eqref{e:Per HT 01} is proved similarly to \eqref{e:cross int 01}. 
\end{proof}

\begin{proof}[Proof of Lemma \ref{l:HilbPer01}]
Let $\Start = \Start^{\vec N}$ be the collection of all starting intervals appearing in the 
iterated periodization, and let $\Start_k = \Start^{\vec N}_k$ be the collection of all starting 
intervals of order $k$. Since the periodization does not change constant functions, using the 
alternative description of the periodization in Section \ref{s:alt per} we can write
\begin{align*}
\per^{\vec N} \bff & = \bE\ci{I^0_+} \bff + \bE\ci{I^0_-} \bff + \sum_{k\ge 1}\ \sum_{I\in\Start_k} 
\per^{N_k}_I \left(\left(\Delta\ci{\Sinv(I)}^2 \bff \right)\circ\psi\ci{\Sinv(I), I}\right), \\
 & =: \bE_1 \bff + \sum_{I\in \Start} D\ci{I} \bff,
\end{align*}
and similarly for $\per^{\vec N}\bg$.  Here, to simplify the write-up, we do not explicitly write 
the 
dependence of 
\begin{align}\label{e:DI f}
D\ci{I}\bff := \per^{N_k}_I \left(\left(\Delta\ci{\Sinv(I)}^2 \bff \right)\circ\psi\ci{\Sinv(I), 
I}\right)
\end{align}
on the generation $k$ and the indices $N_j$, $j\le k$, but we will 
have it in mind.

Then we can write
\begin{align*}
\left(\HT^\R \per^{\vec N} \bff, \per^{\vec N} \bg\right)_{L^2(\R)} =
\left(\HT^\R\bE_1 \bff, \bE_1 \bg\right)\ci{L^2(I^0)}  
& + \sum_{I\in\Start} \left(\HT^\R D\ci{I} \bff, D\ci{I} \bg\right)\ci{L^2(I^0)}
\\
&+ \text{ cross terms.}
\end{align*}
We will show later that we can pick the $N_k$ sufficiently large, so the sum of the cross terms is 
arbitrarily small. As for now, let us compute the limit of the above sum. We fix $I\in \Start_k$, 
and let $J=\Sinv(I)\in\cD_{2k-1} $. Denote
\begin{align*}
f\ci{I}:=\left(\Delta\ci{J}^2 \bff \right)\circ\psi\ci{J, I}, \qquad
g\ci{I}:=\left(\Delta\ci{J}^2 \bg \right)\circ\psi\ci{J, I} . 
\end{align*}
Trivially, $f\ci{I}, g\ci{I}\in \ran \Delta\ci{I}^2$ and for a stopping interval $I$ of order $k$ 
we have 
\[
D\ci{I} \bff = \per\ci{I}^{N_k} f\ci I, \qquad D\ci{I} \bg = \per\ci{I}^{N_k} g\ci I . 
\]
For one given $I$ in the sum above, the term of interest writes also
\[
\left(\HT^\R D\ci{I} \bff, D\ci{I} \bg\right)\ci{L^2(I^0)} =  \left(\HT^\R D\ci{I} \bff, 
D\ci{I} \bg\right)\ci{L^2(I)} |I|,
\]
where asymptotically
\begin{align*}
 \left(\HT^\R D\ci{I} \bff, D\ci{I} \bg\right)\ci{L^2(I)} 
& = \left(\HT^\R \per\ci{I}^{N_k} f\ci{I} , \per\ci{I}^{N_k} g\ci{I} \right)\ci{L^2(I)} 
\\
&\to \left( \per\ci{I}^{N_k} \HT\ci{I}^\T  f\ci{I} , \per\ci{I}^{N_k} g\ci{I} \right)\ci{L^2(I)} 
\\
& = \left(  \HT\ci{I}^\T  f\ci{I} ,  g\ci{I} \right)\ci{L^2(I)}
= \left(  \HT\ci{I}^\T  \left(\Delta\ci{J}^2 \bff \right)\circ\psi\ci{J, I} ,
            \left(\Delta\ci{J}^2 \bg \right)\circ\psi\ci{J, I} \right) \ci{L^2(I)}
\\
& = \left(  \HT\ci{J}^\T  \Delta\ci{J}^2 \bff, \Delta\ci{J}^2 \bg \right) \ci{L^2(J)}  
= \left(\Hdy  \Delta\ci{J}^2 \bff, \Delta\ci{J}^2 \bg \right)\ci{L^2(J)} ,
\end{align*}
and where the last equality follows from Lemma \ref{l:HTvsT}.

To conclude with the diagonal terms, we collect contributions of all starting intervals $I\in 
\Start$ in the main term and 
estimate asymptotically when $N \to \infty$,
\begin{align*}
    \sum_{I\in\Start} \left(\HT^\R D\ci{I} \bff, D\ci{I} \bg\right)\ci{L^2(I^0)}
    & = \sum_{k \ge 1}\sum_{I\in\Start_k} \left(\HT^\R D\ci{I} \bff, D\ci{I} \bg\right)\ci{L^2(I^0)}
    \\
    & = \sum_{k \ge 1}\sum_{I\in\Start_k} \left(\HT^\R D\ci{I} \bff, D\ci{I} \bg\right)\ci{L^2(I)} 
    |I|
    \\
    & \to  \sum_{k \ge 1}\sum_{I\in\Start_k} \left(  \Hdy   \Delta\ci{\Sinv(I)}^2 \bff ,  
    \Delta\ci{\Sinv(I)}^2 \bg \right)\ci{L^2(\Sinv(I))} |I|
    \\
    & =  \sum_{k \ge 1}\sum_{J\in\cD_{2k-1}} \sum_{I:\Sinv(I)=J} \left(  \Hdy   \Delta\ci{J}^2 \bff 
    ,  \Delta\ci{J}^2 \bg \right)\ci{L^2(J} |I|
    \\
    & = \sum_{J\in\cD\ti{odd}} \left(  \Hdy   \Delta\ci{J}^2 \bff ,  \Delta\ci{J}^2 \bg 
    \right)\ci{L^2(J)}  |J|
    = \left(  \Hdy    \bff ,   \bg \right)\ci{L^2(I^0)} \,;
\end{align*}
here we used that the total measure of all starting intervals $I$ such that $\Sinv(I) =J$ is 
exactly $|J|$.

Finally, we can easily see that 
\begin{align*}
\left(\HT^\R \bE_1 \bff, \bE_1 \bg  \right)\ci{L^2(I^0)}  = c_0 \left[\left( \La \bff\Ra\ci{I^0_+}, 
\La 
\bg\Ra\ci{I^0_-} \right)\ci{\R^2}  - 
\left( \La \bff\Ra\ci{I^0_-}, \La \bg\Ra\ci{I^0_+} \right)\ci{\R^2}\right] , 
\end{align*}
so if we can make the sum of the cross terms as small as we want, the lemma is proved. 

To complete the proof, let us show that asymptotically, we can ignore the cross terms. 
Since we have finitely many cross terms, it is sufficient to show that given small  $\e>0$ 
we can make each cross term bounded by $\e$. 

First, using the second limit in Lemma \ref{l:HilbPer02}, we can see that for all starting 
intervals 
$I$, $J$ of order $1$, $I\ne J$,   and for all sufficiently large $N_1$, 
\begin{align}\label{e:DI cross}
\left| \left(\HT^\R D\ci{I}\bff, D\ci{J}\bg \right)\ci{L^2(I^0)}\right| <\e. 
\end{align}
We know that functions $D\ci{I}\bff$, $D\ci{I}\bg$ are just periodizations of appropriate 
functions, so for stopping intervals $I$ of order $1$,  
\[
D\ci{I}\bff\to0, \quad D\ci{I}\bg \to 0 \qquad \text{weakly as } N_1\to\infty. 
\]
That means that for all sufficiently large $N_1$, 
\[
\left| \left(\HT^\R D\ci{I}\bff, \bE_1\bg \right)\ci{L^2(I^0)}\right| <\e, \qquad
\left| \left(\HT^\R \bE_1\bff, D\ci{I}\bg \right)\ci{L^2(I^0)}\right| <\e
\]
for all starting intervals $I$ of order $1$. 

Let us fix $N_1$ such that all the inequalities hold. 

We  
then proceed by induction. We assume that we found indices $N_1, \ldots, N_{k-1}$ such that all 
the cross terms involving (only) $\bE_1$ and $D\ci{I'}$ with the starting intervals $I'$ of rank 
$\le k-1$ have the desired estimates. We now consider all the vectors 
\begin{align}\label{e:fix vectors}
\bE_1\bff, \quad \bE_1\bg, \quad D\ci{I'} \bff, \quad D\ci{I''} \bg 
\end{align}
with $I$ of rank $\le k-1$, to be fixed at this step. 

Again, using the second limit in Lemma \ref{l:HilbPer02}, we can see that for all starting 
intervals 
$I$, $J$ of order $k$, $I\ne J$, and for all sufficiently large $N_k$, the inequality \eqref{e:DI 
cross} holds. 

Since for all starting intervals $I$ of rank $k$
\[
D\ci{I}\bff\to0, \quad D\ci{I}\bg \to 0 \qquad \text{weakly as } N_k\to\infty,  
\]
we can find sufficiently large $N_k$ such that all the inner products of $\HT^\R D\ci{J}\bff$, 
$\HT^\R D\ci{J}\bg$ with the vectors fixed in \eqref{e:fix vectors} are bounded by $\e$. 

Thus we can ignore the cross terms, and the lemma is proved.
\end{proof}

\section{Lower bound for the Hilbert Transform} \label{s:Lower bound Hilbert}

\begin{proof}[Proof of Theorem \ref{MainTheorem1}]

The proof of the main theorem is now in order. We want to find functions $\wt f\in 
L^2(W)$, $\wt g\in L^2(W)$,  $\|\wt f\|\ci{L^2(W)} = \|\wt g\|\ci{L^2(W)} =1$ such that for 
$\HT=\HT^\R$
\begin{align*}
\left(\HT \wt f, \wt g\right)\ci{L^2(W)} \gtrsim \bcQ^{3/2}. 
\end{align*}
Denoting  $\wt \bg := W \wt g$  (and $\wt \bff:= \wt f$  for 
consistency of notation), we can see that 
the above estimate is 
equivalent to 
\begin{align}\label{e:biest Hilb 02}
\left(\HT \wt \bff, \wt \bg\right)\ci{L^2(I^0)} \gtrsim \bcQ^{3/2}, \qquad 
\| \wt \bff\|\ci{L^2(W)} =  \|\wt \bg\|\ci{L^2(W^{-1})} = 1 
\end{align}
(with the same implied constant).  

Recall that from Lemma \ref{l:HilbPer01a}, we have 
for any $\bff$, $\bg$ with finite Haar expansion
\begin{align}\label{e:HilbPer01b}
 \left(\HT^\R \qper^{\vec N} \bff, \qper^{\vec N} \bg\right)_{L^2(\R)}
\longrightarrow \left(\Hdy\bff, \bg\right)\ci{L^2(I^0)} +
c_0 \left[\left( \La \bff\Ra\ci{I^0_+}, \La \bg\Ra\ci{I^0_-} \right)\ci{\R^2}  - 
\left( \La \bff\Ra\ci{I^0_-}, \La \bg\Ra\ci{I^0_+} \right)\ci{\R^2}\right]  . 
\end{align}
We also know,  see \eqref{e:biest02}, that there exist  $\bff$, $\bg$, 
$\|\bff\|\ci{L^2(W)} =  \|\bg\|\ci{L^2(W^{-1})} = 1$,   such that 
\begin{align}\label{e:limitHQP} 
\left(\Hdy \bff, \bg\right)\ci{L^2(I^0)} 
\gtrsim \bcQ^{3/2}.
\end{align}

In our example the functions $\bff$ and $\bg$ were explicitly constructed, but in what follows 
we only need the estimate \eqref{e:limitHQP}.  

Trivially, by taking a slightly smaller implicit constant in the estimate \eqref{e:limitHQP} we can 
always assume that $\bff$ and $\bg$ have finite Haar expansions  (so \eqref{e:HilbPer01b} 
also holds). 

Let us show that the second term in \eqref{e:HilbPer01b} is negligible.  For that it is enough to 
show that for any $\bff$, $\bg$, $\|\bff\|\ci{L^2(W)} =  \|\bg\|\ci{L^2(W^{-1})} = 1$ we have 
\begin{align}\label{e:small (f,g)}
\left| \left(\La \bff\Ra\ci{I^0_+} , \La \bg\Ra\ci{I^0_-} \right)_{\R^2}\right|, \ 
\left| \left( \La \bff\Ra\ci{I^0_-} , \La \bg\Ra\ci{I^0_+}\right)_{\R^2} \right| \le 4 \bcQ^{1/2}. 
\end{align}
But the quantity $ \left\| \La W\Ra\ci{I}^{1/2} \La W^{-1}\Ra\ci{I}^{1/2}\right\|$ is exactly 
the norm of the averaging operator $f\mapsto \bE\ci{I}f = \1\ci{I} \La f\Ra\ci{I}$, so for any 
$\hat\bff$, $\hat\bg$
\begin{align*}
\left|\left(\La \hat\bff \Ra\ci{I^0} , \La \hat\bg \Ra\ci{I^0}\right)_{\R^2} \right| \le 
\bcQ^{1/2} \|\hat\bff\|\ci{L^2(W)} \|\hat\bg\|\ci{L^2(W^{-1})}  . 
\end{align*}
Taking $\hat\bff:= \1\ci{I^0_+} \bff$,  $\hat\bg:= \1\ci{I^0_-} \bg$ we see that
\begin{align*}
\La \bff \Ra\ci{I^0_+} = 2 \La \hat\bff \Ra\ci{I^0}, \qquad 
\La \bg \Ra\ci{I^0_-} = 2 \La \hat\bff \Ra\ci{I^0} .
\end{align*}
Noticing that $\|\hat\bff\|\ci{L^2(W)}\le \|\bff\|\ci{L^2(W)}$, $\|\hat\bg\|\ci{L^2(W^{-1})}\le 
\|\bg\|\ci{L^2(W^{-1})}$ we immediately get the first inequality in \eqref{e:small (f,g)}. Taking 
$\hat\bff:= \1\ci{I^0_-} \bff$,  $\hat\bg:= \1\ci{I^0_+} \bg$ we get the second one. 

Therefore the last term in \eqref{e:HilbPer01b} is bounded by $ 4\bcQ^{1/2} \ll \bcQ^{3/2}$
so for large $\bcQ$ it is negligible. This means that for sufficiently large $\bcQ$  
\begin{align*}
\left(\HT^\R \qper^{\vec N} \bff, \qper^{\vec N} \bg\right)_{L^2(\R)} \gtrsim \bcQ^{3/2}. 
\end{align*}
But by Lemma \ref{l:QP norms limit}
\begin{align*}
\left\| \qper^{\vec{N}} \bff \right\|\ci{ L^2(\REM^{\vec N} W)}   \longrightarrow 
\|\bff\|\ci{L^2(W)}=1, \qquad
\left\| \qper^{\vec{N}} \bg \right\|\ci{ L^2((\REM^{\vec N} W)^{-1})}   \longrightarrow 
\|\bg\|\ci{L^2(W^{-1})} =1, 
\end{align*}
so we got the desired lower bound \eqref{e:biest Hilb 02}  (with $\wt \bff$, $\wt\bg$ being the 
functions $\qper^{\vec{N}} \bff$ and 
$\qper^{\vec{N}} \bg$ normalized  in $L^2(W)$ and $L^2(W^{-1})$ respectively).
\end{proof}

\section{Tracking the \tp{$\bA_2$}{A2} constant}

In the above sections we constructed (for a sufficiently large $\bcQ$) the weight $\wt W$ on 
$I^0:=[0,1)$, $[\wt W]\ci{\bA_2}\ut{dy}\le\bcQ$, and a function $\bff\in L^2(\wt W)$ (supported on 
$I^0$),  $\|f\|\ci{L^2(W)}\ne0$  such 
that 
\begin{align*}
\| \Hdy \bff\|\ci{L^2(\wt W)} \gtrsim \bcQ^{3/2} \|f\|\ci{L^2(\wt W)}. 
\end{align*}
Extending this weight periodically to the real line, and slightly abusing notation using the same 
symbol for it, it is trivial that the above inequality holds for the same function $\bff$ 
(supported on $I^0$, not extended periodically). 

It is trivial that after extending periodically, the weight $\wt W$ still satisfies the dyadic 
$\bA_2$ 
condition and that $[\wt W]\ci{\bA_2}\ut{dy}\le\bcQ$. Here by the dyadic $\bA_2$ condition we mean 
the condition \eqref{e:A2 01} with $\cD$ being the collection of all dyadic subintervals of $\R$, 
not just $\cD(I^0)$. 

In the rest of this section we will follow this agreement and use $\cD$ for the collection of all 
dyadic subintervals of $\R$. 

We claim that  our weight $\wt W$ (periodically extended)  satisfies, in fact, the standard $\bA_2$ 
condition, and that $[\wt W]\ci{\bA_2} \le 16^2 \bcQ$. 

\subsection{The strong dyadic \tp{$\bA_2$}{A2} condition}
We will need the following definition. 

\begin{df}
We say that a weight $W$ satisfies the \emph{strong dyadic} $\bA_2$ condition, denoted 
$W\in\bA_2\ut{sd}$, if the estimate \eqref{e:A2} is satisfied for all intervals $I$ that are a 
union of 2 adjacent dyadic intervals of equal length. The corresponding supremum, the strong dyadic 
$\bA_2$ 
characteristic, will be denoted as $[W]\ci{\bA_2}\ut{sd}$.
\end{df}
Note that any dyadic interval is a union of its children, so
trivially $[W]\ci{\bA_2}\ut{dy}\le [W]\ci{\bA_2}\ut{sd}\le [W]\ci{\bA_2}$. The lemma below shows 
that the classes $\bA_2$ and $\bA_2\ut{sd}$ coincide, and that the corresponding characteristics 
are equivalent.

\begin{lm}\label{l:strong dyadic}
For a matrix weight $W$, 
\begin{align*}
[W]\ci{\bA_2} \le 16 [W]\ci{\bA_2}\ut{sd} .
\end{align*}
\end{lm}

\begin{proof}
Take an arbitrary interval $J$ (not necessarily dyadic), and let $I_1$ be a dyadic interval such 
that
\begin{align*}
|J|\le |I_1|\le 2 |J|, \qquad I_1\cap J\ne \varnothing .   
\end{align*}
 We can find a dyadic interval $I_2$ adjacent to $I_1$, with 
$|I_2|=|I_1|$ such that $J\subset I_1\cup I_2=: I$. Trivially $|I| \le 4 |J|$, so 
\begin{align}\label{e:A2 02}
\La W \Ra\ci{\!J} &\le |J|^{-1}\int_I W(x) \dd x \le  4 \La W \Ra\ci{I} , \qquad \\
\label{e:A2 03}
\La W^{-1} \Ra\ci{\!J} &\le 4 \La W^{-1} \Ra\ci{I} . 
\intertext{The  inequality \eqref{e:A2 03} is equivalent to}
\label{e:A2 04}
\La W^{-1} \Ra\ci{I}^{-1} &\le 4 \La W^{-1} \Ra\ci{\!J}^{-1}, 
\end{align}
so using Lemma \ref{l:A2_LMI} 
we can write
\begin{align*}
\La W \Ra\ci{\!J} \le 4 \La W \Ra\ci{I} \le 4 [W]\ci{\bA_2}\ut{sd} \La W^{-1} \Ra\ci{I}^{-1} 
\le 16 [W]\ci{\bA_2}\ut{sd} \La W^{-1} \Ra\ci{J}^{-1} \,;
\end{align*}
the second inequality here follows from Lemma \ref{l:A2_LMI}, the last one from  
\eqref{e:A2 04}.

Applying Lemma \ref{l:A2_LMI} again we get the conclusion of the lemma. 
\end{proof}

Below we will prove that $[\wt W]\ci{\bA_2}\ut{sd}\le 16 \bcQ$,  which by Lemma \ref{l:strong 
dyadic} gives us 
that $[\wt W]\ci{\bA_2}\le 16^2 \bcQ$. 

\subsection{Procedure \ref{proc:QP step} and strong dyadic \tp{$\bA_2$}{A2}} Let us recall, see 
Remark \ref{r:remod}, that the remodeled weight $\wt W$ was obtained from the dyadic weight $W$ by 
starting with the dyadic weight $\bE_1 W :=\La W \Ra\ci{I^0_+}\1\ci{I^0_+} + \La W 
\Ra\ci{I^0_-}\1\ci{I^0_-}$, applying inductively Procedure \ref{proc:QP step} and taking the limit.

Note, that the inverse $\wt W^{-1}$ of the remodeled weight is obtained exactly the same way from 
the weight $W^{-1}$: start with the weight $\bE_1 (W^{-1})$ and apply Procedure \ref{proc:QP step} 
exactly the same way. 

That was discussed for the remodeling of a weight on the interval $I^0$, but the same can be said 
about the periodic extension of the remodeled weight. Namely, this periodic extension can be 
obtained from taking the weight $\bE_1 W$, extending it periodically, then  applying steps 
described in Procedure \ref{proc:QP step} to starting intervals and taking the limit.  

To see that it can be done, take some numbering  $I_k$, $k\in\N$ of the dyadic intervals of length 
$1$, and start with the periodic extension of the weight $\bE_1 W$. Let us assume, without loss of 
generality, that $I_1=I^0$, and let $I_{1,k}$ be the (ordered) starting subintervals of $I_1$ to 
which we consecutively apply  Procedure \ref{proc:QP step} to get the remodeled weight $\wt W$ on 
the interval $I_1$. 

Let $I_{j,k} $ be the translation of the starting interval $I_{1,k}$ to the interval $I_j$. 
Performing Procedure \ref{proc:QP step} consecutively on intervals $I_{1,1}$, $I_{1,2}$, $I_{2,1}$, 
$I_{1,3}$,   $I_{2,2}$, $I_{3,1}$, $I_{1,4}$ etc, gives us the periodic extension of the remodeled 
weight.

Note that the inverse $\wt W^{-1}$ of the (periodically extended) remodeled weight is obtained 
\emph{exactly} the same way  from the weight $V=W^{-1}$.

Since we will be dealing with the averages of $W$ and $V=W^{-1}$,  we will need an extra 
definition. 

\begin{df}
\label{d:Strong A2 pair}
We say that a pair of weights $V$, $W$ satisfies the strong dyadic $\bA_2$ condition (and write 
$(V,W)\in \bA_2\ut{sd}$) if 
\begin{align*}
\sup_{I} \| \La V\Ra\ci{I}^{1/2}  \La W \Ra\ci{I}^{1/2} \|^2 := [V,W]\ci{\bA_2}\ut{sd} <\infty, 
\end{align*}
where the supremum is taken over all intervals $I$ that are unions of two adjecent dyadic intervals 
of equal length. 

The quantity $[V,W]\ci{\bA_2}\ut{sd}$ is called the \emph{strong dyadic $\bA_2$ characteristic} of 
the pair $V$, $W$. 
\end{df}

The next lemma is crucial for controlling the $\bA_2$ characteristic. 

\begin{lm}\label{l:track A2sd}
Let $V$, $W$ be matrix weights on $\R$ such that $[V,W]\ci{\bA_2}\ut{sd}\le 16\bcQ$. Assuming that 
$V$ and $W$ are constant on an interval $J\in\cD$, define new weights $V_1$, $W_1$ as 
\begin{align*}
V_1 := V + \qper\ci{J}^N     \Delta\ci{J}^2\Psi , \qquad 
W_1 := W + \qper\ci{J}^{N}  \Delta\ci{J}^2\Phi, \qquad N\ge 2, 
\end{align*}
where $\Psi=\Psi^*$, $\Phi=\Phi^*$ are some matrix-valued functions%
\footnote{Here we are implicitly assuming $V_1$, $W_1$ are weights, i.e.\ that $\La V \Ra\ci{J} + 
\Delta\ci{J}^2 \Psi\ge 0$,  $\La W \Ra\ci{J} + 
\Delta\ci{J}^2 \Phi\ge 0$ on $J$.}%
. 

If  

\begin{align}\label{e: A2 on J}
\| \La V\Ra\ci{\!J}^{1/2}  \La W \Ra\ci{\!J}^{1/2} \|^2 \le \bcQ, 
\end{align}
 then 
\begin{align*}
[V_1,W_1]\ci{\bA_2}\ut{sd} \le 16 \bcQ.
\end{align*}
\end{lm}

This lemma implies that $[\wt W]\ci{\bA_2}\ut{sd}\le 16 \bcQ$, which is exactly what we need.

\begin{proof}[Proof of Lemma \ref{l:track A2sd}]
Let $I_1$, $I_2$ be two adjacent dyadic intervals of equal length, and let $I:=I_1\cup I_2$. We 
want to show that
\begin{align}\label{e:A_2 strong pair 01}
\| \La V_1\Ra\ci{I}^{1/2}  \La W_1\Ra\ci{I}^{1/2} \|^2 \le 16\bcQ.
\end{align}
First, notice that if $J\cap I=\varnothing$ or $J\subset I$, then trivially 
\begin{align}\label{e:equal averages}
\La V_1 \Ra\ci{I} = \La V \Ra\ci{I}, \qquad \La W_1 \Ra\ci{I} = \La W \Ra\ci{I}, 
\end{align}
so \eqref{e:A_2 strong pair 01} holds (because $[V,W]\ci{\bA_2}\ut{sd}\le 16\bcQ$). 

It remains to consider the case $I\cap J\ne \varnothing$, but $J\not\subset I$. There are two 
possibilities here:
\begin{enumerate}
\item $J\not\subset I$ and $I\not\subset J$;
\item $I\subsetneq J$. 
\end{enumerate}

In the first case one of the intervals (say $I_1$) does not intersect $J$, but touches its 
boundary, and the other one ($I_2$) is strictly inside of $J$ and touches the boundary of $J$ at 
the same point. 

In this case  the identity \eqref{e:equal averages} holds for $I_2$, see Lemma \ref{l:av QP} below. 
For $I_1$ it holds trivially (the weights do not change), so \eqref{e:equal averages} holds for 
$I=I_1\cup I_2$.

Therefore,  \eqref{e:A_2 strong pair 01} holds for $I$ because 
$[V,W]\ci{\bA_2}\ut{sd}\le 16\bcQ$. 

Finally, let us treat the case \cond2, $I\subsetneq J$. 
 Writing $J$ as a disjoint union
\begin{align*}
J= \bigcup_{k=1}^{2^{N+2}} J_k', \qquad J_k'\in \ch^{N+1}(J), 
\end{align*}
we can see that the weights $V_1$, $W_1$ are constant on the intervals $J_k'$. 

Defining the auxiliary weights 
\begin{align*}
\wt V:=\La V \Ra\ci{J} + \Delta\ci{J}^2 \Psi, \qquad
\wt W:=\La W \Ra\ci{J} + \Delta\ci{J}^2 \Phi,
\end{align*}
we can see that for any $J_k'\in \ch^{N+2}(J) $ there exists $I'\in\ch^2(J)\cup \{J\}$ such that
\begin{align*}
\La V_1 \Ra\ci{J_k'} = \La \wt V \Ra\ci{I'}\,, \qquad \La W_1 \Ra\ci{J_k'} = \La \wt W \Ra\ci{I'}; 
\end{align*}
(the case $I'=J$ happens when $J_k'$ is a subset of an exceptional stopping interval). 

Since for any $\vf\in L^1(J)$
\begin{align*}
\La \vf \Ra\ci{J} = \frac14\sum_{I'\in\ch^2(J)} \La \vf \Ra\ci{I'}, 
\end{align*}
we can see that for any $J_k'\in \ch^{N+2}(J)$ either 
\begin{align}\label{e:aver 03}
\La V_1 \Ra\ci{J_k'} &= \La \wt V \Ra\ci{I'}, & \La W_1 \Ra\ci{J_k'} &= \La \wt W \Ra\ci{I'} 
\quad \text{for some } I'\in\ch^2(J)
\intertext{or}\label{e:aver 04}
\La V_1 \Ra\ci{J_k'} &= \La \wt V \Ra\ci{J}=\frac14 \sum_{I'\in\ch^2(J)} \La \wt V \Ra\ci{I'}, &
\La W_1 \Ra\ci{J_k'} &= \La \wt W \Ra\ci{J}=\frac14 \sum_{I'\in\ch^2(J)} \La \wt W \Ra\ci{I'}\,.
\end{align}

Let again $n\in\N$ be such that $I\in\ch^n(J)$. If $n\ge N+2$ then since $V_1$, $W_1$ are constant 
on intervals $J_k'\in \ch^{N+2}(J)$, we have 
\begin{align*}
\La V_1 \Ra\ci{I} = \La V_1 \Ra\ci{J_k'}, \qquad \text{where\ } I\subset J_k'\in\ch^{N+2}(J). 
\end{align*}
If $n<N+2$, then $I$ is a disjoint union  of some $J_k'\in\ch^{N+2}(J)$, so $\La V_1 \Ra\ci{I}$, 
$\La W_1 \Ra\ci{I}$  are the averages of  $\La V_1\Ra\ci{J_k'}$, $\La W_1\Ra\ci{J_k'}$ over these 
$J_k'$. Combining this with \eqref{e:aver 03}, \eqref{e:aver 04} we get that in any case 
\begin{align}\label{e:aver 05}
\La V_1\Ra\ci{I} &= \sum_{I'\in \ch^2(J)} \alpha\ci{I'} \La \wt V \Ra\ci{I'}, \qquad
\La W_1\Ra\ci{I} = \sum_{I'\in \ch^2(J)} \alpha\ci{I'} \La \wt W \Ra\ci{I'},
\intertext{where}\notag
\alpha\ci{I'} &\ge 0, \qquad \sum_{I'\in \ch^2(J)} \alpha\ci{I'} =1. 
\end{align}

 But for any $I'\in \ch^2(J)$
\begin{align*}
\La \wt V \Ra\ci{I'}\le  4 \La \wt V \Ra\ci{J} =4 \La  V \Ra\ci{J}\,, \qquad 
\La \wt W \Ra\ci{I'}\le  4 \La \wt W \Ra\ci{J} =4 \La  W \Ra\ci{J}  
\end{align*}
so  
\begin{align*}
\La V_1 \Ra\ci{I}\le  4 \La  V \Ra\ci{J}\,, \qquad 
\La W_1 \Ra\ci{I} &\le  4 \La  W \Ra\ci{J}. 
\intertext{The second inequality is equivalent to}
\La  W \Ra\ci{J}^{-1} &\le 4 \La W_1 \Ra\ci{I}^{-1}, 
\end{align*}
so we can write
\begin{align*}
\La V_1 \Ra\ci{I}\le  4 \La  V \Ra\ci{J}\le 4\bcQ \La  W \Ra\ci{J}^{-1} 
\le 16\bcQ \La W_1 \Ra\ci{I}^{-1};
\end{align*}
here in the second inequality we used Lemma \ref{l:A2_LMI_prelim}. 

Applying Lemma \ref{l:A2_LMI_prelim} to the resulting inequality $\La V_1 \Ra\ci{I} \le 16\bcQ \La 
W_1 \Ra\ci{I}^{-1}$ we get the desired estimate \eqref{e:A_2 strong pair 01}.   
\end{proof}

The following lemma applies to arbitrary $L^1$ functions, not only to weights. 
\begin{lm}\label{l:av QP}
Let $I$ be a dyadic subinterval of an interval $J$, touching its boundary, and let $\vf\in 
L^1(J)$. 
Then for any $N\ge 2$
\begin{align*}
\La \qper\ci{J}^N \vf \Ra\ci{I} = \La \vf \Ra\ci{J}. 
\end{align*}
\end{lm}
\begin{proof}
The case $I=J$ is trivial, so we assume that $I\subsetneq J$, so $I$ touches the boundary of $J$ at 
one point. Let $J_0\in\ch^N(J)$ touch the boundary of $J$ at the same point as $I$. 

Let $n$ be such that $I\in\ch^n(J)$. If $n\ge N$, then $I\subset J_0$, and since $\qper\ci{J}^N 
\vf$ 
is constant on $J_0$ we conclude that 
\begin{align*}
\La \qper\ci{J}^N \vf \Ra\ci{I} = \La \qper\ci{J}^N \vf \Ra\ci{J_0} = \La  \vf \Ra\ci{J};
\end{align*}
the last equality holds by the construction of $\qper\ci{J}^N \vf$. 

If $n<N$ then $I$ can be represented as a disjoint union
\begin{align*}
I = \bigcup_{k=0}^m J_k, \qquad J_k\in\ch^N(J). 
\end{align*}
By the construction of $\qper\ci{J}^N \vf$ 
\[
\La \qper\ci{J}^N \vf \Ra\ci{J_k} = \La  \vf \Ra\ci{J}, 
\]
and we immediately get the conclusion of the lemma. 
\end{proof}

We used this Lemma in the proof of Lemma \ref{l:track A2sd}, which is now completely proved. Tracking  the strong dyadic $\mathbf{A}_2$ characteristic of the remodeled weight $\tilde W$ is accomplished. This  finishes the construction of our counterexample to the matrix $\mathbf{A}_2$ conjecture.


\begin{thebibliography}{99999}


\bibitem{AIS} {\sc K.~Astala, T.~Iwaniec, E.~Saksman},
{\em Beltrami operators in the plane.}
Duke Math. J. 107 (2001), no. 1, pp. 27-–56.

\bibitem{BCTW2019}
{\sc K.~Bickel, A.~Culiuc, S.~Treil,   B.~D. Wick}, {\em Two weight estimates
  with matrix measures for well localized operators}, Trans. Amer. Math. Soc.,
  371 (2019), pp.~6213--6240.

\bibitem{BPW} {\sc K.~Bickel, S.~Petermichl, B.~Wick}, 
{\em Bounds for the Hilbert transform with matrix $A_2$ weights.} 
J. Funct. Anal. 270 (2016), no. 5, pp. 1719--1743. 

\bibitem{Bol} {\sc A.~A.~Bolibrukh}, 
{\em The Riemann--Hilbert problem.} 
Russian Math. Surveys, 45:2 (1990), pp. 1--47.

\bibitem{Bo} {\sc J.~Bourgain},
{\em Some remarks on {B}anach spaces in which martingale difference sequences are unconditional.}
Ark. Mat., (1983), 21(2), pp. 163--168.

\bibitem{Buckley1993} {\sc S.~M.~Buckley}, 
{\em Estimates for operator norms on weighted spaces and reverse {J}ensen inequalities.} Trans. 
Amer. Math. Soc., 340 (1993), pp.~253--272.

\bibitem{CG} {\sc M.~Christ, M.~Goldberg}, 
{\em Matrix $A_2$ weights and a Hardy-Littlewood  maximal function.} 
Trans. Amer. Math. Soc., (2001), v. 353, no. 5, pp. 1995--2002.

\bibitem{CR} {\sc J.~M.~Conde-Alonso, G.~Rey}, 
{\em A pointwise estimate for positive dyadic shifts and some applications.}
Math. Ann. 365 (2016), pp. 1111--1135.

\bibitem{CT} {\sc A.~Culiuc, S.~Treil}, 
{\em The {C}arleson embedding theorem with matrix weights.}
Int. Math. Res. Not., (2019), no. 11, pp. 3301--3312.

\bibitem{DPlower} {\sc K.~Domelevo, S.~Petermichl}, 
{\em The dyadic and the continuous Hilbert transforms with values in Banach spaces.} 
arXiv:2212.00090 (2022), pp. 1--8.

\bibitem{DPupper} {\sc K.~Domelevo, S.~Petermichl}, 
{\em The dyadic and the continuous Hilbert transforms with values in Banach spaces. Part 2.} 
arXiv:2303.12480 (2023), pp. 1--19. 

\bibitem{DoPeSk} {\sc K.~Domelevo, S.~Petermichl,  K.~A.~\v{S}kreb},
{\em Failure of the matrix weighted bilinear {C}arleson embedding theorem.} 
Linear Algebra Appl., (2019), no. 582, pp. 452--466.

\bibitem{HS} {\sc H.~Helson, G.~Szeg\"o}, 
{\em A problem in prediction theory.} 
Ann. Mat. Pura Appl. (4) 51 (1960), pp. 107--138. 

\bibitem{HSa} {\sc H.~Helson, D.~Sarason}, {\em Past and future.} Math. Scand., 21 (1967), pp. 
5--16.

\bibitem{HMW} {\sc R.~Hunt, B.~Muckenhoupt, R.~Wheeden}, 
{\em Weighted norm inequalities for the conjugate function and Hilbert transform.} 
Trans. Amer. Math. Soc. 176 (1973), pp. 227--251.

\bibitem{Hy} {\sc T.~Hyt\"onen}, 
{\em The sharp weighted bound for general Calder\'on-Zygmund operators.} 
Annals of Math. (2) 175 (2012), no. 3, pp. 1473--1506. 

\bibitem{HyPeVo} {\sc T.~Hyt\"onen, S.~Petermichl, A.~Volberg}, 
{\em The Sharp Square Function Estimate with Matrix Weights.} 
Discr. Anal., (2019), no. 2, pp. 1--8.

\bibitem{HyPTV} {\sc T.~Hyt\"onen, C.~P\'erez, S.~Treil, A.~Volberg}, 
{\em Sharp weighted estimates for dyadic shifts and the $A_2$ conjecture.} 
J. Reine Angew. Math. 687 (2014), pp. 43--86.

\bibitem{Ib} {\sc I.~A.~Ibragimov}, 
{\em Completely regular multidimensional stationary processes with discrete time.} 
Proc. Steklov Inst. Math., 111 (1970), pp. 269--301.

\bibitem{IsKwPo} {\sc J.~Isralowitz, H.~Kwon, S.~Pott}, 
{\em Matrix weighted norm inequalities for commutators and paraproducts with matrix symbols.}
J. Lond. Math. Soc. 96(1): (2017), pp. 243–270.

\bibitem{IwM} {\sc T.~Iwaniec, G.~Martin}, {\em Quasiregular mappings in even dimensions}, Acta 
Math., 170 (1993), 29--81.

\bibitem{JL-HGBS2001}
{\sc W.~B. Johnson and J.~Lindenstrauss}, {\em Basic concepts in the geometry
  of {B}anach spaces}, in Handbook of the geometry of {B}anach spaces, {V}ol.
  {I}, North-Holland, Amsterdam, 2001, pp.~1--84.

\bibitem{KakTre21}
{\sc S.~Kakaroumpas, S.~Treil}, 
{\em``{S}mall step'' remodeling and counterexamples for weighted estimates with arbitrarily 
``smooth'' weights.} 
Adv. Math., 376 (2021), pp. 1--52. 
  
\bibitem{La} {\sc M. Lacey}, 
{\em An elementary proof of the $A_2$ bound.} 
Israel J. Math., 217 (2017), pp. 181–-195.

\bibitem{LPR} {\sc M.~T. Lacey, S.~Petermichl, M.~C. Reguera}, 
{\em Sharp $ A_2$ Inequality for Haar Shift Operators.} 
Math. Ann. \textbf{348} (2010), no.~1, pp. 127--141.

\bibitem{LaTr2007} {\sc M.~Lauzon, S.~Treil}, 
{\em Scalar and vector {M}uckenhoupt weights.} 
Indiana Univ. Math. J., 56 (2007), pp.~1989--2015.

\bibitem{Le} {\sc A. Lerner}, 
{\em A simple proof of the $A_2$ conjecture.} 
Intern. Math. Res. Notices, (2013), no. 14, pp. 3159--3170. 

\bibitem{LN} {\sc A.~Lerner, F.~Nazarov}, 
{\em Intuitive dyadic calculus.} 
arXiv:1508.05639 [math.CA], 2015. 

\bibitem{MW} {\sc P. Masani, N. Wiener}, 
{\em On bivariate stationary processes and the factorization of matrix-valued functions.} Theor. 
Probability Appl., 4 (1959), pp. 300--308.

\bibitem{Nazarov} {\sc F.~Nazarov}, 
{\em A counterexample to Sarason's conjecture.} 
unpublished manuscript, available at \\ \url{https://users.math.msu.edu/users/fedja/prepr.html}

\bibitem{NaPeSkTr} {\sc F.~Nazarov, S.~Petermichl, K. A.~\v{S}kreb, S. Treil}, 
{\em The matrix-weighted dyadic convex body maximal operator is not bounded.} 
Adv. Math. \textbf{410} (2022), pp. 1--20.

\bibitem{NaPeTrVo} {\sc F. Nazarov, S.~Petermichl, S.~Treil, A.~Volberg}, 
{\em Convex Body Domination and weighted Estimates with Matrix Weights.} 
Adv. Math., 318 (2017), pp. 279--306.

\bibitem{NaTr-Hunt}  {\sc F. Nazarov, S.~Treil}, 
{\em The hunt for a Bellman function: applications to estimates for singular integral operators and 
to other classical problems of harmonic analysis.} 
Algebra i Analiz, \textbf{8}(1996), no.~5, 32--162; translation in
St.~Petersburg Math.~J., \textbf{8}(1997), no.~5, 721--824.

\bibitem{Pell} {\sc V. V. Peller}, 
{\em Hankel operators and multivariate stationary processes.} 
Operator theory:
Operator algebras and applications, Part 1, (Durham, NH, 1988), 357-371, Proc.
Sympos. Pure Math., 51(1), Amer. Math. Soc., Providence, RI, 1990.

\bibitem{PeKh} {\sc V. V. Peller, S. V. Khruschev}, 
{\em Hankel operators, best approximation, and stationary Gaussian processes.} 
Russian Math. Surveys, 37 (1982), pp. 53--124.

\bibitem{Pe} {\sc S.~Petermichl}, 
{\em Dyadic Shifts and a Logarithmic Estimate for Hankel Operators with Matrix Symbol.}
Comptes Rendus Acad. Sci. Paris, 1 (2000), no. 1, pp. 455--460.

\bibitem{P} {\sc S.~Petermichl}, 
{\em The sharp bound for the Hilbert transform on weighted Lebesgue spaces in terms of the 
classical $A_p$ characteristic.} 
Amer. J. Math. 129 (2007), no. 5, pp. 1355--1375.

\bibitem{PP} {\sc S.~Petermichl, S. Pott}, 
{\em An Estimate for Weighted Hilbert Transform via Square Functions.}
Trans. Amer. Math. Soc. 354 (2002), pp. 1699-–1703.
 
\bibitem{PV} {\sc S. Petermichl, A. Volberg}, 
{\em Heating of the Ahlfors-Beurling operator: weakly quasiregular maps on the plane are 
quasiregular.} 
Duke Math. J. 112 (2002), no. 2, pp. 281--305.

\bibitem{PottStoica2014}
{\sc S.~Pott, A.~Stoica}, {\em Linear bounds for {C}alder\'{o}n-{Z}ygmund
  operators with even kernel on {UMD} spaces}, J. Funct. Anal., 266 (2014),
  pp.~3303--3319. 

\bibitem{Sa} {\sc D. Sarason}, 
{\em An addendum to “Past and future”.} 
Math. Scand., 30 (1972), pp. 62--64.

\bibitem{Treil-OTAA-89}{\sc S.~Treil}, {\em Geometric methods in spectral theory of
  vector-valued functions: some recent results}, in Toeplitz operators and
  spectral function theory, vol.~42 of Oper. Theory Adv. Appl., Birkh\"auser,
  Basel, 1989, pp.~209--280.


\bibitem{Tr_ideals_2007} {\sc S.~Treil}, 
{\em The problem of ideals of {$H^\infty$}: beyond the exponent {$3/2$}.} 
J. Funct. Anal., 253 (2007), pp.~220--240.

\bibitem{Tr} {\sc S.~Treil}, 
{\em Mixed $A_2 - A_\infty$ estimates of the nonhomogeneous vector square function with matrix 
weights.} 
Proc. Amer. Math. Soc., 151 (2023), pp.~3381--3389.  
\bibitem{TV} {\sc S.~Treil, A.~Volberg}, 
{\em Wavelets and the angle between past and future.} 
J. Funct. Anal. 143 (1997), no. 2, pp. 269--308.

\bibitem{WM1} {\sc N. Wiener, P. Masani}, 
{\em The prediction theory of multivariate stochastic processes. I. The regularity conditions.} 
Acta Math., 98 (1957), pp. 111--150.

\bibitem{WM2} {\sc N. Wiener, P. Masani},  
{\em The prediction theory of multivariate stochastic processes, II. The linear
predictor.} 
Acta Math., 99 (1958), pp. 93--137. 

\bibitem{Wittwer2000} {\sc J.~Wittwer}, 
{\em A sharp estimate on the norm of the martingale transform.} Math. Res. Lett., 7 (2000), 
pp.~1--12.



\end{thebibliography}
\end{document}